\documentclass[a4paper,reqno,11pt]{amsart}

\usepackage[foot]{amsaddr}
\usepackage[margin=2.42cm]{geometry}
\usepackage{mathtools}
\usepackage{amsthm}
\usepackage{dsfont}
\usepackage{amssymb}
\usepackage{bm}
\usepackage{enumitem}
\usepackage{tipa}
\usepackage{upgreek}
\usepackage{graphicx}
\usepackage{color}
\usepackage{xcolor}
\usepackage[colorlinks,linkcolor={black},citecolor={blue},urlcolor={black}]{hyperref}
\usepackage{epsfig,color}
\usepackage{tikz}
\usetikzlibrary{decorations.pathreplacing,positioning}
\usepackage{caption}
\usepackage{subcaption}
\usepackage{multirow}
\usepackage{hyperref}
\usepackage[capitalise]{cleveref}
\usepackage{glossaries}
\expandafter\def\csname ver@etex.sty\endcsname{3000/12/31}

\usepackage{autonum}
\usepackage{xpatch} % for boldface proof
\usepackage[square,numbers]{natbib}
\numberwithin{equation}{section}
\newtheorem{theorem}{Theorem}[section]
\newtheorem{lemma}[theorem]{Lemma}
\newtheorem{corollary}[theorem]{Corollary}
\newtheorem{proposition}[theorem]{Proposition}

\xpatchcmd{\proof} % boldface proof
  {\itshape}
  {\bfseries}
  {}
  {}

\theoremstyle{definition}
\newtheorem{definition}[theorem]{Definition}

\theoremstyle{definition}
\newtheorem{remark}[theorem]{Remark}
\newtheorem{problem}{Problem}
\newtheorem{example}[theorem]{Example}

\let\emptyset\varnothing

\renewcommand{\d}{{\mathrm{\,d}}}
\DeclareMathAlphabet{\mathbbmsl}{U}{bbm}{m}{sl}

\makeatletter
\renewcommand\subsubsection{\@startsection{subsubsection}{3}{\z@}%
                                     {.5\linespacing\@plus.7\linespacing}{-.5em}
                                     {\normalfont\bfseries}}

\title[{Fractional Sobolev processes on Wasserstein spaces and energy minimization}]{
{Fractional Sobolev processes on Wasserstein spaces and their energy-minimizing particle representations \\with applications
}
}
\author{Ehsan Abedi$^\dagger$}

\address{$^\dagger$Faculty of Mathematics \\
         Bielefeld University \\
         Postfach 10 01 31 \\
         33501 Bielefeld \\
         Germany.}
\email{ehsan.abedi@math.uni-bielefeld.de}

\thanks{}

\begin{document}

\begin{abstract}    
    Given a probability-measure-valued process $(\mu_t)$, we aim to find, among all path-continuous stochastic processes whose one-dimensional time marginals coincide almost surely with $(\mu_t)$ (if there is any), a process that minimizes a given energy in expectation. Building on our recent study (arXiv:2502.12068), where the minimization of fractional Sobolev energy was investigated for deterministic paths on Wasserstein spaces, we now extend the results to the stochastic setting to address some applications that originally motivated our study. Two applications are given. We construct minimizing particle representations for processes on Wasserstein spaces on $\mathbb{R}$ with H\"{o}lder regularity, using optimal transportation. We prove the existence of minimizing particle representations for solutions to stochastic Fokker--Planck--Kolmogorov equations on $\mathbb{R}^\mathrm{d}$ satisfying an integrability condition, using the stochastic superposition principle of Lacker--Shkolnikov--Zhang (J. Eur. Math. Soc. 25, 3229--3288 (2023)).
\end{abstract}

%\begingroup %this disables uppercasing
%\def\uppercasenonmath#1{} % this disables uppercasing title
%\let\MakeUppercase\relax % this disables uppercasing authors
\maketitle
%\endgroup %this disables uppercasing

%\tableofcontents

\section{Introduction and main results}\label{sec:introduction_s}
\subsection{A preliminary result}
    This paper continues our recent work in the deterministic setting \cite{Abedi2025paths} by studying the stochastic setting.
    Let $(\mathcal{X},d)$ be a complete separable metric space, and $(\Omega,\mathcal{F},\mathbb{P})$ be a probability space.
    We denote by $\mathcal{B}(\mathcal{X})$ the $\sigma$-algebra of Borel sets of $\mathcal{X}$.
    Let $P(\mathcal{X})$ and $P_{\Omega}(\mathcal{X})$ denote the set of deterministic and random Borel probability measures on $\mathcal{X}$, respectively. 
    A random probability measure is a map $\mu: \mathcal{B}(\mathcal{X}) \times \Omega \to [0,1]$ such that it is a probability measure when the second argument is fixed and is $\mathcal{F}$-measurable when the first argument is fixed.
    %The product space $\mathcal{X} \times \Omega $ is regarded as a measurable space with the product $\sigma$-algebra $\mathcal{B}(\mathcal{X})\otimes\mathcal{F}$.
    % (continuous-time)
    A (probability-)measure-valued process is a collection $(\mu_t)\coloneqq (\mu_t)_{t \in I} \subset P_{\Omega} (\mathcal{X})$ of random measures indexed by time $t \in I \coloneqq [0,T] \subset \mathbb{R}$.
    Given a measure-valued process $(\mu_t)$, we aim to find, among all path-continuous stochastic processes whose one-dimensional time marginals coincide $\mathbb{P}$-almost surely with $(\mu_t)$ (if there is any), a process that minimizes a given energy functional in expectation.
    The path law of such a process is a random path measure $\pi \in P_{\Omega}(\Gamma_T)$, referred to as a random lift.
    Throughout the paper, $\Gamma_T \coloneqq C([0, T ]; \mathcal{X}) $ denotes the space of continuous paths, equipped with the supremum distance, whose induced topology is known as the supremum % or uniform
    topology and whose induced Borel $\sigma$-algebra coincides with the $\sigma$-algebra generated by evaluation maps. %$\{ e_t: t \in I\}$
    We formulate the aforementioned question as a variational problem on the space of random path measures, which can be regarded as a generalization of the Monge--Kantorovich transport problem to a time-dependent and stochastic setting:

    \vspace{6pt}
    
    \begin{minipage}{0.97\textwidth}
        \begin{problem}\label{prob:variational_problem_random}
        Given $(\mu_t)_{t \in I} \subset P_\Omega (\mathcal{X})$ with $I \coloneqq [0,T] \subset \mathbb{R}$ defined on a probability space $(\Omega, \mathcal{F}, \mathbb{P})$ and a measurable functional $\Psi: \Gamma_T \to [0,+\infty]$, consider the variational problem
            \begin{equation}\label{eq:variational_problem_random}
                \inf_{\pi \in \mathrm{Lift}_\Omega(\mu_t)} \mathbb{E} \left[\int_{\Gamma_T} \Psi (\gamma) \d \pi (\gamma) \right]
            \end{equation}
            where the infimum is taken over the set of \emph{random lifts} of $(\mu_t)$ defined by
            \begin{equation}\label{eq:def:random_lift_intro}
            \mathrm{Lift}_\Omega (\mu_t) \coloneqq \Big\{\pi \in P_\Omega(\Gamma_T) : \quad (e_t)_{\#} \pi^\omega  = \mu_t^\omega \,\, \mathbb{P}\text{-a.s. } \text{for all } t \in I \Big\},
        \end{equation}
             where $e_t: \Gamma_T \to \mathcal{X}$ is the evaluation map defined by $e_t(\gamma) \coloneqq \gamma_t$ for $\gamma \in \Gamma_T$. 
        \end{problem}
    \end{minipage}
    
    \vspace{6pt}

    If $\mathrm{Lift}_\Omega (\mu_t) = \emptyset$, the infimum  \eqref{eq:variational_problem_random} is set to $+\infty$, by the usual convention.
    We refer to the expectation $\mathbb{E}[\int \Psi \d \pi ]$ as the \emph{energy of the random lift} $\pi$ with respect to the \emph{energy functional} $\Psi$.
    We note that the application $\omega \mapsto \int \Psi \d \pi^\omega$ is measurable (see e.g. \cite[Proposition 3.3]{Crauel2002}), and thus the expectation is well-defined.
    As a first step, we prove the existence of a minimizer for the problem above using the direct method in the calculus of variations and Prokhorov's theorem for random measures (see e.g. \cite[Theorem 4.29]{Crauel2002}). Below is Proposition \ref{prop:existence_of_minimizer_random}. 

    \begin{proposition}[Existence of a random minimizer]\label{prop:existence_of_minimizer_random_intro}
         Let $(\mathcal{X}, d)$ be a complete separable metric space, $(\Omega, \mathcal{F},\mathbb{P})$ be a probability space, and $I \coloneqq [0,T] \subset \mathbb{R}$.
         Let $\Psi: C(I;\mathcal{X}) \to [0,+\infty]$ be a map whose sublevels are compact in $C(I;\mathcal{X})$.
         Assume that the infimum  \eqref{eq:variational_problem_random} is finite. Then there exists a random minimizer $\pi \in P_\Omega(C(I;\mathcal{X}) ) $ to Problem \ref{prob:variational_problem_random}. 
    \end{proposition}    

    We will apply this result to concrete energy functionals. Before that, we need to fix some notations for studying processes on Wasserstein spaces.

    \subsection{The space \texorpdfstring{$(P_{p,\Omega}(\mathcal{X}),\mathbb{W}_p)$}{(PpOmega(X),Wp)} for random measures}\label{subsec:prob_measure_valued_process_PpOmega_intro}
    This section aims to formulate a metric space for random measures similar to the Wasserstein space for deterministic measures.

    Given $p \in [1, \infty) $, let $P_p(\mathcal{X}) \subset P(\mathcal{X})$ denote the subset of probability measures with finite $p$-th moment, equipped with the $p$-(Kantorovitch–Rubinstein–)Wasserstein distance $W_p$. Given $\mu,\nu \in P_p(\mathcal{X})$, we denote by $\mathrm{Cpl}(\mu,\nu)$ the set of their couplings and by $\mathrm{OptCpl}(\mu,\nu)$ the set of optimal couplings for $W_p$. The space $(P_p(\mathcal{X}),W_p)$ is usually called the $p$-Wasserstein space.
    
    Given $p \in [1,\infty)$, we define a subset $P_{p,\Omega}(\mathcal{X}) \subset P_{\Omega}(\mathcal{X})$ of random probability measures:
    \begin{equation}\label{eq:P_p_Omega(X)_def_intro}
                P_{p,\Omega} (\mathcal{X}) \coloneqq \left\{ \mu \in P_{\Omega} (\mathcal{X}) : \, \mathbb{E} \left[  \int_{\mathcal{X}} d(x, \bar{x})^p \d \mu (x) \right] < + \infty \right\}, 
    \end{equation}
    where $\bar{x} \in \mathcal{X}$ is an arbitrary (non-random) point. The condition above means that a random measure $\mu$ lies in $P_{p,\Omega}(\mathcal{X})$ if and only if its average measure $\mathbb{E} \mu \coloneqq \int_{\Omega} \mu \d \mathbb{P}$ lies in $P_p(\mathcal{X})$. 
    Between two random measures $\mu, \nu \in P_{p,\Omega}(\mathcal{X})$, we set
    \begin{equation}\label{eq:def:mathbbW_p_min_intro}
        \mathbb{W}_p(\mu,\nu) \coloneqq \left( \min_{\Upsilon \in \mathrm{Cpl}_{\Omega}(\mu,\nu)} \mathbb{E} \left[ \int_{\mathcal{X}\times\mathcal{X}} d(x,y)^p \d \Upsilon (x,y) \right] \right)^{1/p},
    \end{equation}
    where the minimum is taken over the set of random couplings of $(\mu,\nu)$ defined by 
    \begin{equation}
        \mathrm{Cpl}_\Omega(\mu,\nu) \coloneqq \Big\{ \Upsilon \in P_\Omega(\mathcal{X}^2) : \quad \Upsilon^\omega \in \mathrm{Cpl}(\mu^\omega,\nu^\omega) \,\, \mathbb{P}\text{-a.s.} \Big\}.
    \end{equation}
    Since it is possible to select optimal couplings between random measures in a measurable way \cite[Corollary 5.22]{Villani2009} (formulated here as Lemma  \ref{lemma:measurable_selection_opt_cpl_two}), we have the existence of a random coupling that is optimal a.s. Consequently, a minimizer to \eqref{eq:def:mathbbW_p_min_intro} exists.
    For the set of random optimal couplings, we write $\Upsilon \in \mathrm{OptCpl}_{\Omega}(\mu,\nu)$, meaning again that $\Upsilon^\omega \in \mathrm{OptCpl}(\mu^\omega,\nu^\omega)$ $\mathbb{P}$-a.s. 
    Using this optimal coupling, \eqref{eq:def:mathbbW_p_min_intro} can be equivalently  expressed as
    \begin{equation}\label{eq:mathbbW_p_def_intro}
        \mathbb{W}_p (\mu,\nu) = \Big( \mathbb{E} \big[ W^p_p(\mu,\nu) \big] \Big)^{1/p}. 
    \end{equation}
    It is easy to verify that $\mathbb{W}_p$ defines a metric on $P_{p,\Omega}(\mathcal{X})$ (Proposition \ref{prop:mathbbWp_metric_result}). 
    The relationship between the topology induced by $\mathbb{W}_p$ and the topology of narrow convergence on $P_{\Omega} (\mathcal{X})$ are discussed in  \cref{subsec:random_probability_measures_finite_moment}, using the references  \cite{Crauel2002,HauslerLuschgy2015}.

    \subsection{Measure-valued processes in \texorpdfstring{$(P_{p,\Omega}(\mathcal{X}),\mathbb{W}_p)$}{(PpOmega(X),Wp)}}\label{subsec:Measure_valued_processes_intro}
    In this paper, we study measure-valued processes $(\mu_t)_{t \in I}$ that lie in the subspace $P_{p,\Omega} (\mathcal{X})$ and have certain path regularity with respect to the metric $\mathbb{W}_p$. 
    We thus view such a measure-valued process as a ``single curve'' in the metric space $(P_{p,\Omega}(\mathcal{X}),\mathbb{W}_p)$, with its regularity always understood with respect to the metric $\mathbb{W}_p$. For clarity, all path regularities computed with respect to $\mathbb{W}_p$ will be denoted by $\Vert \cdot \Vert$, distinguishing them from regularities $|\cdot|$ computed with respect to ${W}_p$.

    It is worth noting that while $(\mu_t)_{t \in I} \subset P_{p,\Omega} (\mathcal{X})$ implies that $\mu_t^\omega \in P_p(\mathcal{X})$ a.s. for all $t \in I$, it does not imply that $\mu_t^\omega \in P_p(\mathcal{X})$ for all $t \in I$ a.s. 
    In other words, one cannot say that the sample paths $t \mapsto \mu_t^\omega$ lie in the Wasserstein space $P_p(\mathcal{X})$ a.s. Even if they do, the continuity of $(\mu_t)$ with respect to $\mathbb{W}_p$ is obviously not enough to imply the continuity of the sample paths $(\mu_t^\omega)$ with respect to ${W}_p$.
     
    \subsection{Main theorems}\label{subsec:introduction_s_main_theorems}
    We come back to Problem \ref{prob:variational_problem_random} and now focus on energy functionals $\Psi$ of the form $
    \Psi (\gamma) \coloneqq d(\gamma_0,\bar{x}) + | \gamma |$ where $\bar{x} \in \mathcal{X}$ is an arbitrary point and  $| \cdot | : \Gamma_T \to [0,+\infty]$ is a (semi-)norm on the space of continuous paths.
    As discussed in \cite[Section 1]{Abedi2025paths}, a candidate for $| \cdot |$ in the low-regularity setting that makes the energy functional $\Psi$ satisfy the compactness requirement in Proposition \ref{prop:existence_of_minimizer_random_intro} is \emph{fractional Sobolev} regularity $|\cdot|_{W^{\alpha,p}}$ or certain \emph{Besov} regularity $|\cdot|_{b^{\alpha,p}}$ (at least when $(\mathcal{X},d)$ has the additional structure that closed bounded sets are compact).
    These norms (given in Definitions \ref{def:Walphap_regularity} and \ref{def:balphap_regularity}) are equivalent on the space of continuous paths under the condition $1<p<\infty$ and $1/p<\alpha < 1$, thanks to the characterization by \cite{LiuPromelTeichmann2020}. The key computations in this paper are performed using $|\cdot|_{b^{\alpha,p}}$, since as a discrete object, it has many technical advantages, including measurability.
    
    As a first step, we start from a random path measure $\pi$ with finite $W^{\alpha,p}$-energy and study the regularity of its one-dimensional time marginals, which is a measure-valued process. This is also analyzed when $\pi$ has finite H\"{o}lder or $p$-variation energy in \cref{subsec:from_pi_to_mu_random}.  
    Below is Theorem \ref{thm:lift_to_mu_Walphap_stoch}.
    \begin{theorem}\label{thm:lift_to_mu_Walphap_stoch_intro}
    Let $(\mathcal{X},d)$ be a complete separable metric space, and $(\Omega, \mathcal{F}, \mathbb{P})$ be a probability space. Let $\pi \in P_{\Omega} (C([0,T];\mathcal{X})) $ satisfy
    \begin{equation}\label{eq:lift_to_mu_Walphap_integrability_stoch_intro}
            \mathbb{E} \left[  \int_{\Gamma_T} \Big( d(\gamma_0,\bar{x})^p + | \gamma |_{W^{\alpha,p}}^p  \Big) \d \pi(\gamma) \right] < + \infty
        \end{equation}
        for some $1<p<\infty$ and $ \frac{1}{p} < \alpha < 1$ and $\bar{x}\in \mathcal{X}$.
        Then, for $\mathbb{P}$-a.e. $\omega \in \Omega$, the curve $t \mapsto  \mu_t^\omega \coloneqq {(e_t)}_{\#} \pi^\omega$ is in $W^{\alpha, p} ([0,T];P_p(\mathcal{X}))$ with
        \begin{equation}\label{eq:lift_to_mu_Walphap_random_ineq_pointwise_intro} 
            |\mu^\omega|_{W^{\alpha,p}}^p \leq \int_{\Gamma_T} |\gamma|_{W^{\alpha,p}}^p \d \pi^\omega(\gamma).
        \end{equation}
        In addition, $(\mu_t) \in W^{\alpha, p} ([0,T];P_{p,\Omega}(\mathcal{X}))$ with 
        \begin{equation}\label{eq:lift_to_mu_Walphap_random_ineq_exp_intro} 
            \Vert\mu\Vert_{W^{\alpha,p}}^p \leq \mathbb{E} \left[\int_{\Gamma_T} |\gamma|_{W^{\alpha,p}}^p \d \pi (\gamma) \right].
        \end{equation}
        The same statement holds for $|\cdot|_{b^{\alpha,p}}$.
    \end{theorem}
    
    By combining Proposition \ref{prop:existence_of_minimizer_random_intro} and Theorem \ref{thm:lift_to_mu_Walphap_stoch_intro}, we immediately arrive at our first main result: %in the stochastic setting: 

    \begin{theorem}[Existence of a minimizing random lift]\label{thm:existence_of_Walphap_minimizer_random_intro}
        Let $(\mathcal{X},d)$ be a complete separable metric space in which closed bounded sets are compact, and $I \coloneqq [0,T] \subset \mathbb{R}$.
        Let $(\mu_t)_{t \in I}$ be a probability measure-valued stochastic process defined on a probability space $(\Omega, \mathcal{F}, \mathbb{P})$ such that $\mu_0 \in P_{p,\Omega}(\mathcal{X})$ and it has a random lift with finite $W^{\alpha,p}$-energy with $1<p<\infty$ and $\frac{1}{p}< \alpha <  1$, i.e., \eqref{eq:lift_to_mu_Walphap_integrability_stoch_intro} holds. 
        Then, there exists a random minimizer $\pi \in P_{\Omega}(C(I;\mathcal{X}) ) $ to Problem \ref{prob:variational_problem_random} for the energy $\Psi(\gamma)  = |\gamma|^p_{W^{\alpha,p}} $. In particular, 
        \begin{enumerate}[label=(\roman*), font=\normalfont]
            \item $\pi^\omega$ is concentrated on $W^{\alpha,p}(I;\mathcal{X})$ and $t \mapsto {(e_t)}_{\#} \pi^\omega$ is in $ W^{\alpha, p} (I;P_p(\mathcal{X}))$ a.s.;
            \item $(e_t)_\#\pi^\omega =\mu_t^{\omega} $ a.s. for all $t\in I$;
            \item $\pi$ satisfies 
            \begin{equation}\label{eq:minimizer_pi_Walphap_random}
                \Vert\mu\Vert_{W^{\alpha,p}}^p \leq \mathbb{E} \left[\int_{\Gamma_T} |\gamma|_{W^{\alpha,p}}^p \d \pi (\gamma) \right] < + \infty.
            \end{equation}
        \end{enumerate}
     The same statement holds for $|\cdot|_{b^{\alpha,p}}$.
    \end{theorem}

    In the deterministic setting, it is observed in \cite[Proposition 1.4]{Abedi2025paths} that the existence of a lift satisfying the equality
    $|\mu|_{W^{\alpha,p}}^p = \int |\gamma|_{W^{\alpha,p}}^p \d \pi$ imposes a strong condition on $(\mu_t)$, called \emph{compatibility}. This is the case in the stochastic setting as well.
    In Proposition \ref{prop:equality_implies_compatibility_random}, we show that the existence of a random lift $\pi$ achieving equality in \eqref{eq:minimizer_pi_Walphap_random} implies compatibility of $(\mu_t)$ in $P_{p,\Omega}(\mathcal{X})$ in the sense defined below and spelled out after Definition \ref{def:compatibility_random}.
    \begin{definition}[Compatibility of random measures in $P_{p,\Omega}(\mathcal{X})$]\label{def:compatibility_random_intro}
    We say a collection of random measures $\mathcal{M} \subset P_{p,\Omega}(\mathcal{X})$ is compatible in $P_{p,\Omega}(\mathcal{X})$, if, for every finite subcollection of $\mathcal{M}$, there exists a random multi-coupling such that all of its two-dimensional marginals are optimal $\mathbb{P}$-a.s.
    \end{definition}

    Our next goal is to construct a lift that achieves equality in \eqref{eq:minimizer_pi_Walphap_random}, for which it is natural to assume compatibility.
    Our construction is the random version of
    \cite{Abedi2025paths}, which itself was an adapted version of Lisini's construction \cite{Lisini2007}. Let $T=1$ for simplicity in notation. 
    
    \vspace{6pt}
    
    \noindent
    \textbf{Construction {\footnotesize$\bigstar$}.}
    Let $(\mu_t)_{t\in [0,1]} \subset P_{p,\Omega}(\mathcal{X})$ be a compatible collection on a complete, separable, and locally compact length metric space $(\mathcal{X},d)$. For each $n \in \mathbb{N}_0$,
    \begin{enumerate}\label{itm:ConstructionB_random}
        \item[1.] Divide the time interval $[0,1]$ into the dyadic dissection $t_i^{(n)} \coloneqq \frac{i}{2^n}, i \in \{0,1,\cdots 2^n\}$. 
        \item[2.] Let $\Upsilon_{n} \in P_{\Omega}(\mathcal{X}^{2^n+1})$ be a random multi-coupling such that for $\mathbb{P}$-a.e. $\omega \in \Omega$,
        \begin{equation}
        (\mathrm{Pr}^{i,j})_\#\Upsilon_n^\omega
        \in
        \mathrm{OptCpl}
        \bigl(
            \mu^\omega_{t_i^{(n)}},
            \mu^\omega_{t_j^{(n)}}
        \bigr)
        \quad
        \forall i,j\in\{0,\ldots,2^n\},
        \end{equation}
        %\begin{equation}\label{eq:compatibility_dyadic_random}
         %       (\text{Pr}^{i,i+\frac{2^n}{2^m}})_{\#} \Upsilon_{n}^\omega \in \text{OptCpl}\big(\mu^\omega_{t^{(n)}_{i}},\mu^\omega_{t^{(n)}_{i+\frac{2^n}{2^m}}}\big) 
        %\end{equation}
        %for all $i \in \big\{k \frac{2^n}{2^m} \big| k \in \{ 0,1,\cdots, 2^m-1\} \big\} $ and $ m \in \{0,1,\cdots n\}$,
        where $\text{Pr}^{i,j}: \mathcal{X}^{2^n+1} \to {\mathcal{X}^2}$ is the projection map to $(i,j)$-th component.

        \item[3.] Construct the random path measure $\pi_n \coloneqq (\ell)_{\#} \Upsilon_n \in P_{\Omega}(\Gamma_1)$, where $\ell : \mathcal{X}^{2^n+1} \to \Gamma_1$ is a measurable geodesic selection and interpolation map connecting the points with constant-speed geodesics. (See Remark \ref{rmk:measurable_geodesic_selection} on the measurability of this map.)
        \item[4.] Take the limit $n \to \infty$ and verify the narrow convergence of the sequence of random path measures $\{\pi_n\}_{n \in \mathbb{N}}$ in $P_{\Omega}(\Gamma_1)$. 
    \end{enumerate}

    The next theorem can be seen as a stochastic superposition principle under the strong assumption of compatibility  (which is crucial here; otherwise, a finite $W^{\alpha,p}$-energy lift may not exist \cite[Example 4.5]{Abedi2025paths}). It is a continuation of deterministic superposition principles based on optimal transport \cite{LottVillani2009,Villani2009,Sturm2006I,AGS2008GFs,Lisini2007,Lisini2016,AbediLiSchultz2024}.
    When applied in the deterministic setting, the theorems in this section recover our previous results in \cite{Abedi2025paths}. Our second main result is Theorem \ref{thm:optimal_lift_mu_Walphap_compatible_random}. 
    
    \begin{theorem}[Construction of a realizing random lift]\label{thm:optimal_lift_mu_Walphap_compatible_random_intro}
        Let $(\mathcal{X}, d)$ be a complete, separable, and locally compact length metric space (e.g. $\mathbb{R}^\mathrm{d}$), and $I \coloneqq [0,T] \subset \mathbb{R}$.
        Let $(\mu_t)_{t \in I}$ be a probability measure-valued stochastic process defined on a probability space $(\Omega, \mathcal{F}, \mathbb{P})$ such that $(\mu_t) \in W^{\alpha,p} (I; P_{p,\Omega}(\mathcal{X}))$ for some $1<p<\infty$ and $\frac{1}{p}< \alpha <  1$. 
        Assume that $(\mu_t)_{t \in I}$ is compatible in $P_{p,\Omega}(\mathcal{X})$.
        Then, construction \hyperref[itm:ConstructionB_random]{{\footnotesize$\bigstar$}} converges narrowly (up to a subsequence) to a random probability measure $\pi \in P_{\Omega} (C(I;\mathcal{X}))$ satisfying
        \begin{enumerate}[label=(\roman*), font=\normalfont]
            \item $\pi^\omega$ is concentrated on $W^{\alpha,p}(I;\mathcal{X})$ and $t \mapsto {(e_t)}_{\#} \pi^\omega$ is in $ W^{\alpha, p} (I;P_p(\mathcal{X}))$ a.s.;
            \item $(e_t)_\#\pi^\omega =\mu_t^{\omega} $ a.s. for all $t\in I$;
            \item $(e_s,e_t)_\# \pi^\omega \in \mathrm{OptCpl}(\mu_s^{\omega}, \mu_t^{\omega})$ a.s. for all $s,t \in I$; and in particular,
            \begin{equation}\label{eq:optimality_pi_1_random_intro}
                \Vert\mu\Vert_{W^{\alpha,p}}^p = \mathbb{E} \left[\int_{\Gamma_T} |\gamma|_{W^{\alpha,p}}^p \d \pi (\gamma) \right].
            \end{equation}
        \end{enumerate}
    The same statement holds for $|\cdot|_{b^{\alpha,p}}$.
    \end{theorem}

    \begin{remark}\label{rmk:random_curves_intro}
        We would like to stress that $(\mu_t) \in W^{\alpha,p} (I; P_{p,\Omega}(\mathcal{X}))$ does not necessarily imply $(\mu_t^\omega) \in W^{\alpha,p} (I; P_{p}(\mathcal{X}))$ a.s.
        However, this holds under the stronger assumption that $\mu$ is given by a measurable map $\mu: (\Omega,\mathcal{F},\mathbb{P}) \to C(I;P_p(\mathcal{X}))$, which might be the case in some applications. Under this extra assumption, we can strengthen the statements above as justified in Remark \ref{rmk:random_curves}:
        \textit{
         \begin{enumerate}[label=(\roman*), font=\normalfont]
            \item[(ii)] $(e_t)_\#\pi^\omega =\mu_t^{\omega} $ for all $t\in I$ a.s.;
            \item[(iii)] $(e_s,e_t)_\# \pi^\omega \in \mathrm{OptCpl}(\mu_s^{\omega}, \mu_t^{\omega})$ for all $s,t \in I$ a.s.; and,
            \begin{equation}\label{eq:pointwise_Walphap_energy_equality}
            |\mu^\omega|_{W^{\alpha,p}}^p = \int_{\Gamma_T} |\gamma|_{W^{\alpha,p}}^p \d \pi^\omega(\gamma) \quad \textrm{a.s}.
            \end{equation}
        \end{enumerate}
        }
    \end{remark}
    
\subsection{Applications} We now focus on measure-valued processes $(\mu_t)$ that satisfy           
 \begin{equation}\label{eq:applications_Holder_condition_2_intro}
            \mu_0 \in P_{p,\Omega} (\mathcal{X}), \qquad \textrm{and} \qquad
            \mathbb{W}^p_p (\mu_s,\mu_t) \leq c^p |t-s|^{\upgamma p } \quad \forall s,t \in I ,
    \end{equation}
    for some exponents $p \in [1,\infty)$, $\upgamma \in [0,1]$, and a positive constant $c$. 
    By the triangle inequality for the metric $\mathbb{W}_p$, we have that $\mu_t \in P_{p,\Omega}(\mathcal{X})$ for all $t\in I$.
    Thus, the condition above can be equivalently expressed in the following compact way:
    \begin{equation}\label{eq:applications_Holder_condition_3_intro}
        (\mu_t) \in C^{\upgamma\textrm{-} \mathrm{H\ddot{o}l}}(I;P_{p,\Omega}(\mathcal{X})).
    \end{equation}
    We will use the fact that $C^{\upgamma\textrm{-}\mathrm{H\ddot{o}l}} \subset W^{\alpha, p}$ holds trivially on any metric space when $0 < \alpha < \upgamma \leq 1$.

\begin{figure}
\centering
\begin{subfigure}{.430\textwidth}
  \centering
  \includegraphics[width=0.87\linewidth]{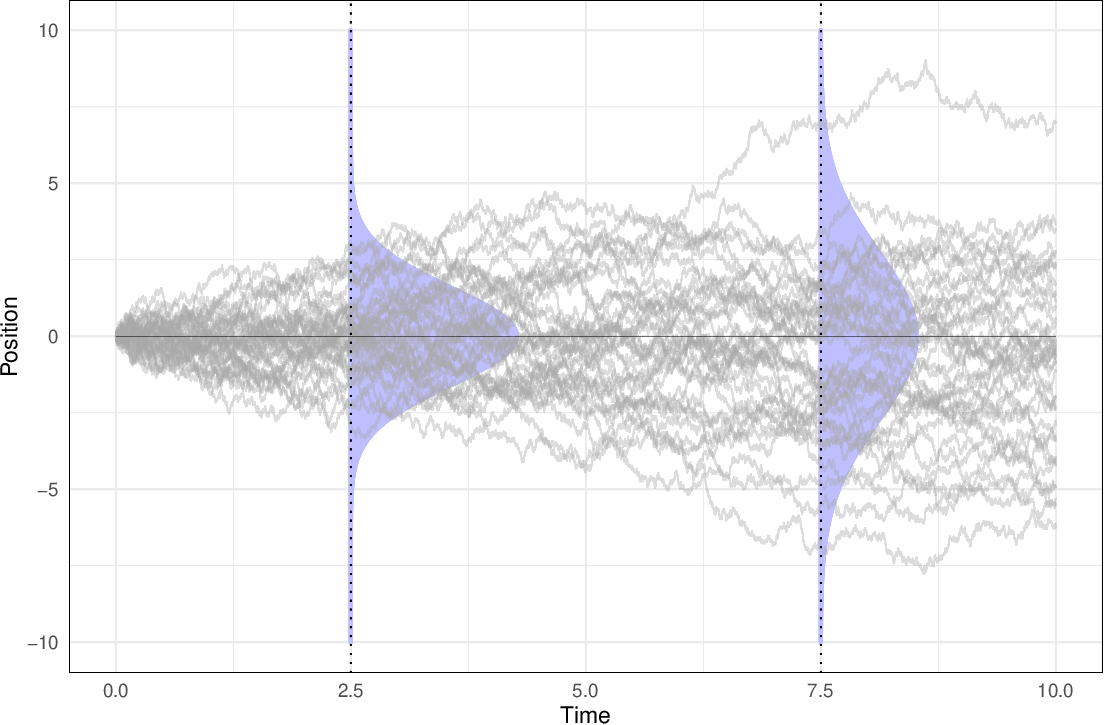}
  \captionsetup{font=footnotesize,justification=centering}
\end{subfigure}
\begin{subfigure}{.430\textwidth}
  \centering
  \includegraphics[width=0.87\linewidth]{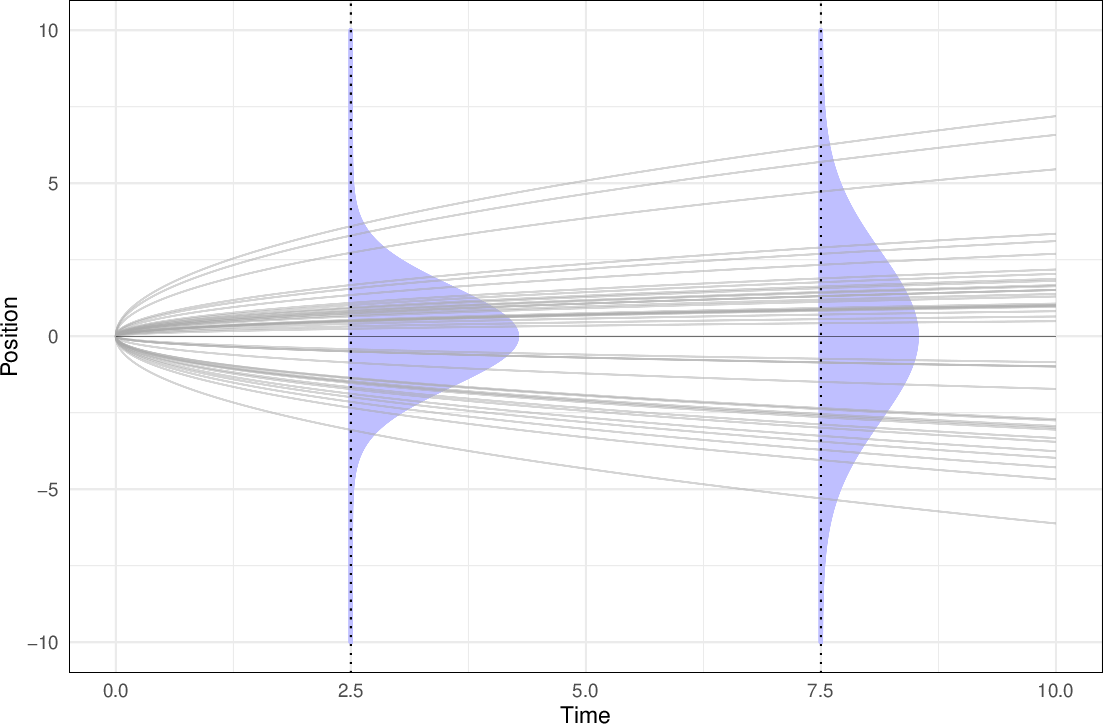}
  \captionsetup{font=footnotesize,justification=centering}
\end{subfigure}
\vspace{-5pt}
\captionsetup{font=footnotesize}
\caption{Two different stochastic processes for the fundamental solution of the heat equation $\partial_t \mu_t\mkern-1mu=\mkern-1mu\frac12 \Delta \mu_t$  with $\mu_0 \mkern-1mu=\mkern-1mu\delta_0 $ on $\mathbb{R}$. The \textbf{left} shows sample paths of the Wiener measure \cite{Wiener1923}--Brownian motions $B_t$. We know $\mu_t = \mathrm{Law} (B_t)$.
The \textbf{right} corresponds to the process obtained from optimal transport, studied e.g. in \cite{AGS2008GFs,Lisini2007,Taghvaei2016}. The path law $\pi$ of this process minimizes the kinetic energy and relates 2-Wasserstein metric speed of the heat flow to its Fisher information: $|\dot\mu_t|^2 = \int_{\Gamma_T} |\dot \gamma_t |^2 \d \pi =  \int_{\mathbb{R} } \frac{|\nabla \rho_t|^2}{4\rho_t} \d x$ for all $t>0$, where $\rho_t$ is the density of $\mu_t$. The lift $\pi$ also minimizes the fractional Sobolev energy: $ |\mu|_{W^{\alpha,p}}^p = \int_{\Gamma_T} |\gamma|_{W^{\alpha,p}}^p \d \pi$ e.g. for $p=2$ and $\alpha \in (\frac1p,1)$, as a result of Corollary \ref{crl:lift_mu_Walphap_R1} and the path regularity of the heat flow in the Wasserstein space.
% for all $p\in (1,\infty)$ and $\alpha \in (\frac{1}{p},\min(1,\frac12+\frac1p))$
}
\label{fig:heat_eq_lifts}
\vspace{13 pt}
\centering
\begin{subfigure}{.430\textwidth}
  \centering
  \includegraphics[width=0.87\linewidth]{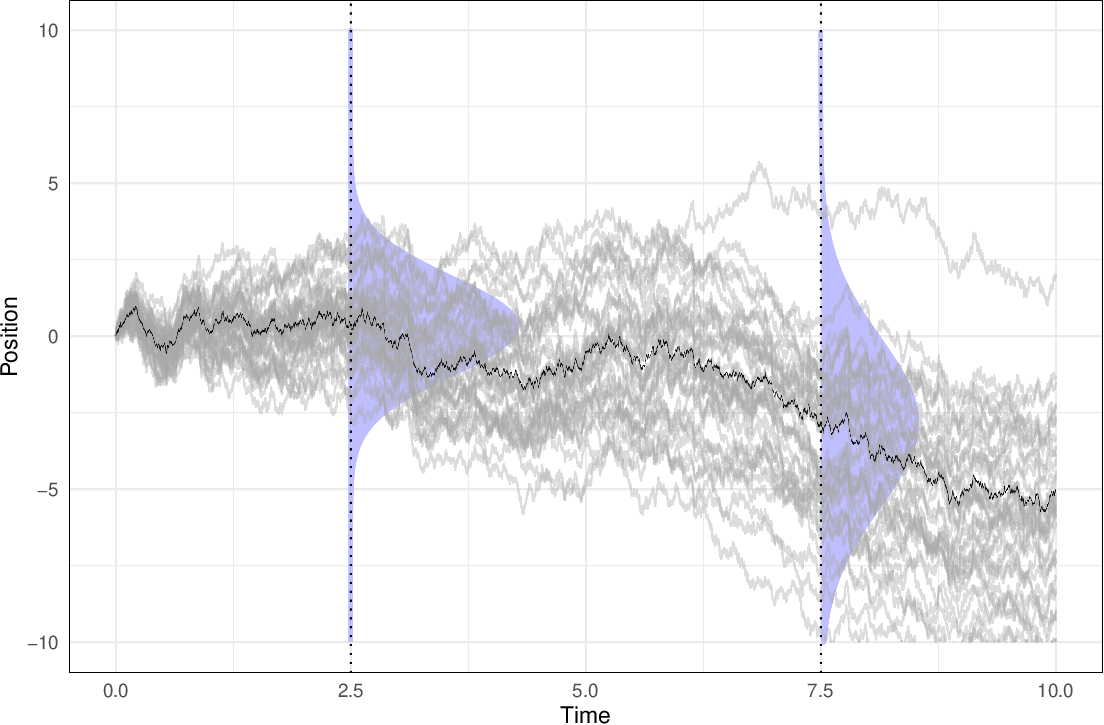}
  %\captionsetup{font=footnotesize,justification=centering}
  %\caption*{\footnotesize \color{darkgray} ...}
\end{subfigure}
\begin{subfigure}{.430\textwidth}
  \centering
\includegraphics[width=0.87\linewidth]{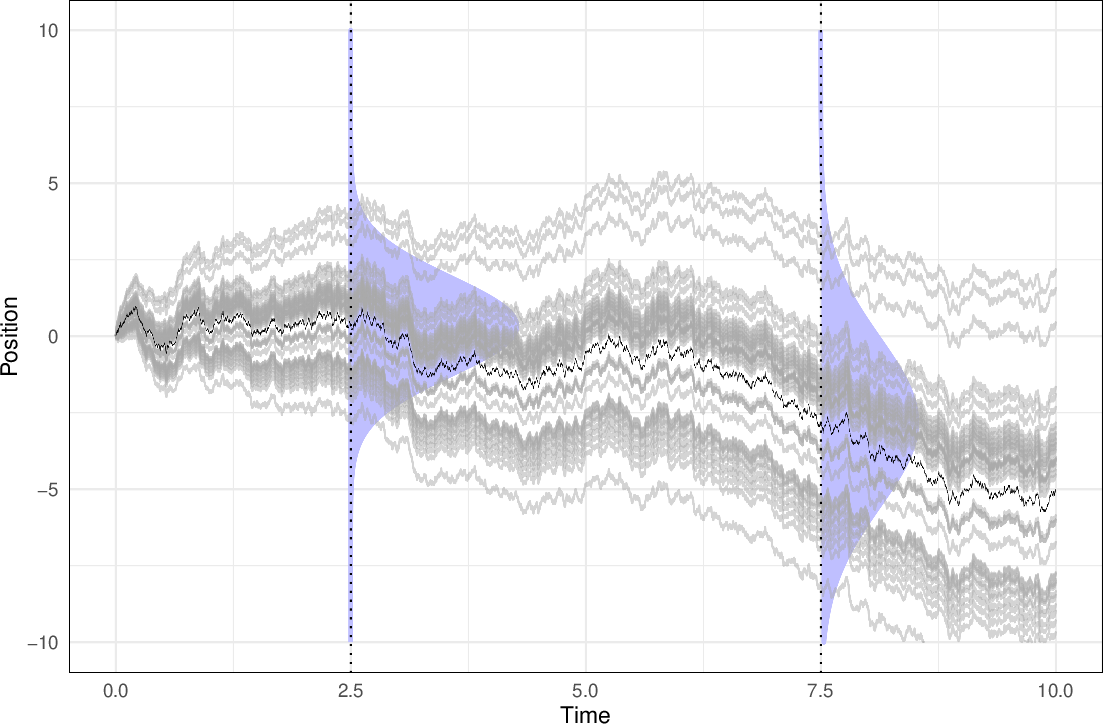}
  %\captionsetup{font=footnotesize,justification=centering}
  %\caption*{\footnotesize \color{darkgray} ...}
\end{subfigure}
\vspace{-5pt}
\captionsetup{font=footnotesize}
\caption{Two different stochastic processes for the solution of the stochastic heat equation $\d \mu_t\mkern-2mu=\mkern-2mu \Delta \mu_t \d t - \nabla\mkern-4mu\cdot\mkern-4mu \mu_t \d W_t$ with $\mu_0\mkern-2mu = \mkern-2mu\delta_0$ on $\mathbb{R}$, perturbed by a standard Brownian\,\,motion\,\,$W_t$. The \textbf{left} shows realizations of $B_t + W_t$, where $B_t$ is a Brownian motion independent of the (highlighted) common noise $W_t$. We have $\mu_t = \mathrm{Law} (B_t + W_t| \mathcal{F}_t^W)$ a.s., where $\mathcal{F}_t^W \coloneqq \sigma (W_s: 0\leq s \leq t)$, and thus
$(\mu_t) \in C^{\frac{1}{2}\textrm{-} \mathrm{H\ddot{o}l}}(I;P_{p,\Omega}(\mathcal{\mathbb{R}}))$ for all $p \in [1,\infty)$. 
The \textbf{right} corresponds to the process obtained from optimal transport, as a result of Corollary \ref{crl:lift_mu_Holder_R1_random}.
The random path law $\pi$ of this process minimizes the fractional Sobolev energy in expectation and, by taking Remark \ref{rmk:random_curves_intro} into account, even point-wisely: $|\mu|_{W^{\alpha,p}}^p = \int_{\Gamma_T} |\gamma|_{W^{\alpha,p}}^p \d \pi $ a.s. for all $p\in (2,\infty)$ and $\alpha \in (\frac1p,\frac12)$.
%optimal transport lift with respect to $W_p$ for all $p>2$
}
\label{fig:heat_eq_stoch_lifts}
\end{figure}
    
\subsubsection{Applications for measure-valued processes on \texorpdfstring{$\mathbb{R}$}{R}}\label{subsec:application_R}

As a first application, we apply Theorem \ref{thm:optimal_lift_mu_Walphap_compatible_random_intro} to the case $\mathcal{X} = \mathbb{R}$, on which all probability measures with finite $p$-moments form a compatible family. Below is an excerpt from Corollary \ref{crl:lift_mu_Holder_R1_random}.

\begin{corollary}\label{crl:lift_mu_Holder_R1_random_intro}
        Let $(\mu_t)_{t \in I}$ with $I \coloneqq [0,T]\subset \mathbb{R}$ be a probability measure-valued stochastic process defined on a probability space $(\Omega, \mathcal{F}, \mathbb{P})$ such that $(\mu_t) \in C^{\upgamma\textrm{-} \mathrm{H\ddot{o}l}}(I;P_{p,\Omega}(\mathbb{R}))$ for some $1<p<\infty$ and $\frac{1}{p}<\upgamma \leq 1$.
        Denote by $F_t$ the CDF of $\mu_t$ and by $F_t^{-1}$ its generalized inverse.
        Then there exists a random probability measure $\pi \in P_{\Omega}(C(I;\mathbb{R}))$ such that
        \begin{equation}\label{eq:lift_mu_Holder_R1_random_intro}
            (e_{t_1}, \cdots, e_{t_j})_\#\pi^\omega = ((F_{t_1}^\omega)^{-1}, \cdots, (F_{t_j}^\omega)^{-1} )_\# \mathrm{Leb}|_{[0,1]} \quad \mathbb{P}\textrm{-a.s. }
        \end{equation}
         for any finite sequence $t_1, \cdots, t_j \in I$. In particular, $\pi$ satisfies
    \begin{enumerate}[label=(\roman*), font=\normalfont]
                \item $\pi^\omega$ is concentrated on $W^{\alpha,p}(I;\mathbb{R}) \subset C^{\alpha - \frac{1}{p}\textrm{-}\mathrm{H\ddot{o}l}} (I;\mathbb{R})$ a.s. for any $\alpha \in (\frac{1}{p},\upgamma)$;
                %and $t\mapsto (e_t)_\#\pi^\omega$ is in $W^{\alpha,p} (I;P_p(\mathbb{R})) \subset C^{(\alpha-\frac{1}{p})\textrm{-} \mathrm{H\ddot{o}l}}(I;P_p(\mathbb{R})) $ a.s. for any $\alpha \in (\frac{1}{p},\upgamma)$;
                \item $(e_t)_\#\pi^\omega=\mu_t^\omega$ a.s. for all $t\in I$;
                \item $(e_s,e_t)_\# \pi^\omega \in \mathrm{OptCpl}(\mu_s^\omega, \mu_t^\omega)$ a.s. for all $s,t \in I$; and \eqref{eq:optimality_pi_1_random_intro} holds for any $\alpha \in (\frac{1}{p},\upgamma)$.   
            \end{enumerate}
\end{corollary}
The associated process to $\pi$ is usually called the \emph{quantile process} of $(\mu_t)$, see also \cite{BoubelJuillet2022, BoubelJuillet2025}. Note that $\pi$ here does not depend on $(\upgamma,\alpha,p)$. In the deterministic setting (\cref{subsec:corollaries_R_deterministic}), \eqref{eq:lift_mu_Holder_R1_random_intro} gives a unique lift, since finite-dimensional marginals uniquely determine the path measures on $C(I;\mathcal{X})$.  
The application of these corollaries to the solution of the heat equation and a stochastic heat equation on $\mathbb{R}$ is illustrated in Figs. \ref{fig:heat_eq_lifts} and \ref{fig:heat_eq_stoch_lifts}, respectively, and will be discussed in \cref{subsec:example_sHE}.
%Before that, we consider a general class of stochastic PDEs.

\subsubsection{Applications for stochastic Fokker--Planck equations on \texorpdfstring{$\mathbb{R}^\mathrm{d}$}{Rd}}\label{subsec:application_sFPE}
   We consider measure-valued solutions $(\mu_t)_{t \in [0,T]}$ to the stochastic Fokker--Planck(--Kolmogorov) equations of the form 
    \begin{equation}\label{eq:stochastic_Fokker-Planck}\tag{S-FPE}
    \d \mu_t  \,\, =   - \nabla \cdot (\mu_t b_t) \d t    + \nabla^2 : (\mu_t a_t ) \d t - \nabla \cdot (\mu_t \sigma_t \d W_t) \qquad \text{in }  (0,T) \times \mathbb{R}^{\mathrm{d}},
    \end{equation}
    with a (possibly random) probability measure $\mu_0$ as initial condition.
    This stochastic PDE is posed on a filtered probability space  $(\Omega,\mathcal{F},\mathbb{F}\coloneqq (\mathcal{F}_t)_{t \in [0,T]}, \mathbb{P})$  supporting a standard $\mathbb{R}^\mathrm{d}$-valued $\mathbb{F}$-adapted Brownian motion $(W_t)$.
    The coefficients $(b,a,\sigma): [0,T]\mkern-1mu\times\mkern-1mu\mathbb{R}^\mathrm{d} \times \Omega \to \mathbb{R}^{\mathrm{d} }\times\mkern-1mu\mathbb{S}^\mathrm{d}_{\geq 0}\mkern-1mu\times \mathbb{R}^{\mathrm{d} \times \mathrm{d} }$, where $\mathbb{S}^\mathrm{d}_{\geq 0}$ denotes the set of symmetric nonnegative definite $\mathrm{d} \times \mathrm{d}$-matrices, are measurable with respect to the product of the $\mathbb{F}$-progressive $\sigma$-algebra on $[0,T]\times \Omega$ and $\mathcal{B}(\mathbb{R}^\mathrm{d})$.  
    We sometimes write $b_t(x,\omega) \coloneqq b(t,x,\omega)$, and similarly for the others. This framework also covers the important case where the dependence on $\omega$ is through the measure $\mu_t$ itself (see \cite[Remark 1.1]{LackerShkolnikovZhang2023}).
    An $\mathbb{F}$-adapted probability measure-valued process $(\mu_t)$ with a.s. continuous sample paths with respect to narrow convergence in $P(\mathbb R^\mathrm{d})$ is called here a solution to \eqref{eq:stochastic_Fokker-Planck} if the quadruplet $(\mu,b,a,\sigma)$ satisfies
    \begin{equation} \label{eq:a.s._integrability_assumption_SFPE}
         \mathbb{E} \left[\int_{0}^{T} \int_{\mathbb{R}^\mathrm{d}} \big( |b_t| + |a_t| + |\sigma_t \sigma_t^\top| \big) \d \mu_t  \d t \right] < + \infty ,
    \end{equation}
    and \eqref{eq:stochastic_Fokker-Planck} holds in the distributional sense, i.e.,
    \begin{multline}
        \int_{\mathbb{R}^\mathrm{d}} \varphi \d \mu_t - \int_{\mathbb{R}^\mathrm{d}} \varphi \d \mu_0  \\ =  \int_0^t \int_{\mathbb{R}^\mathrm{d}}  \nabla \varphi \cdot b_s  \d \mu_s \d s +  \int_0^t  \int_{\mathbb{R}^\mathrm{d}}  \nabla^2 \varphi : a_s \d \mu_s \d s + \sum\nolimits_{i=1}^{\mathrm{d}}  \int_0^t \int_{\mathbb{R}^\mathrm{d}} \nabla \varphi \cdot \sigma^{\cdot i}_s  \d \mu_s \d W_s^i
    \end{multline}
    holds $\mathbb{P}$-a.s. for each $t \in [0,T]$ and $\varphi \in C^\infty_c (\mathbb{R}^\mathrm{d})$
    (see also \cite[Definition 4.1 and Remark 4.2]{Rehmeier2023} and \cite[Lemma 4.3]{Rehmeier2023} for the conservation of mass). All stochastic integrals are in the It\^{o} sense. The following notation is used. Given vectors $v, w \in \mathbb{R}^{\mathrm{d} \times 1}$ and matrices $A, B \in \mathbb{R}^{\mathrm{d} \times \mathrm{d}}$, we set: 
    \begin{align}
    v \cdot w   \coloneqq  \sum\nolimits_{i=1}^{\mathrm{d}} v^i w^i,   \qquad
    A : B   \coloneqq \sum\nolimits_{i,j =1 }^{\mathrm{d}} A^{ij} B^{ij}, \qquad
    |v|   \coloneqq   \sqrt{v.v}, \qquad\, |A|  \coloneqq   \sqrt{A:A}.
    \end{align}
    
    These types of SPDEs naturally arise as the governing equations for the conditional time-marginals of solutions to McKean–Vlasov SDEs with two independent noises, conditioned on the common noise (see \cite{Carmona2016Lectures,KurtzXiong1999,CoghiGess2019,GessGvalaniKonarovskyi2025}).
    In this setting, the diffusion matrix $a$ carries contributions from both sources of randomness. Consequently, a condition that naturally arises in this SDE-to-SPDE connection is that the couple $(a,\sigma)$ satisfies the so-called \emph{parabolicity condition}:  
    \begin{equation}\label{eq:parabolicity_condition}
        (2 a  - \sigma \sigma^\top ) \in \mathbb{S}^\mathrm{d}_{\geq 0} \quad \forall(t,x,\omega).
    \end{equation}
    This ensures the existence of a well-defined diffusion coefficient corresponding to the individual noise, given by $\alpha \coloneqq ( 2 a  - \sigma \sigma^\top )^{1/2}$.

    Conversely, Lacker--Shkolnikov--Zhang \cite{LackerShkolnikovZhang2023} recently established an SPDE-to-SDE connection by developing a stochastic superposition principle (as a continuation of the Ambrosio--Figalli--Trevisan superposition principle and subsequent works \cite{Ambrosio2008,AGS2008GFs,Figalli2008,Trevisan2016,Bogachev2021,BarbuRockner2020nonlinearsuperpositionprinciple,RoecknerXieZhang2020nonlocalsuperpositionprinciple}) that asserts that any solution $(\mu_t)$ to \eqref{eq:stochastic_Fokker-Planck} can be lifted to a solution of a conditional SDE if the couple $(a,\sigma)$ satisfies \eqref{eq:parabolicity_condition} and the triple $(\mu,b,a)$ satisfies the \emph{$p$-integrability condition}:
    \begin{equation} \label{eq:p_integrability_assumption_SFPE}
         \mathcal{E}_{p,T} (\mu,b,a) \coloneqq T^{\frac{p-1}{2}} \, \mathbb{E} \left[ \int_{0}^{T}\int_{\mathbb{R}^\mathrm{d}} |b_t|^p \d \mu_t  \d t  \right]  + \left( \mathbb{E} \left[ \int_{0}^{T} \int_{\mathbb{R}^\mathrm{d}} |a_t|^p \d \mu_t  \d t  \right]  \right)^{1/2} < + \infty,
    \end{equation}
    for some $p>1$ (note that \eqref{eq:p_integrability_assumption_SFPE} and \eqref{eq:parabolicity_condition} guarantee \eqref{eq:a.s._integrability_assumption_SFPE}, but we define the $p$-energy $\mathcal{E}_{p,T}$ this way for reasons that will soon be clear).
    More precisely, \cite[Theorem 1.3]{LackerShkolnikovZhang2023} states that there exists a complete filtered probability space $(\tilde{\Omega},\tilde{\mathcal{F}},\tilde{\mathbb{F}}, \tilde{\mathbb{P}})$, extending $(\Omega,\mathcal{F},\mathbb{F},\mathbb{P})$, which supports a standard $\mathbb{R}^\mathrm{d}$-valued $\tilde{\mathbb{F}}$-adapted Brownian motion $(B_t)$ independent of $\mathcal{F}_T$ and a continuous $\tilde{\mathbb{F}}$-adapted $\mathbb{R}^\mathrm{d}$-valued process $(X_t)$ such that 
    \begin{enumerate}[label=(\roman*), font=\normalfont]
            \item the sample paths of $(X_t)$ are integral solutions to the SDE
            \begin{align}\label{eq:SDE_associated_sFPE}
                  \d X_t   = b_t(X_t,\omega) \d t  + \alpha_t (X_t,\omega) \d B_t + \sigma_t(X_t,\omega) \d W_t,
        \end{align}
        \item $\mu_t = \mathrm{Law} (X_t | \mathcal{F}_t) = \mathrm{Law} (X_t | \mathcal{F}_T)$ a.s. for all $t \in [0,T]$.
        \end{enumerate}
    Moreover, $(W_t)$ is also an $\tilde{\mathbb{F}}$-adapted Brownian motion. In applications, $\mathbb{F}$ is usually the filtration generated by the common noise, i.e., $\mathcal{F}_t = \mathcal{F}^W_t \coloneqq \sigma (W_s : 0\leq s \leq t )$ for all $t \in [0,T]$.

    The particle solution \eqref{eq:SDE_associated_sFPE} provides us with a random lift given by $\tilde{\pi}  \coloneqq \mathrm{Law} (X | \mathcal{F}_T)$ a.s. However, this lift does not a priori have finite fractional Sobolev energy, nor does it necessarily minimize the energy.
    We first show that, under the additional assumption that $p > 3$, this random lift has finite energy.   
    This allows us to apply Theorem \ref{thm:existence_of_Walphap_minimizer_random_intro} and immediately conclude the existence of an energy-minimizing process in the next corollary. Its proof is given in \cref{sec:appendix_s}.

    \begin{corollary}[Existence of minimizing processes for S-FPEs]\label{crl:minimizing_processes_sFPE_intro}
        Let $(\Omega,\mathcal{F},\mathbb{F}\coloneqq(\mathcal{F}_t)_{t \in [0,T]}, \mathbb{P})$ be a filtered probability space supporting a standard $\mathbb{R}^\mathrm{d}$-valued $\mathbb{F}$-adapted Brownian motion $(W_t)$ and suppose that $\mathcal{F}_T$ is countably generated. 
        Let $(\mu_t)_{t \in [0,T]}$ be an a.s. continuous $($with respect to narrow convergence in $P(\mathbb
        R^\mathrm{d}))$ $\mathbb{F}$-adapted probability measure-valued process satisfying \eqref{eq:stochastic_Fokker-Planck} with coefficients $(b,a,\sigma): [0,T] \times \mathbb{R}^\mathrm{d} \times \Omega \to \mathbb{R}^{\mathrm{d} }\times \mathbb{S}^\mathrm{d}_{\geq 0} \times \mathbb{R}^{\mathrm{d} \times \mathrm{d} }$ that are measurable with respect to the product of the $\mathbb{F}$-progressive $\sigma$-algebra on $[0,T]\times \Omega$ and $\mathcal{B}(\mathbb{R}^\mathrm{d})$.
        Assume that the couple $(a,\sigma)$ satisfies \eqref{eq:parabolicity_condition}, the triple $(\mu,b,a)$ satisfies the $p$-integrability condition \eqref{eq:p_integrability_assumption_SFPE}, and the initial condition satisfies $\mu_0 \in P_{p,\Omega}(\mathbb{R}^\mathrm{d})$ for some $p>1$. 
        Then, we have $(\mu_t) \in C^{\upgamma\textrm{-} \mathrm{H\ddot{o}l}}([0,T];P_{p,\Omega}(\mathbb{R}^\mathrm{d}))$ with $\upgamma \coloneqq \frac{1}{2} - \frac{1}{2p}.$
        If further $p>3$, then for each $\alpha \in (\frac{1}{p},\upgamma)$, we have $(\mu_t) \in W^{\alpha,p}([0,T];P_{p,\Omega}(\mathbb{R}^\mathrm{d}))$ and there exists a random minimizer $\pi \in P_{\Omega}(C([0,T];\mathbb{R}^\mathrm{d}) ) $ to Problem \ref{prob:variational_problem_random} for the energy $\Psi(\gamma)  = |\gamma|^p_{W^{\alpha,p}} $. In particular, 
        \begin{enumerate}[label=(\roman*), font=\normalfont]
            \item $\pi^\omega$ is concentrated on $W^{\alpha,p}([0,T];\mathbb{R}^\mathrm{d}) \subset  C^{\alpha - \frac{1}{p}\textrm{-}\mathrm{H\ddot{o}l}} ([0,T];\mathbb{R}^\mathrm{d}) $ a.s.;
            %and $t \mapsto {(e_t)}_{\#} \pi^\omega$ is in $ W^{\alpha, p} (I;P_p(\mathbb{R}^\mathrm{d}))$ a.s.;
            \item $(e_t)_\#\pi^\omega =\mu_t^{\omega} $ a.s. for all $t\in [0,T]$;
            \item $\pi$ satisfies 
            \begin{equation}\label{eq:minimizer_to_SFPEs}
                \Vert\mu\Vert_{W^{\alpha,p}}^p \leq \mathbb{E} \left[\int_{\Gamma_T} |\gamma|_{W^{\alpha,p}}^p \d \pi (\gamma) \right] \leq c \, \mathcal{E}_{p,T} (\mu,b,a),
            \end{equation}
            where $c = c (\alpha,p,\mathrm{d},T)$ is a positive constant. 
        \end{enumerate}
     %The same statement holds for $|\cdot|_{b^{\alpha,p}}$. 
    \end{corollary} 

    We emphasize that the first inequality in \eqref{eq:minimizer_to_SFPEs} is \emph{not} expected to be an equality in general. As already discussed, this happens only under the strong condition of compatibility.
    In the second inequality, the $p$-energy $\mathcal{E}_{p,T} (\mu,b,a)$ in the form of \eqref{eq:p_integrability_assumption_SFPE} arises in the energy estimate as a result of applying the Burkholder--Davis--Gundy inequality and H\"{o}lder's inequality.
    
    Given a solution $(\mu_t)$ to \eqref{eq:stochastic_Fokker-Planck}, there can be many choices of coefficients $(b,a,\sigma)$ that have finite $p$-energy and satisfy \eqref{eq:stochastic_Fokker-Planck}. Each provides us with a particle representation of the form \eqref{eq:SDE_associated_sFPE}.
    It is an open question---to the best of our knowledge---whether $W^{\alpha,p}$-energy-minimizing particle representations are of this form, i.e., whether the sample paths of $\pi$ above also solve an SDE like \eqref{eq:SDE_associated_sFPE} for some choice of $(b,a,\sigma)$. 
    This is true in simple cases, such as the next example in $\mathbb{R}$, but in full generality, it is not clear to us.
    Nevertheless, we would like to emphasize that even in such cases, \emph{$W^{\alpha,p}$-energy-minimizing processes do not necessarily correspond to the choice of $(b,a,\sigma)$ that minimizes the energy $\mathcal{E}_{p,T} (\mu,b,a)$.}
    This can already be seen in the deterministic setting (i.e. $\sigma = 0$), for example, in the case of the heat equation on $\mathbb{R}$.
    If it turns out that $W^{\alpha,p}$-energy-minimizing processes for \eqref{eq:stochastic_Fokker-Planck} do solve an SDE, 
    we speculate that this corresponds to a choice of $(b,a,\sigma)$ that at least makes $\alpha = 0$, eliminating the additional noise. This is what we observe in the next example.
    
    \subsubsection{Example: a stochastic heat equation on \texorpdfstring{$\mathbb{R}$}{R}}\label{subsec:example_sHE}
    We consider probability measure-valued solutions $(\mu_t)_{t \in [0,T]}$ to the following stochastic heat equation on the real line
    \begin{align}\label{eq:stoch_heat_eq_1d}\tag{S-HE}
        \begin{cases}
        \d \mu_t  =  \Delta \mu_t \d t - \nabla \cdot ( \mu_t \d W_t ), \qquad \textrm{in }  (0, T) \times \mathbb{R} ,\\
        \mu_0  = \delta_0,
        \end{cases}
    \end{align}
    where $(W_t)$ is a standard $\mathbb{R}$-valued Brownian motion defined on a filtered probability space  $(\Omega,\mathcal{F},\mathbb{F}\coloneqq (\mathcal{F}_t)_{t \in [0,T]}, \mathbb{P})$. We write $\mathcal{F}_t^W \coloneqq \sigma (W_s: 0\leq s \leq t)$ for all $t \in [0,T]$. We discuss two particle solutions to the equation.

    \vspace{2pt}
    
    \noindent
    \textbf{Particle representation 1.}
    \eqref{eq:stoch_heat_eq_1d} is a particular case of \eqref{eq:stochastic_Fokker-Planck} with coefficients $b=0$, $a=1$, $\sigma=1$, and thus $\alpha \coloneqq \sqrt{2a -\sigma^2} = 1$. 
    Then, \eqref{eq:SDE_associated_sFPE} gives the following  particle solution
    \begin{equation}
        \begin{cases}
        \d X_t = \d B_t + \d W_t,\\
        X_0 = 0,
        \end{cases}
    \end{equation}
    where $(B_t)$ is another standard Brownian motion on $\mathbb{R}$ independent of the common noise $(W_t)$. 
    The fact that
    $\mu_t = \mathrm{Law} (X_t| \mathcal{F}_t^W)$ a.s. 
    can be also easily verified as a result of It\^{o}'s formula, taking conditional expectation with respect to $\mathcal{F}_t^W$, and using properties of the It\^{o} integral (see e.g. \cite[Lemma B.2]{CoghiGess2019}, \cite[Lemma 2]{YangMehtaMeyn2011}, \cite[Lemma 10]{AbediSuracePfister2022}). In particular, we conclude $\mu_t = \mathcal{N}(W_t, t)$ (and this is actually the unique solution of \eqref{eq:stoch_heat_eq_1d} by \cite[Theorem 5.4]{CoghiGess2019}).
    Fig. \ref{fig:heat_eq_stoch_lifts} (left) shows the sample paths of $(X_t)$. 

    \vspace{2pt}
    
    \noindent
    \textbf{Particle representation 2.}
    We now present a particle solution $(Y_t)$ that minimizes the fractional Sobolev energy.
    First, using $(X_t)$ and a Gaussian computation, we estimate the path regularity of $(\mu_t)$ and conclude that
    $(\mu_t) \in C^{\frac{1}{2}\textrm{-} \mathrm{H\ddot{o}l}}([0,T];P_{p,\Omega}(\mathbb{R}))$
    for all $p\geq 1$. 
    Therefore, Corollary \ref{crl:lift_mu_Holder_R1_random_intro} applies for $p> 2$ and we have the existence of a minimizing lift $\pi$ that achieves the equality \eqref{eq:optimality_pi_1_random_intro} even point-wisely here as in \eqref{eq:pointwise_Walphap_energy_equality} for $\alpha \in (\frac1p,\frac12)$.
    In this example, we can even show that the sample paths of $\pi$ are integral solutions to an SDE.
    Let $ \rho_t$ and $F_t$ be the density and the cumulative distribution function of the random measure $\mu_t$, respectively.
    To construct $\pi$ according to \eqref{eq:lift_mu_Holder_R1_random_intro}, we label the particles with a uniformly distributed random variable $q \sim \mathrm{Leb}|_{[0,1]}$ (say on the same probability space) and define $Y_t^q \coloneqq F_t^{-1} (q)$ for all time. Then, for fixed $q$, the stochastic differential of $Y_t^q$ is given by
    \begin{align}
    \d Y_t^q & = \d F_t^{-1} (q) = \d \left(c(q) \sqrt{t}  + W_t \right) =  \frac{1}{2\sqrt{t}} \frac{Y_t^q - W_t}{\sqrt{t}} \d t + \d W_t, 
    \end{align}
    where $c(q) \coloneqq \sqrt{2} \, \mathrm{erf}^{-1}(2q -1 )$.
    Since the logarithmic gradient of the density is given here by $\nabla \log \rho_t(x) = - \frac{x - W_t}{t}$, we can shortly write:
    \begin{equation}\label{eq:SDE_Yt}
        \begin{cases}
        \d Y_t = - \frac12 \nabla \log \rho_t (Y_t) \d t + \d W_t,\\
        Y_0 = 0.
        \end{cases}
    \end{equation}
    By its very construction, $(Y_t)$ is a particle solution to \eqref{eq:stoch_heat_eq_1d}, i.e., 
    $
    \mu_t = \mathrm{Law} (Y_t| \mathcal{F}_t^W)$ a.s. 
    and we have 
    $\pi =  \mathrm{Law} (Y| \mathcal{F}_T^W) \textrm{ a.s. }$ Fig. \ref{fig:heat_eq_stoch_lifts} (right) shows the sample paths of $(Y_t)$. 

    \vspace{2pt}
    
    \noindent
    \textbf{Remark on SDE \eqref{eq:SDE_Yt}.}
    We would like to add that the particle solution \eqref{eq:SDE_Yt} can be recovered in a different way by rewriting \eqref{eq:stoch_heat_eq_1d} in an alternative form:
    \begin{equation}
        \d \rho_t  =  \frac12 \nabla \cdot \left(\rho_t  \nabla \log \rho_t \right) \d t  + \frac12 \Delta \rho_t \d t -  \nabla \cdot (\rho_t \d W_t), \qquad \textrm{in }  (0, T) \times \mathbb{R}.
    \end{equation}
    This corresponds to \eqref{eq:stochastic_Fokker-Planck} with coefficients of Nemytskii-type: $b=- \frac12 \nabla \log \rho $, $a= \frac12$, $\sigma=1$, and thus $\alpha \coloneqq \sqrt{2a -\sigma^2} = 0$.
    One can verify that these coefficients also satisfy the $p$-integrability condition \eqref{eq:p_integrability_assumption_SFPE} for some $p>1$. 
    Then, the superposition principle of Lacker--Shkolnikov--Zhang \cite{LackerShkolnikovZhang2023} immediately recovers \eqref{eq:SDE_Yt}.  
    We note that this choice forces the diffusion coefficient $\alpha$ of the additional noise to be zero, which may give insight into why it corresponds to the energy-minimizing process.
    The SDE \eqref{eq:SDE_Yt} has infinitely many strong solutions (of the form $c\sqrt{t}+W_t$, where $c$ can be any constant) and the path measure $\pi$ is precisely concentrated on such paths. We observe that although the drift coefficient $b$ is highly singular at $t=0$, it poses no issue. This is the power of the superposition principle, which needs no regularity in the coefficients.

    \subsection{Comments} We end this section with a remark on the treatment of the stochastic setting.

    \begin{remark}[Going from deterministic to stochastic setting]\label{rmk:treatment}
    Although this paper relies on our work on the deterministic case \cite{Abedi2025paths}, we would like to emphasize that a separate treatment is needed in the stochastic setting to avoid measurability issues and loss of regularities.
    \\
    To illustrate the issues, suppose we aim to construct an optimal random lift for a measure-valued process $(\mu_t)$ satisfying $\mathbb{E} [W^p_p (\mu_s,\mu_t)] \leq c^p |t-s|^{\upgamma p }$ for some $1/p<\upgamma \leq 1$. For the moment, assume that all sample paths $t \mapsto \mu_t^\omega$ lie in the $p$-Wasserstein space and are compatible. At first, we may apply the Kolmogorov–\v Centsov continuity theorem on the metric space $(P_p(\mathcal{X}), W_p)$ to obtain an a.s. continuous modification $(\tilde{\mu}_t)$. The sample paths $t \mapsto \tilde{\mu}_t^\omega$ can be chosen to have H\"{o}lder continuity with an exponent less than $\upgamma - 1/p$. Now we can apply our corresponding deterministic H\"{o}lder result from \cite{Abedi2025paths} in a ``path-wise'' way, meaning that for each $(\tilde{\mu}^\omega_t)$, we obtain a lift $\pi^\omega$, depending on $\omega$.
    By this result,
    $\pi^\omega$ will then be concentrated on H\"{o}lder curves with an exponent less than $\upgamma - 2/p$. 
    In other words, we have lost regularity twice, so the stronger condition $2/p<\upgamma \leq 1$ must be in force from the beginning. The main issue, however, is that it is not clear whether the map $\omega \mapsto \pi^\omega$ is measurable, i.e., whether $\pi$ is actually a random measure.
    \\
    This necessitates a separate approach in the random setting, for which we are inspired by a stochastic superposition principle developed earlier by Flandoli \cite{Flandoli2009}. 
    Although the constructions differ, 
    the key steps are constructing a family of random path measures and then passing it to the limit. 
    This requires the notions of tightness and narrow convergence for random measures, along with a generalization of Prokhorov's theorem relating them. For this, we follow the books by Crauel \cite{Crauel2002} and H\"{a}usler--Luschgy \cite{HauslerLuschgy2015}.
    \end{remark}

    \subsection{Organization of the paper}
    The rest of the paper is structured as follows:
    \begin{itemize}
     \item \textbf{Section~\ref{sec:preliminaries_s} (Preliminaries).} We collect the required notions and results. In addition, the following two elementary results may be of independent interest: 
      \begin{itemize}[label=$\circ$]
      \item Proposition \ref{prop:narrow convergence_vs_Ewpp} and Remark \ref{rk:narrow convergence_vs_Ewpp}. The relationship between the topology of narrow convergence on the space of random measures $P_{\Omega}(\mathcal{X})$ and the one induced by $\mathbb{W}_p$.
      \item Corollary \ref{crl:tightness_Holder_discrete_random}. A discrete relaxation of the well-known Kolmogorov--Lamperti tightness condition for random path measures.
      \end{itemize} 
    \item \textbf{Section~\ref{sec:results_s} (Main results).} We present the main theorems in four subsections.
    \item \textbf{Section~\ref{sec:corollaries_s} (Corollaries).} We apply the theorem of the construction of a realizing lift to Euclidean spaces, both in deterministic and stochastic settings.
    For $\mathcal{X} = \mathbb{R}$, this is applicable as all probability measures with finite $p$-moments are compatible. 
    For higher dimensions $\mathcal{X} = \mathbb{R}^\mathrm{d}$, where this property can easily fail, we give an alternative assumption to compatibility by replacing the Wasserstein metric in the regularity assumptions with the so-called $\nu$-based Wasserstein metric. 
    \item \textbf{Appendix~\ref{sec:appendix_s}.} We give the proof of Corollary \ref{crl:minimizing_processes_sFPE_intro}. 
    \end{itemize}

    \subsection{Acknowledgements} 
    The author would like to thank his supervisor, Matthias Erbar, for invaluable guidance and numerous discussions.
    He also thanks Vitalii Konarovskyi for insightful discussions and for pointing out the work \cite{LackerShkolnikovZhang2023}. He expresses appreciation to Zhenhao Li and Timo Schultz for the helpful conversations. 

    \noindent
    \textbf{Funding declaration.}
    This work is supported by the Deutsche Forschungsgemeinschaft (DFG, German Research Foundation) -- Project-ID 317210226 -- SFB 1283.

\section{Preliminaries}\label{sec:preliminaries_s}
\subsection{Path spaces}
    In this paper, $(\mathcal{X},d)$ is a metric space (with any additional structure explicitly stated), and $[0,T] \subset \mathbb{R}$ is a time interval. We denote by $\Gamma_T \coloneqq C([0, T ]; \mathcal{X})$ the space of all continuous paths equipped with the supremum distance, whose induced topology is known as the supremum topology or uniform topology. When $T=1$, we write $\Gamma \coloneqq \Gamma_1$.
    In this section, we fix the notation for the path spaces used in this paper.
    
    \begin{definition}[{$\upgamma$-H\"{o}lder} continuity]\label{def:Holder_regularity_r_paper}
        Given $\upgamma \in [0,1]$, the $\upgamma$-H\"{o}lder continuity of a  path $X \in C ([0,T] ; \mathcal{X})$ over $[s,t] \subset [0,T]$ is defined by 
        \begin{equation}
            | X |_{\upgamma\textrm{-}\mathrm{H\ddot{o}l};[s,t]} \, \coloneqq \sup_{s \leq u < v \leq t} \frac{d(X_u,X_v)}{|v-u|^\upgamma}.
        \end{equation}
        $C^{\upgamma\textrm{-} \mathrm{H\ddot{o}l}} ([0,T];\mathcal{X})$ denotes the set of all paths
        $X \in C ([0,T] ; \mathcal{X})$ such that 
        \begin{equation}
            | X |_{\upgamma\textrm{-}\mathrm{H\ddot{o}l}}  \coloneqq | X |_{\upgamma\textrm{-}\mathrm{H\ddot{o}l};[0,T]} < \infty.
        \end{equation}
    \end{definition}
    
    Let $D \coloneqq \big\{ s = t_0 < t_1 < \cdots < t_N = t \big\}$ be a dissection of the time interval $[s,t]$, where $N \in \mathbb{N}$. The $p$-variation ($p\geq 1$) of a path $X:[0,T]\to \mathcal{X}$ over a fixed dissection $D$ is defined as
    \begin{equation}
        \sum_{t_i \in D} d(X_{t_{i} }, X_{t_{i+1}} )^p,
    \end{equation}
    with the convention $t_{N+1} = t_{N}$.
    We let $\mathcal{D}([s,t])$ denote the set of all partitions of $[s,t]$. 
    \begin{definition}[{$p$-variation}]\label{def:p-variation_r_paper}
         Given  $p\geq 1$, the $p$-variation of a path $X\in C([0,T]; \mathcal{X})$ over $[s,t] \subset [0,T]$ is defined by 
         \begin{equation}
             | X |_{p\textrm{-}\mathrm{var};[s,t]} \coloneqq \left( \sup_{t_i \in \mathcal{D}([s,t])} \sum_{i} d(X_{t_i}, X_{t_{i+1}})^p\right)^{1/p}. 
         \end{equation}
         $C^{p\textrm{-} \mathrm{var}} ([0,T];\mathcal{X})$ denotes the set of all continuous paths $X$ such that
        \begin{equation}
            | X |_{p\textrm{-}\mathrm{var}}  \coloneqq | X |_{p\textrm{-}\mathrm{var};[0,T]} < \infty.
        \end{equation}
    \end{definition}
    
    \begin{definition}[{Fractional  Sobolev regularity $W^{\alpha, p}$}]\label{def:Walphap_regularity}
        Given $1\leq p<\infty$ and $0<\alpha < 1$, the fractional Sobolev regularity of a measurable function $X: [0,T] \to \mathcal{X}$ over $[s,t] \subset [0,T]$ is defined by  
        \begin{equation}
            | X |_{W^{\alpha,p};[s,t]} \coloneqq \left( \iint_{[s,t]^2} \frac{d(X_u,X_v)^p}{|v-u|^{1+\alpha p}} \d u \d v \right)^{1/p}.
        \end{equation}
        The fractional Sobolev space $W^{\alpha, p}([0,T]; \mathcal{X})$ in this paper is the space of continuous functions $X$ such that
        $
            | X |_{W^{\alpha,p}} \coloneqq | X |_{W^{\alpha,p};[0,T]} < \infty.
        $
    \end{definition}
    
    \begin{definition}[{Besov regularity $b^{\alpha, p}$}]\label{def:balphap_regularity}
        Given $1\leq p<\infty$ and $0<\alpha < 1$, the Besov regularity of a function $X: [0,1] \to \mathcal{X}$ over $[0,1]$ is defined by
        \begin{equation}\label{eq:Walphap_sum_r_paper}
                | X |_{b^{\alpha,p}} \coloneqq \left( \sum_{m=0}^{\infty} 2^{m(\alpha p -1)} \sum_{k=0}^{2^m-1} d \big(X_{t_{k}^{(m)}}, X_{t_{k+1}^{(m)}}\big)^p \right)^{1/p},
        \end{equation}
        where $t_k^{(m)} \coloneqq \frac{k}{2^m}$.
        \\
        The Besov space $b^{\alpha,p}([0,1];\mathcal{X})$ is the space of continuous functions $X$ such that
        $
            | X |_{b^{\alpha,p}} < \infty.
        $
    \end{definition}

    In the definition of $b^{\alpha,p}([0,1];\mathcal{X})$, continuity is essential, since the dyadic Besov seminorm only sees the function on a countable subset of the interval.
    
    \begin{theorem}[{\cite[Theorem 2.2]{LiuPromelTeichmann2020}}]\label{thm:Walphap_balphap_r_paper}
            Given $1<p<\infty$ and $\frac{1}{p}<\alpha < 1$, we have
            $$
            W^{\alpha, p}([0,1]; \mathcal{X}) = b^{\alpha, p}([0,1]; \mathcal{X})
            $$
            as subsets of $C([0,1]; \mathcal{X})$.
            Furthermore, $| \cdot |_{W^{\alpha,p}}$ and $| \cdot |_{b^{\alpha,p}} $  are equivalent on the set of continuous paths, i.e., there exist positive constants  $c_1,c_2$ depending only on $(\alpha, p)$ such that
            \begin{equation}\label{eq:equiv_norm_Walphap_r_paper}
                 c_1 | X |_{W^{\alpha,p}} \leq | X |_{b^{\alpha,p}}   \leq c_2  | X |_{W^{\alpha,p}}
            \end{equation}
            for all $X \in C([0,1];\mathcal{X})$. 
    \end{theorem}
    
    The estimate \eqref{eq:distance_estimate_GRR_r_paper} below and the results thereafter can be derived by the \emph{Garsia--Rodemich--Rumsey inequality}. See \cite{Garsia_Rodemich_Rumsey1970}, \cite[Theorem A.1 and Corollary A.2-3]{FrizVictoir2010book}, \cite[Theorem 2]{FrizVictoir2006}. 
    
    \begin{theorem}[Fractional Sobolev-H\"{o}lder and -variation embeddings]\label{thm:FractionalSobolev-Holder_r_paper}  
        Given $1<p<\infty$ and $ \frac{1}{p} < \alpha < 1$, let $X \in W^{\alpha, p} ([0,T];\mathcal{X})$.
        Then there exists a constant $\bar{c}$ depending only on $(\alpha, p)$ such that for all $0 \leq s < t \leq T$,
        \begin{equation}\label{eq:distance_estimate_GRR_r_paper}
            d (X_s,X_t) \leq \bar{c} |t-s|^{\alpha - \frac{1}{p}} | X |_{W^{\alpha,p};[s,t]},
        \end{equation}
         and in particular,
        \begin{align}
            | X |_{\alpha - \frac{1}{p}\textrm{-}\mathrm{H\ddot{o}l};[s,t]}  & \leq \bar{c} | X |_{W^{\alpha,p};[s,t]}, \\
            | X |_{\frac{1}{\alpha}\textrm{-}\mathrm{var};[s,t]}
            & \leq \bar{c} |t-s|^{\alpha-\frac{1}{p}} | X |_{W^{\alpha,p};[s,t]}, 
        \end{align}
        where a possible choice of the constant is $\bar{c} = \big( 32 \frac{\alpha p +1}{\alpha p -1} \big)^{1/p}$.
        \\
        In particular, the continuous embeddings hold:        \begin{equation}\label{eq:FS_H_v_embeddings_r_paper}
            W^{\alpha, p} \subset C^{\alpha - \frac{1}{p}\textrm{-}\mathrm{H\ddot{o}l}} \quad \textrm{and} \quad W^{\alpha, p} \subset C^{\frac{1}{\alpha}\textrm{-}\mathrm{var}}. 
        \end{equation}
    \end{theorem}
    
    \begin{remark}[A (trivial) H\"{o}lder-Fractional Sobolev embedding]\label{rmk:Holder-FractionalSobolev_r_paper}
        Given $0 < \alpha < \upgamma \leq 1$, one can easily compute for any $X \in C^{\upgamma\textrm{-} \mathrm{H\ddot{o}l}}([0,T];\mathcal{X} )$ that
        \begin{equation}
            | X |_{W^{\alpha,p}}^p \leq | X |_{\upgamma\textrm{-}\mathrm{H\ddot{o}l}}^p \frac{2 T^{(\upgamma p - \alpha p+1)} }{(\upgamma p - \alpha p)(\upgamma p - \alpha p+1)}  < + \infty
        \end{equation}
         In particular, we have the continuous embedding
        \begin{equation}
        C^{\upgamma\textrm{-}\mathrm{H\ddot{o}l}} \subset W^{\alpha, p}.
        \end{equation}
    \end{remark}

\subsection{Random probability measures, narrow convergence, and tightness}\label{subsec:random_probability_measures}
    The ultimate goal of this section is to present a generalization of Prokhorov's theorem for random probability measures (see Theorem \ref{thm:Prokhorov_random}). We begin by recalling the definition of random probability measures and then introduce a notion of narrow convergence and tightness for random measures, which ultimately allows us to have the Prokhorov result in the random setting. We primarily follow the references \cite{Crauel2002,HauslerLuschgy2015}.
    
    Let $(\mathcal{X},d)$ be a complete separable metric space. We let $\mathcal{B}(\mathcal{X})$ be the $\sigma$-algebra of Borel sets of $\mathcal{X}$ (generated by open balls in $\mathcal{X}$) and we let $P(\mathcal{X})$ be the set of all Borel probability measures on $(\mathcal{X}, \mathcal{B}(\mathcal{X}))$.
    In addition, we consider $(\Omega,\mathcal{F},\mathbb{P})$ to be a probability space. The product space $\mathcal{X} \times \Omega $ is regarded as a measurable space with the product $\sigma$-algebra $\mathcal{B}(\mathcal{X})\otimes\mathcal{F}$.
     \begin{definition}[Random probability measure]\label{def:random_measure}
       A random probability measure $\mu$ on $\mathcal{X}$ is a map
       \begin{align}
        \mu: \mathcal{B}(\mathcal{X}) \times \Omega & \to \, [0,1] \\
        (A, \omega) \,  & \mapsto \,  \mu^{\omega}  (A)
       \end{align}
       such that 
       \begin{enumerate}[label=(\roman*), font=\normalfont]
           \item for any $A \in \mathcal{B}(\mathcal{X})$, the function $\omega \mapsto \mu^{\omega}  (A) $ is $\mathcal{F}$-$\mathcal{B}([0,1])$-measurable;
           \item for $\mathbb{P}$-a.e.
           $\omega \in \Omega$, the function $A \mapsto \mu^{\omega}  (A) $ is a Borel probability measure, i.e., $\mu^{\omega} \in P(\mathcal{X})$. 
       \end{enumerate}
       We use $P_{\Omega} (\mathcal{X})$ 
       to denote the set of all random probability measures on $\mathcal{X}$. 
    \end{definition}

    To any random probability measure $\omega \mapsto \mu^{\omega} \in P(\mathcal{X})$ and probability measure $\mathbb{P} \in P(\Omega) $, two measures are naturally associated.
    First, the formula
    \begin{equation}\label{eq:bold_mu_def}
         \bm{\mu} (\bm{A}) \coloneqq \int_{\Omega} \int_{\mathcal{X}} \mathds{1}_{\bm{A}} (x,\omega) \d \mu^{\omega} (x) \d \mathbb{P} (\omega), \qquad \forall \bm{A}\in \mathcal{B}(\mathcal{X})\otimes\mathcal{F},
    \end{equation}
    uniquely defines a measure $\bm{\mu} \in P(\mathcal{X} \times \Omega)$ on the product space. Second, in a similar way,
    \begin{equation}\label{eq:exp_mu_def}
            \mathbb{E}[\mu] (A) \coloneqq \int_{\Omega} \int_{\mathcal{X}} \mathds{1}_{A} (x) \d \mu^{\omega} (x) \d \mathbb{P} (\omega), \qquad \forall A\in \mathcal{B}(\mathcal{X}),
        \end{equation}
     also results in a probability measure $\mathbb{E}[\mu] \in P(\mathcal{X})$. Note that \eqref{eq:exp_mu_def} can be simply written as $\mathbb{E}[\mu] (A) \coloneqq \mathbb{E}[\mu(A)]$, which explains the choice of notation. 
     We will refer to this measure as the \emph{average measure}.
     For any measurable and  $\mathbb{E}[\mu]$-integrable function $\varphi : \mathcal{X} \to \mathbb{R}$, we have (see \cite[Lemma 3.22]{Crauel2002})
     \begin{equation}\label{eq:Exp_int_phi_mu}
        \mathbb{E} \left[\int \varphi (x) \d \mu (x) \right] = \int \varphi (x) \d \mathbb{E}[\mu] (x).
    \end{equation}

    \begin{definition}[Random continuous bounded  functions] \label{def:random_cb_functions}
        A random continuous bounded function is a mapping
        \begin{align}
            f: \mathcal{X} \times \Omega & \to \, \mathbb{R} \\
            (x, \omega)  & \mapsto \,  f(x, \omega)
       \end{align}
       such that
       \begin{enumerate}[label=(\roman*), font=\normalfont]
           \item for all $x \in \mathcal{X}$, the function $\omega \mapsto f (x,\omega)$ is $\mathcal{F}$-$\mathcal{B}(\mathbb{R})$-measurable;
           \item for all $\omega \in \Omega $, the function $x \mapsto f (x,\omega)$ is continuous and bounded, i.e.,  $f (\cdot, \omega) \in C_b(\mathcal{X})$;
           \item it satisfies
           \begin{equation}
               \lVert f \rVert_{\infty \times 1} \coloneqq \int \sup_{x \in \mathcal{X}} |f(x,\omega)| \d \mathbb{P} (\omega) < \infty. 
           \end{equation}
        \end{enumerate}
        We use $C_{b, \Omega} (\mathcal{X})$ to denote the space of all random continuous bounded functions, equipped with the norm $ \lVert \cdot \rVert_{\infty \times 1}$.
    \end{definition}

    The functions from $C_{b, \Omega} (\mathcal{X})$ generate a topology on $P_{\Omega} (\mathcal{X})$, which is called the narrow topology on the space of random probability measures, similar to the deterministic case. 

    \begin{definition}[Narrow convergence of random measures]\label{def:narrow_convergence_random}
        A sequence $(\mu_n) \subset P_{\Omega}(\mathcal{X})$ narrowly converges to $\mu \in P_{\Omega}(\mathcal{X})$  if
        \begin{equation}\label{eq:narrow_convergence_random_f}
            \lim_{n\to \infty} \int_{\Omega} \int_{\mathcal{X}} f (x,\omega) \d \mu_n^\omega (x) \d \mathbb{P} (\omega) = \int_{\Omega} \int_{\mathcal{X}} f (x,\omega) \d \mu^{\omega} (x) \d \mathbb{P} (\omega)
        \end{equation}
        for every $f \in C_{b,\Omega}(\mathcal{X})$. 
    \end{definition}

    Narrow convergence for random measures can be characterized (similar to the deterministic case) by a subset of $C_{b,\Omega}(\mathcal{X})$, namely, the space of Random Lipschitz bounded  functions defined as below (for a proof, see \cite[Proposition 4.9 and Corollary 4.10]{Crauel2002}) \begin{equation}\label{eq:random_Lip_functions}
            \mathrm{Lip}_{b,\Omega} (\mathcal{X}) \coloneqq \{ g \in C_{b,\Omega} (\mathcal{X}) : \textrm{for } \mathbb{P}\textrm{-a.e. } \omega, \, \mathrm{Lip} (g (\cdot,\omega)) \leq L \textrm{ for some } L \in \mathbb{R} \} .
        \end{equation}
    We also recall that if $(\mu_n) \subset P_\Omega(\mathcal{X})$ narrowly converges to $\mu \in P_\Omega(\mathcal{X})$ and $f: \mathcal{X} \times \Omega \to [0,+\infty]$ is a measurable function such that $f(\cdot, \omega)$ is lower semi-continuous in first variable for every $\omega \in \Omega$, then (see \cite[Theorem 2.6 (iv)]{HauslerLuschgy2015})
    \begin{equation}\label{eq:narrow_conv_lsc_func_random}
        \liminf_{n \to \infty } \int_{\Omega} \int_{\mathcal{X}} f(x,\omega) \d \mu^\omega_n (x) \d \mathbb{P} (\omega) \geq \int_{\Omega} \int_{\mathcal{X}} f(x,\omega) \d \mu^\omega (x) \d \mathbb{P} (\omega). 
    \end{equation}
    \begin{definition}[Tightness for random measures]\label{def:tightness_random}
        A family of random measures $ \mathcal{K} \subset P_\Omega(\mathcal{X})$ is said to be tight if for any $ \varepsilon>0$, there exists a (non-random) compact set $K_\varepsilon \subset \mathcal{X}$ such that 
        \begin{equation}\label{eq:tightness_random_exp}
            \sup_{\mu \in \mathcal{K}} \mathbb{E}[\mu] (K_\varepsilon^{\complement}) \leq \varepsilon.
        \end{equation}
    \end{definition}
    In other words, a family of random measures is tight in the sense above if the family of (non-random) average measures is tight in the classical sense.
    This means that, in practice, we only need to prove the classical tightness of the average measure, and no extra condition is needed.
    For example, by combining \eqref{eq:Exp_int_phi_mu} with the integral tightness criterion for deterministic measures (see \cite[Remark 5.1.5]{AGS2008GFs}), we immediately obtain an integral tightness criterion for random measures:
    
    \begin{lemma}[A tightness criterion for random measures]\label{lemma:tightness_criterion_random}
     A family of random measures $ \mathcal{K} \subset P_{\Omega}(\mathcal{X})$ is tight if and only if there exists a function 
     $\Psi: \mathcal{X} \to [0, + \infty]$ such that
    \begin{enumerate}[font=\normalfont]
        \item its sublevels $\lambda_c (\Psi) \coloneqq \{ |\Psi| \leq c \} \subset \mathcal{X}$ are compact for any $c \geq 0$;
        \item it satisfies the bound
        \begin{equation}\label{eq:tightness_e_criterion_bound}
            \sup_{\mu \in \mathcal{K}} \mathbb{E} \left[  \int_{\mathcal{X}}  \Psi(x) \d \mu (x) \right] < + \infty.
        \end{equation}
    \end{enumerate}
    \end{lemma}

    Lastly, we give Prokhorov's theorem for random measures. See \cite[Theorem 4.29]{Crauel2002} for a proof.
  
    \begin{theorem}[Prokhorov for random measures]\label{thm:Prokhorov_random}
        A family of random measures $\mathcal{K} \subset P_{\Omega} (\mathcal{X})$ is tight if and only if it is relatively compact with respect to the narrow topology of $P_{\Omega}(\mathcal{X})$. 
    \end{theorem} 

\subsection{Tightness conditions for random path measures}

Combining the observation in Lemma \ref{lemma:tightness_criterion_random} with the four tightness conditions for deterministic path measures formulated in \cite[Section 2.10]{Abedi2025paths}, we immediately obtain the following conditions for random path measures. In this paper, we only use Corollary \ref{prop:tightness_balphap_random}.
The last condition, Corollary \ref{crl:tightness_Holder_discrete_random}, can be regarded as a discrete relaxation of the well-known \emph{Kolmogorov--Lamperti} tightness condition generalized to random path measures.  

    \begin{corollary}[A tightness condition via $W^{\alpha,p}$]\label{prop:tightness_Walphap_random}
         Let $(\mathcal{X},d)$ be a complete metric space in which closed bounded sets are compact, and let $(\Omega, \mathcal{F}, \mathbb{P})$ be a probability space.
         Let the family of random measures $ \mathcal{K} \subset P_{\Omega}(C([0,T];\mathcal{X}))$ satisfy
         \begin{equation}
             \sup_{\pi \in \mathcal{K}} \mathbb{E} \left[  \int_{\Gamma_T} \Big(  d(\gamma_0, \bar{x}) +  | \gamma |_{W^{\alpha,p}} \Big) \d \pi (\gamma) \right] < +\infty
         \end{equation}
         for some $1<p<\infty$ and $\frac{1}{p}<\alpha < 1$ and $\bar{x} \in \mathcal{X}$. Then $\mathcal{K}$ is tight in $P_\Omega(C([0,T];\mathcal{X}))$.  
    \end{corollary}

    \begin{corollary}[A tightness condition via $b^{\alpha,p}$]\label{prop:tightness_balphap_random}
         Let $(\mathcal{X},d)$ be a complete metric space in which closed bounded sets are compact, and let $(\Omega, \mathcal{F}, \mathbb{P})$ be a probability space.
         Let the family of random measures $ \mathcal{K} \subset P_{\Omega}(C([0,1];\mathcal{X}))$ satisfy
         \begin{equation}
             \sup_{\pi \in \mathcal{K}} \mathbb{E} \left[ \int_{\Gamma_1} \Big(  d(\gamma_0, \bar{x}) +  | \gamma |_{b^{\alpha,p}} \Big) \d \pi (\gamma) \right] < +\infty
         \end{equation}
         for some $1<p<\infty$ and $\frac{1}{p}<\alpha < 1$ and $\bar{x} \in \mathcal{X}$. Then $\mathcal{K}$ is tight in $P_\Omega(C([0,1];\mathcal{X}))$.  
    \end{corollary}

    \begin{corollary}[A tightness condition via H\"{o}lder regularity]
        Let $(\mathcal{X},d)$ be a complete metric space in which closed bounded sets are compact, and let $(\Omega, \mathcal{F}, \mathbb{P})$ be a probability space.
        Let the family of random measures $ \mathcal{K} \subset P_{\Omega}(C([0,T];\mathcal{X}))$ satisfy
         \begin{equation}
        \sup_{\pi \in \mathcal{K}} \mathbb{E} \left[ \int_{\Gamma_T}  d(\gamma_0, \bar{x}) \d \pi (\gamma) \right] < +\infty
        \end{equation}
        for some $\bar{x} \in \mathcal{X}$ and 
        \begin{equation}\label{eq:tightness_Holder_random}
            \sup_{\pi \in \mathcal{K}} \mathbb{E} \left[ \int_{\Gamma_T} d(\gamma_t, \gamma_s )^p \d \pi (\gamma) \right]  \leq c |t-s|^{p \upgamma} \quad\qquad \forall t,s \in [0,T],
        \end{equation}
        for some constant $c$ and $1<p<\infty$ and $\frac{1}{p}<\upgamma \leq 1$. Then $\mathcal{K}$ is tight in $P_{\Omega}(C([0,T];\mathcal{X}))$.
    \end{corollary}

    \begin{corollary}[A tightness condition via H\"{o}lder regularity on a countable set]\label{crl:tightness_Holder_discrete_random}
    Let $(\mathcal{X},d)$ be a complete metric space in which closed bounded sets are compact, and let $(\Omega, \mathcal{F}, \mathbb{P})$ be a probability space.
        Let the family of random measures $ \mathcal{K} \subset P_{\Omega}(C([0,T];\mathcal{X}))$ satisfy
         \begin{equation}
        \sup_{\pi \in \mathcal{K}} \mathbb{E} \left[ \int_{\Gamma_T}  d(\gamma_0, \bar{x}) \d \pi (\gamma) \right] < +\infty
        \end{equation}
        for some $\bar{x} \in \mathcal{X}$ and 
        \begin{equation}\label{eq:tightness_Holder_discrete_random}
         \sup_{\pi \in \mathcal{K}} \mathbb{E} \left[\int_{\Gamma_1}  d (\gamma_{t_{k}^{(m)}},\gamma_{t_{k+1}^{(m)}})^p \d \pi (\gamma) \right] \leq \tilde{c} |\Delta t_m |^{p\upgamma} \quad\qquad \forall m \in \mathbb{N}_0, \, k \in \{0,1,\cdots, 2^m-1\},
    \end{equation}
    for some constant $\tilde{c}$ and $1<p<\infty$ and $\frac{1}{p}<\upgamma \leq 1$, where $t_k^{(m)} \coloneqq \frac{k}{2^m} $ and $\Delta t_m \coloneqq \frac{1}{2^m}$. Then $\mathcal{K}$ is tight in $P_{\Omega}(C([0,T];\mathcal{X}))$.
    \end{corollary}

    \subsection{Probability measures with finite \texorpdfstring{$p$}{p}-moment}
    Let $(\mathcal{X},d)$ be a complete separable metric space. Given $p \in [1, \infty) $, $P_p(\mathcal{X}) \subset P(\mathcal{X})$ denotes the subset of Borel probability measures with finite $p$-th moment:
        \begin{equation}\label{eq:def:P_p(X)}
            P_p (\mathcal{X}) \coloneqq \left\{ \mu \in P (\mathcal{X}) : \, \int_{\mathcal{X}} d(x, \bar{x})^p \d \mu (x) < + \infty \right\}, 
    \end{equation}
    where $\bar{x} \in \mathcal{X}$ is an arbitarty point.
    \\
    We equip the space $P_p (\mathcal{X})$ with the $p$-(Kantorovitch–Rubinstein–)Wasserstein distance defined as below between any $\mu,\nu \in P_p(\mathcal{X})$, 
    \begin{equation}\label{eq:def:Wp}
        W_p(\mu,\nu) \coloneqq \left( \min_{\Upsilon \in \mathrm{Cpl}(\mu,\nu)} \int_{\mathcal{X}\times\mathcal{X}} d(x,y)^p \d \Upsilon (x,y) \right)^{1/p},
    \end{equation}
    where the minimum runs over all couplings between $\mu$ and $\nu$, denoted by
    \begin{equation}
        \mathrm{Cpl}(\mu,\nu) \coloneqq \Big\{ \Upsilon \in P(\mathcal{X}^2) : \quad (\mathrm{Pr}^{1})_{\#} \Upsilon  = \mu, \,(\mathrm{Pr}^{2})_{\#} \Upsilon  = \nu \Big\}.
    \end{equation}
    The set of optimal couplings for $W_p$ will be denoted by $\mathrm{OptCpl}(\mu,\nu)$.

    \subsection{Random probability measure with finite \texorpdfstring{$p$}{p}-moments in expectation}\label{subsec:random_probability_measures_finite_moment}
    We define a subset $P_{p,\Omega} (\mathcal{X})\subset P_{\Omega} (\mathcal{X})$ of random Borel probability measures whose $p$-th moments are finite in expectation:
    \begin{definition}[Space $P_{p,\Omega}(\mathcal{X})$]\label{def:P_p_Omega(X)}
         Let $(\mathcal{X},d)$ be a complete separable metric space and $(\Omega, \mathcal{F}, \mathbb{P})$ be a probability space. Given $p \in [1,\infty)$, we define
        \begin{equation}\label{eq:P_p_Omega(X)_def}
            P_{p,\Omega} (\mathcal{X}) \coloneqq \left\{ \mu \in P_{\Omega} (\mathcal{X}) : \, \mathbb{E} \left[  \int_{\mathcal{X}} d(x, \bar{x})^p \d \mu (x) \right] < + \infty \right\}, 
        \end{equation}
        where $\bar{x} \in \mathcal{X}$ is an arbitarty (non-random) point.
    \end{definition}
    The defining condition above can be formulated in terms of measures $\bm{\mu}$ or $\mathbb{E} \mu $, defined in \eqref{eq:bold_mu_def}-\eqref{eq:exp_mu_def}. By disintegration and \eqref{eq:Exp_int_phi_mu}, we have 
       \begin{align}
        \mathbb{E} \left[  \int_{\mathcal{X}} d(x, \bar{x})^p \d \mu (x) \right]  & = \int_{\Omega}  \int_{\mathcal{X}} d(x, \bar{x})^p \d \mu^\omega (x) \d \mathbb{P} (\omega) \\
        & = \int_{\mathcal{X}\times \Omega } d(x, \bar{x})^p \d \bm{\mu} (x,\omega)= \int_{\mathcal{X}}  d(x, \bar{x})^p \d \mathbb{E} \mu (x). \label{eq:E_p_moment_bold}
    \end{align}
    Thus, $\mu \in P_{p,\Omega}(\mathcal{X})$ if and only $\mathbb{E}\mu \in P_{p}(\mathcal{X})$. In other words, a random measure is in $P_{p,\Omega}(\mathcal{X})$ if and only if its average measure has a finite $p$-moment.

    We equip the space $P_{p,\Omega} (\mathcal{X})$ with a distance akin to the $p$-Wasserstein distance in the deterministic setting \eqref{eq:def:Wp}.
    \begin{definition}[Distance $\mathbb{W}_p (\cdot,\cdot)$]
    \label{def:mathbbW_p}
    Given any two random measures $\mu, \nu \in P_{p,\Omega}(\mathcal{X})$, we set
    \begin{equation}\label{eq:def:mathbbW_p_inf}
        \mathbb{W}_p(\mu,\nu) \coloneqq \left( \inf_{\Upsilon \in \mathrm{Cpl}_{\Omega}(\mu,\nu)} \mathbb{E} \left[ \int_{\mathcal{X}\times\mathcal{X}} d(x,y)^p \d \Upsilon (x,y) \right] \right)^{1/p},
    \end{equation}
    where the infimum is taken over the set of random couplings of $(\mu,\nu)$ defined by 
    \begin{equation}
        \mathrm{Cpl}_\Omega(\mu,\nu) \coloneqq \Big\{ \Upsilon \in P_\Omega(\mathcal{X}^2) : \quad \Upsilon^\omega \in \mathrm{Cpl}(\mu^\omega,\nu^\omega) \,\, \mathbb{P}\text{-a.s.} \Big\}.
    \end{equation}
    \end{definition}

    We note that $\mu,\nu \in P_{p,\Omega}(\mathcal{X})$ implies $ \mu^\omega,\nu^\omega \in P_{p}(\mathcal{X})$ a.s.
    Lemma below ensures that there exists a measurable way to select optimal couplings between these random measures. Thus, we have the existence of a random measure that is optimal a.s. with respect to $W_p$.

    \begin{lemma}[{Measurable selection of optimal couplings \cite[Corollary 5.22]{Villani2009}}]\label{lemma:measurable_selection_opt_cpl_two}
        Let $\mu,\nu \in P_{p,\Omega}(\mathcal{X})$ with $p\geq 1$. Then there exists a random measure $\Upsilon \in P_{\Omega}(\mathcal{X}^2)$ such that for $\mathbb{P}$-a.e. $\omega$, we have 
        $\Upsilon^\omega \in \mathrm{OptCpl} (\mu^\omega, \nu^\omega).$
    \end{lemma}

    Accordingly, given $\mu,\nu \in P_{p,\Omega}(\mathcal{X})$, we write 
    \begin{equation}
        \mathrm{OptCpl}_\Omega(\mu,\nu) \coloneqq \Big\{ \Upsilon \in \mathrm{Cpl}_\Omega(\mu,\nu) : \quad  \Upsilon^\omega \in \mathrm{OptCpl} (\mu^\omega, \nu^\omega) \,\, \mathbb{P}\text{-a.s.} \Big\}.
    \end{equation}
    As a result of the lemma above, the infimum in \eqref{eq:def:mathbbW_p_inf} is actually attained as a minimum, and
    \begin{equation}\label{eq:mathbbW_p_Exp}
        \mathbb{W}_p (\mu,\nu) = \Big( \mathbb{E} \big[ W^p_p(\mu,\nu) \big] \Big)^{1/p}.  
    \end{equation}
    Note that $\omega \mapsto W_p (\mu^\omega, \nu^\omega)$ is measurable (see e.g. \cite[Remark 3.20 (ii)]{Crauel2002}) and thus the expectation is well-defined. 
    Using $\Upsilon \in \mathrm{OptCpl}_{\Omega}(\mu,\nu)$, we can write, similar to \eqref{eq:E_p_moment_bold}:
    \begin{align}
         \mathbb{W}^p_p (\mu,\nu) & =  \int_{\Omega} \int_{\mathcal{X}^2} d(x,y)^p \d \Upsilon^\omega (x,y) \d \mathbb{P} (\omega) \\
         & = \int_{\mathcal{X}^2 \times \Omega }  d(x,y)^p \d \bm{\Upsilon}(x,y,\omega) = \int_{\mathcal{X}^2} d(x,y)^p \d \mathbb{E} \Upsilon (x,y). \label{eq:EWpp_bold_Upsilon}
    \end{align}
    Let us comment further on the definitions above. First, in restriction to deterministic measures, the space $(P_{p,\Omega} (\mathcal{X}),\mathbb{W}_p)$ obviously reduces to $(P_p(\mathcal{X}),W_p)$, the space of probability measures with finite $p$-moment endowed with $p$-Wasserstein metric.
    On the other hand, if we restrict our attention to the set of random measures of the form $\omega \mapsto \mu^\omega = \delta_{X{(\omega)}} $, where $X$ is a random variable taking values in $\mathcal{X}$, then the space $P_{p,\Omega} (\mathcal{X})$ reduces to the space of $p$-integrable random variables
    %$$ L^p (\mathbb{P}) : = \left\{ X: \Omega \to \mathbb{R}^d \textrm{  measurable and } \mathbb{E} |X|^p < + \infty \right\}, $$
    and $\mathbb{W}_p$ reduces to the $L^p$ distance between them.
    When $\mathcal{X}$ is compact (so that all measures have finite moments), we can view random measures as random variables taking values in the metric space $(P_p(\mathcal{X}),W_p)$ and hence $\mathbb{W}_p$ is nothing but the $L^p$-distance between them. 
    In general, $\mathbb{W}_p$ defines a metric on $P_{p,\Omega} (\mathcal{X})$. This can be confirmed via \eqref{eq:mathbbW_p_Exp} or \eqref{eq:EWpp_bold_Upsilon}. We do the latter. To this end, let us first note that one can easily generalize  Lemma \ref{lemma:measurable_selection_opt_cpl_two}  to more than two measures by repeatedly applying the so-called gluing lemma.

    \begin{lemma}[Measurable selection of optimal multi-couplings]\label{lemma:measurable_selection_opt_cpl_2}
         Let $\mu_i \in P_{p,\Omega}(\mathcal{X})$, $1 \leq i \leq N$, with $N \in \mathbb{N}$ fixed and $p\geq 1$. Then there exists a random measure $ \Upsilon \in P_{\Omega} (\mathcal{X}^N)$ such that for $\mathbb{P}$-a.e. $\omega$, 
        \begin{gather}
            (\mathrm{Pr}^{i,i+1})_{\#} \Upsilon^{\omega} \in \mathrm{OptCpl}(\mu^\omega_{i},\mu^\omega_{i+1}), \quad \forall i\in \left\{1, \cdots, N -1 \right\}. 
        \end{gather}
    \end{lemma}

    \begin{remark}[$\mathbb{P}$-zero sets and completion of $\mathcal{F}$]\label{rmk:zero_set}
        We would like to emphasize that throughout this paper, how the random measure is defined on the zero sets plays no role. 
        This is already the case in how they are defined in  Definition \ref{def:random_measure} (i). 
        To clarify this, let us consider Lemma \ref{lemma:measurable_selection_opt_cpl_two}. 
        By assumption, there is a $\mathbb{P}$-full measure set $\Omega_\circ \subset \Omega$ on which $\mu^\omega,\nu^\omega \in P_p(\mathcal{X})$. 
        For all $\omega \in \Omega_\circ$, there is a measurable way of selecting optimal couplings between $\mu^\omega$ and $\nu^\omega$ by \cite[Corollary 5.22]{Villani2009} (note that $\omega \mapsto (\mu^\omega, \nu^\omega)$ is indeed measurable because both $\omega \mapsto \nu^\omega$ and $\omega \mapsto \mu^\omega$ are measurable).
         For any $\omega \in \Omega \setminus \Omega_\circ$, we let  $\Upsilon^\omega$ be an arbitrary coupling (take e.g. the product measure, which ensures that the set of couplings is always nonempty). 
         Then for all $\omega \in \Omega$, $ \Upsilon^\omega$ is a probability measure on $\mathcal{X}^2$. 
         To prove that $\Upsilon$ is a random measure, it remains to show that $\omega \mapsto  \Upsilon^\omega$ is measurable. 
         For any $A \in \mathcal{B}(\mathcal{X}^2)$ and $B \in \mathcal{B}([0,1])$, we have
        \begin{align}
            \{ \omega \in \Omega : \Upsilon^\omega (A) \in B \} = 
            \underbrace{\{ \omega \in \Omega_\circ : \Upsilon^\omega (A) \in B \}}_{\qquad \in \mathcal{F} \subset \mathcal{F}^*} \cup
            \underbrace{ \{ \omega \in \Omega \setminus \Omega_\circ : \Upsilon^\omega (A) \in B \} }_{\,\, \in \mathcal{F}^*},
        \end{align} 
        while this set is in the completion $\mathcal{F}^*$ of  $\mathcal{F}$, we do not need to do the completion of the $\sigma$-algebra (and we will not). As mentioned in \cite[Lemma 1.2]{Crauel2002}, for random variables taking values in a separable metric space that are measurable with respect to the completion, one can always find another random variable measurable with respect to the original $\sigma$-algebra such that the two match on a full measure set. Here, our random variables take values in $p$-Wasserstein space, which is indeed separable. 
    \end{remark}

    Proposition below confirms that $(P_{p,\Omega}(\mathcal{X}),\mathbb{W}_p)$ is a metric space. We consider two random measures $\mu$ and $\nu$ to be equivalent if $\mu^\omega = \nu^\omega$ almost surely. 
    \begin{proposition}\label{prop:mathbbWp_metric_result}
        Let $\mu_1, \mu_2, \mu_3 \in P_{p,\Omega}(\mathcal{X})$. Then
        \begin{enumerate}[label=(\roman*), font=\normalfont]
            \item $\mathbb{W}_p (\mu_1,\mu_2) = 0 \, \Leftrightarrow \, \mu_1=\mu_2$ $($i.e., $\mu_1^\omega  = \mu_2^\omega $ $\mathbb{P}$-a.s.$)$. 
            \item $\mathbb{W}_p (\mu_1,\mu_2) = \mathbb{W}_p (\mu_2,\mu_1)$.
            \item $\mathbb{W}_p(\mu_1, \mu_3) \leq \mathbb{W}_p(\mu_1, \mu_2) + \mathbb{W}_p(\mu_2, \mu_3)$.
        \end{enumerate}
    \end{proposition}
    \begin{proof}
        (i) is straightforward, and (ii) is obvious from \eqref{eq:mathbbW_p_Exp} since $W_p$ is a distance. To show (iii), one approach is to use \eqref{eq:mathbbW_p_Exp} and apply the triangle inequality for $W_p$ and Minkowski's inequality for $L^p(\mathbb{P})$ functions. An alternative approach, which we follow here, is to apply Lemma \ref{lemma:measurable_selection_opt_cpl_2} to get an explicit random multi-coupling $\Upsilon \in P_{\Omega } (\mathcal{X}^3)$ such that for $\mathbb{P}$-a.e $\omega \in \Omega$,
        $$
        (\mathrm{Pr}^{1,2})_{\#} \Upsilon^{\omega} \in \mathrm{OptCpl} (\mu_1^{\omega},\mu_2^{\omega}), \quad (\mathrm{Pr}^{2,3})_{\#} \Upsilon^{\omega} \in \mathrm{OptCpl} (\mu_2^{\omega},\mu_3^{\omega}).
        $$
        Using this measure, we can write similar to \eqref{eq:EWpp_bold_Upsilon}:
        \begin{align}
            \mathbb{W}_p (\mu_1,\mu_3)
            & = \left( \int_{\Omega} W_p^p(\mu_1^\omega,\mu_3^\omega) \d \mathbb{P} (\omega) \right)^{1/p} \\
            & \leq \left( \int_{\Omega} \int_{\mathcal{X}^3} d(x_1,x_3)^p \d \Upsilon^\omega (x_1,x_2,x_3) \d \mathbb{P} (\omega) \right)^{1/p} \\
            & = \left( \int_{\mathcal{X}^3 \times \Omega} d(x_1,x_3)^p \d \bm{\Upsilon} (x_1,x_2,x_3,\omega) \right)^{1/p} \\
            & \leq \left( \int_{\mathcal{X}^3 \times \Omega} d(x_1,x_2)^p \d \bm{\Upsilon} \right)^{1/p} \hspace{-6pt} +  \left( \int_{\mathcal{X}^3 \times \Omega} d(x_2,x_3)^p \d \bm{\Upsilon} \right)^{1/p} \hspace{-6pt} = \mathbb{W}_p (\mu_1,\mu_2) + \mathbb{W}_p (\mu_2,\mu_3),
        \end{align}
        where we used the triangle inequality for $d$ and Minkowski's inequality for $L^p(\bm{\Upsilon})$ functions. 
    \end{proof}

    Let's proceed with some preliminary results. The first simple observation is that the expectation of any random coupling between two random measures produces a coupling between their average measures:

    \begin{lemma}\label{lemma:Ecoupling}
        Let $\mu, \nu \in P_{\Omega} (\mathcal{X})$ and $\Upsilon \in \mathrm{Cpl}_{\Omega}(\mu, \nu)$. Then $\mathbb{E}\Upsilon \in \mathrm{Cpl} (\mathbb{E} \mu, \mathbb{E} \nu)$.
    \end{lemma}

    \begin{proof}
        Let $\varphi \in C_b (\mathcal{X})$. Then by \eqref{eq:Exp_int_phi_mu}, we have
        \begin{align}
            \int_{\mathcal{X}^2} \varphi (x) \d \mathbb{E} \Upsilon (x,y)
            & = \mathbb{E} \left[ \int_{\mathcal{X}^2} \varphi (x) \d \Upsilon (x,y) \right] \\
            & = \mathbb{E} \left[ \int_{\mathcal{X}} \varphi (x) \d \mu (x) \right] = \int_{\mathcal{X}} \varphi (x) \d \mathbb{E}  \mu (x)
        \end{align}
        and the same holds for the second marginal.
    \end{proof}
    
     We note that if $\mu,\nu \in P_{p, \Omega} (\mathcal{X})$ and $\Upsilon \in \mathrm{OptCpl}_{\Omega}(\mu, \nu)$, then this does \emph{not} necessarily imply that $\mathbb{E}\Upsilon \in \mathrm{OptCpl} (\mathbb{E} \mu, \mathbb{E} \nu)$ (see Example \ref{exp:average_measure_example} below). In general, we have the following relation:

    \begin{corollary}
        Let $\mu,\nu \in P_{p,\Omega} (\mathcal{X})$. Then 
        \begin{align}\label{eq:WpEmuEnu_Wpmunu}
            W_p (\mathbb{E} \mu, \mathbb{E} \nu) \leq \mathbb{W}_p (\mu, \nu ) .
        \end{align}
    \end{corollary}

    \begin{proof}
        Let $\Upsilon \in \mathrm{OptCpl}_{\Omega}(\mu, \nu)$. By Lemma \ref{lemma:Ecoupling}, $\mathbb{E}\Upsilon$ produces a coupling between $\mathbb{E}\mu$ and $\mathbb{E}\nu$, both lie in $p$-Wasserstein space by the assumption. As a result, we can estimate 
        \begin{align}
            W^p_p (\mathbb{E} \mu, \mathbb{E} \nu)
            \leq \int d(x,y)^p \d \mathbb{E} \Upsilon (x,y) 
            = \mathbb{E} \left[ \int d(x,y)^p \d  \Upsilon (x,y) \right]  = \mathbb{E} \left[ W_p^p (\mu,\nu) \right] = \mathbb{W}_p^p(\mu,\nu).
        \end{align}
    \end{proof}

    \begin{example}\label{exp:average_measure_example}
        Let $X \sim \mathrm{Unif} [-1,1]$ be a uniformly distributed random variable defined on a probability space $(\Omega,\mathcal{F},\mathbb{P})$. On this probability space, we define two random measures on $\mathbb{R}$ by
    	\begin{equation}
    		\mu^\omega \coloneqq \delta_{X(\omega)}, \qquad  \nu^\omega \coloneqq \delta_{-X(\omega)}. 
    	\end{equation} 
    	Then $\mathbb{E} \mu = \mathbb{E} \nu = \mathrm{Unif} [-1,1]$, and therefore, $W_p(\mathbb{E} \mu , \mathbb{E} \nu) = 0$. The random optimal coupling $\Upsilon \in \mathrm{OptCpl}_{\Omega}(\mu, \nu)$ is given by 
    	$		\Upsilon^\omega \coloneqq \delta_{(X(\omega),-X(\omega))}$.
    	However, we have $\mathbb{E}\Upsilon \notin \mathrm{OptCpl} (\mathbb{E} \mu, \mathbb{E} \nu)$. Indeed, 
    	\begin{align}
    		\int_{\mathbb{R}^2} |x-y|^p \d \mathbb{E}\Upsilon (x,y) = \mathbb{E} \left[ \int_{\mathbb{R}^2} |x-y|^p \d \Upsilon (x,y) \right] =  \mathbb{E} \left[ |2 X|^p \right] >0.
    	\end{align}
    \end{example}

    We now discuss the relationship between the topology of narrow convergence on $P_{\Omega} (\mathcal{X})$ and the topology induced by $\mathbb{W}_p$. As the next proposition shows, it's easy to prove that the topology induced by $\mathbb{W}_p$ is stronger than the narrow topology. To study the converse, it is crucial to note that, unlike the deterministic case, narrow convergence on $P_{\Omega} (\mathcal{X})$ is not metrizable in general. According to \cite[Corollary 4.31]{Crauel2002}, the narrow topology on $P_{\Omega}(\mathcal{X})$ for a Polish space $\mathcal{X}$ with at least two points is metrizable if and only if the $\sigma$-algebra $\mathcal{F}$ is \emph{countably generated} $(\mathrm{mod}\,\mathbb{P})$. However, even under this condition, the topology induced by $\mathbb{W}_p$ can still be strictly stronger than narrow convergence, as discussed in the next remark.
    
    \begin{proposition}\label{prop:narrow convergence_vs_Ewpp}
         Let $(\mathcal{X}, d)$ be a complete separable metric space and $(\Omega, \mathcal{F}, \mathbb{P})$ be a probability space. Take a sequence of random probability measures $(\mu_n) \subset P_{p,\Omega}(\mathcal{X})$ and $\mu \in P_{p,\Omega}(\mathcal{X})$  with $1 \leq p <\infty$. Consider two conditions:
        \begin{enumerate}[label=(\roman*), font=\normalfont]
           \item $\mathbb{W}_p(\mu_n,\mu) \underset{n \to \infty }{\longrightarrow} 0$.
           \item $\mu_n \underset{n \to \infty }{\longrightarrow} \mu$ in $P_\Omega (\mathcal{X})$ and $\mathbb{E}[\int d(x,\bar{x})^p \d \mu_n] \underset{n \to \infty }{\longrightarrow} \mathbb{E}[\int d(x,\bar{x})^p \d \mu]$ for some $\bar{x} \in \mathcal{X}$.
        \end{enumerate}
        Then (i) implies (ii).
        \end{proposition}
        \begin{remark}
        \label{rk:narrow convergence_vs_Ewpp}
             In general, (ii) $\Rightarrow$ (i) does \textit{not} hold, even if the $\sigma$-algebra $\mathcal{F}$ is countably generated $(\mathrm{mod}\,\mathbb{P})$, which is necessary for the narrow topology on $P_{\Omega}(\mathcal{X})$ to be metrizable.
             By an application of \cite[Lemma 5.2]{Crauel2002}, it turns out that when $d$ is a bounded metric on $\mathcal{X}$, the topology induced by any $\mathbb{W}_p$, $1 \leq p < \infty$, coincides with that induced by the Ky Fan metric on $P_{\Omega}(\mathcal{X})$. As a matter of fact, the Ky Fan topology doesn't depend on the choice of
             the metric \cite[Lemma 5.3]{Crauel2002} and is also stronger than the narrow topology \cite[Proposition 5.4]{Crauel2002}. In \cite[Example 5.7]{Crauel2002} a case is presented where $d$ is bounded, $\mathcal{F}$ is countably generated, and a sequence of random measures converges narrowly but doesn't converge in the Ky Fan metric, and thus not in $\mathbb{W}_p$ either.
        \end{remark}

    \begin{proof} ``(i) $\Rightarrow$ (ii)'' Using the the reverse triangle inequality, we have
    $$
    \Big | \mathbb{W}_p (\mu_n,\delta_{\bar{x}}) - \mathbb{W}_p (\mu,\delta_{\bar{x}}) \Big | \leq \mathbb{W}_p (\mu_n,\mu).
    $$
    Taking the limit yields the convergence of the expected values of moments.
    \\
    Now, let us verify that $(\mu_n)$ narrowly converges to $\mu$. We recall that narrow convergence in $P_{\Omega}(\mathcal{X})$ can also be characterized by random Lipschitz bounded functions.
    Take an arbitrary $g \in \mathrm{Lip}_{b,\Omega} (\mathcal{X})$ with $L$ being the bound on its Lipschitz constant as in \eqref{eq:random_Lip_functions}. For each $n$, let $\Upsilon_n \in \mathrm{OptCpl}_{\Omega} (\mu, \mu_n)$.
    We have  
         \begin{align}
        \bigg| \int_{\Omega} \int_{\mathcal{X}} g(x,\omega) \d \mu^\omega (x) \d \mathbb{P} (\omega) & - \int_{\Omega} \int_{\mathcal{X}} g(x,\omega) \d \mu_n^\omega (x) \d \mathbb{P} (\omega) \bigg| \\
        & = \left| \int_\Omega \int_{\mathcal{X}^2} \left( g (x,\omega) - g (y, \omega) \right) \d \Upsilon_n^\omega (x,y) \d \mathbb{P} (\omega) \right| \\
        & \leq  \int_\Omega \mathrm{Lip} (g (\cdot, \omega) ) \int_{\mathcal{X}^2} d (x,y) \d \Upsilon_n^\omega (x,y)  \d \mathbb{P} (\omega) \\
        & \leq  L \int_\Omega  \left( \int_{\mathcal{X}^2} d (x,y)^p \d \Upsilon_n^\omega (x,y) \right)^{1/p} \d \mathbb{P} (\omega) \\
        & = L \, \mathbb{E} \big[ W_p (\mu_n, \mu )\big] \leq L \, \mathbb{W}_p(\mu_n, \mu ).
    \end{align}
    Taking the limit, we obtain the result.
    \end{proof}

    As the final topic in this section, we generalize the notion of compatibility for deterministic measures in $P_p(\mathcal{X})$ to the setting of random measures. This notion was studied in our recent work \cite{Abedi2025paths}, following \cite{Boissard2015,PanaretosZemel2020}.

    \begin{definition}[Compatibility of random measures in $P_{p,\Omega}(\mathcal{X})$]\label{def:compatibility_random}
    We say a collection of random measures $\mathcal{M} \subset P_{p,\Omega}(\mathcal{X})$ is compatible in $P_{p,\Omega}(\mathcal{X})$, if, for every finite subcollection of $\mathcal{M}$, there exists a random multi-coupling such that all of its two-dimensional marginals are optimal $\mathbb{P}$-a.s.
    \end{definition}
    
     More precisely, this property tells that for any $N \in \mathbb{N}$ and any subcollection $\{ \mu_i \}_{1 \leq i \leq N} \subset \mathcal{M} \subset P_{p,\Omega}(\mathcal{X})$, there exists a random measure $ \Upsilon \in P_{\Omega} (\mathcal{X}^N)$ such that for $\mathbb{P}$-a.e. $\omega \in \Omega$, we have
     \begin{align}
         (\mathrm{Pr}^{i})_{\#} \Upsilon^{\omega} & = \mu^\omega_{i} \qquad \qquad \qquad \qquad \,\,  \forall i \in \left\{1, \cdots, N  \right\}, \\  
        (\mathrm{Pr}^{i,j})_{\#} \Upsilon^{\omega} & \in \mathrm{OptCpl}(\mu^\omega_{i},\mu^\omega_{j}) \qquad \forall i,j\in \left\{1, \cdots, N  \right\},
    \end{align}
    with a $\mathbb{P}$-zero set depending on this subcollection. 
    In other words, the random multi-coupling $\Upsilon$ realizes the distance $\mathbb{W}_p$ between the random marginals:
    $$
    \mathbb{W}_p^p (\mu_i,\mu_j) = \mathbb{E} \left[ \int_{\mathcal{X}^N} d(x_i,x_j)^p \d \Upsilon  \right].
    $$

    \subsection{Probability measure-valued processes}
    Continuing the discussion from \cref{subsec:random_probability_measures,subsec:random_probability_measures_finite_moment}, once the notion of random probability measure is established, we can discuss \emph{probability measure-valued processes}. These are collections $(\mu_t)_{t \in I} \subset P_{\Omega} (\mathcal{X})$ of random measures defined on some probability space $(\Omega,\mathcal{F}, \mathbb{P})$ and 
    indexed by $t$ in some time interval $I \coloneqq [0,T] \subset \mathbb{R}$.
    
    In this paper, we are specifically interested in processes $(\mu_t)_{t \in I} \subset P_{p,\Omega} (\mathcal{X})$ that lie in the corresponding subspace and have certain path regularity with respect to the metric $\mathbb{W}_p$ for some $p\in [1,\infty)$. The former is a condition on the random measure at individual time points, and the latter is a condition on the measure at pairs of time points.
    
    Accordingly, we view such a measure-valued process as a ``single curve'' in the metric space $(P_{p,\Omega}(\mathcal{X}),\mathbb{W}_p)$, with its regularity always understood with respect to the metric $\mathbb{W}_p$.
    For clarity, all norms computed with respect to $\mathbb{W}_p$ will be denoted by $\Vert \cdot \Vert$, distinguishing them from norms $|\cdot|$ computed with respect to ${W}_p$.
    The observation below is useful.

    \begin{lemma}\label{lemma:Exp_balphap_mu} 
        Let $(\mu_t)_{t \in [0,1]} \subset P_{p,\Omega}(\mathcal{X})$ be a measure-valued stochastic process defined on some probability space $(\Omega, \mathcal{F}, \mathbb{P})$. Let $ p \in [1,\infty)$ and $\alpha \in (0,1)$. Then 
        \begin{equation}
            \Vert \mu \Vert^p_{b^{\alpha,p}} = \mathbb{E} \left[ |\mu|^p_{b^{\alpha,p}} \right],
        \end{equation}
        where both may be finite or infinite. 
    \end{lemma}
    \begin{remark}\label{rmk:Exp_Walphap_mu}
        Suppose $\Vert \mu \Vert^p_{W^{\alpha,p}} < + \infty$. 
        Without additional assumptions, we \emph{cannot} similarly state that 
        $$
        \Vert \mu \Vert^p_{W^{\alpha,p}} = \mathbb{E} \left[ |\mu|^p_{W^{\alpha,p}} \right]
        $$
        because the function $\omega \mapsto |\mu^\omega|^p_{W^{\alpha,p}}$ may \emph{not} be measurable. See further discussion on these measurability issues in \cite[Section 2.8]{Abedi2025paths}. These arise for the same reason that  $C(I;\mathcal{X}) \notin \mathcal{B}(\mathcal{X})^I$. See e.g. \cite{RogersWilliams2000_V1,StroockVaradhan2006,FrizVictoir2010book}.
    \end{remark}
    \begin{proof}[Proof of Lemma \ref{lemma:Exp_balphap_mu}]
        Let $D\coloneqq \bigcup_{m\in \mathbb{N}_0} D_m$ be the set of dyadic points of $[0,1]$, where 
        $$
        D_m \coloneqq \big\{ \tfrac{k}{2^m} : \,  k \in \{0,1,\cdots 2^m \} \big\}.
        $$
        For each $t \in [0,1]$, there is a $\mathbb{P}$-zero set outside of which $\mu_t^\omega \in P_p(\mathcal{X})$. Since $D$ is countable, we can say that there is a $\mathbb{P}$-zero set $\mathrm{Z}$ outside of which $\mu_t^\omega \in P_p(\mathcal{X})$ for all $t \in D$. We have
        \begin{align}
             \Vert\mu\Vert_{b^{\alpha,p}}^p
             & \coloneqq \lim_{M \to \infty} \sum_{m=0}^{M} 2^{m(\alpha p -1)} \sum_{k=0}^{2^m-1} \mathbb{W}_p^p (\mu_{t_{k}^{(m)}}, \mu_{t_{k+1}^{(m)}}) \\
             & = \lim_{M \to \infty} \int_{\Omega \setminus \mathrm{Z}} \sum_{m=0}^{M} 2^{m(\alpha p -1)} \sum_{k=0}^{2^m-1}  {W}_p^p (\mu^\omega_{t_{k}^{(m)}}, \mu^\omega_{t_{k+1}^{(m)}}) \, \d \mathbb{P} (\omega)\\
             & = \int_{\Omega \setminus \mathrm{Z}} \sum_{m=0}^{\infty} 2^{m(\alpha p -1)} \sum_{k=0}^{2^m-1} {W}_p^p (\mu^\omega_{t_{k}^{(m)}}, \mu^\omega_{t_{k+1}^{(m)}}) \, \d \mathbb{P} (\omega) = \mathbb{E} \left[ |\mu|^p_{b^{\alpha,p}} \right],
         \end{align}
         where we used Beppo Levi's lemma. In particular, we stress that the function $\omega \mapsto |\mu^\omega|^p_{b^{\alpha,p}}$ is measurable. 
    \end{proof}

 \subsection{Set of random lifts on continuous paths}\label{subsec:set_of_random_lifts}
    Given a Borel probability measure-valued process $(\mu_t)_{t \in I} \subset P_{\Omega} (\mathcal{X})$, indexed by $t \in I \coloneqq [0,T] \subset \mathbb{R}$ and defined on some probability space $(\Omega,\mathcal{F}, \mathbb{P})$, we define the following (possibly empty) set
    \begin{equation}
        \mathrm{Lift}_\Omega (\mu_t) \coloneqq \Big\{\pi \in P_\Omega(C(I;\mathcal{X})) : \quad (e_t)_{\#} \pi^\omega  = \mu_t^\omega \,\, \mathbb{P}\text{-a.s. } \text{for all } t \in I \Big\}.
    \end{equation}
    We study some properties of this set.

    \begin{lemma}\label{lemma:lift_property_convex_random}
        The set $\mathrm{Lift}_\Omega(\mu_t)$, provided it is non-empty, is convex. 
    \end{lemma}

    \begin{proof}
        Let $\pi_0,\pi_1 \in \mathrm{Lift}_\Omega (\mu_t)$, take $\beta \in [0,1]$, and define the random path measure $\pi_{\beta} \coloneqq (1-\beta) \pi_0 + \beta \pi_1$. For a fixed $t \in I$, there is a $\mathbb{P}$-zero set $\mathrm{Z}_0$ outside of which $(e_t)_{\#} \pi_0^\omega  = \mu_t^\omega$ and $(e_t)_{\#} \pi_1^\omega  = \mu_t^\omega$. 
        Let $\phi \in \mathrm{Lip}_b (\mathcal{X})$ and $F \in  \mathcal{F}$  be arbitarty. We have
        \begin{multline}
            \int_{\Omega}\int_{\Gamma_T} \mathds{1}_{F} (\omega) \phi (\gamma_t) \d \pi_\beta^\omega (\gamma) \d \mathbb{P}(\omega) \\
             = (1-\beta) \int_{\Omega \setminus \mathrm{Z}_0}\int_{\Gamma_T} \mathds{1}_{F} (\omega)  \phi (\gamma_t) \d \pi_0^\omega (\gamma)  \d \mathbb{P}(\omega) + \beta \int_{\Omega \setminus \mathrm{Z}_0}\int_{\Gamma_T} \mathds{1}_{F} (\omega)  \phi (\gamma_t) \d \pi_1^\omega (\gamma)  \d \mathbb{P}(\omega) \\
             = \int_{\Omega} \int_{\mathcal{X}} \mathds{1}_{F} (\omega) \phi(x) \d \mu^\omega_t(x)  \d \mathbb{P}(\omega).
        \end{multline}
        This means (by \cite[Lemma 3.14]{Crauel2002}) that $(e_t)_{\#} \pi_\beta^\omega  = \mu_t^\omega$ a.s. and thus $\pi_\beta \in \mathrm{Lift}_\Omega (\mu_t)$.
    \end{proof}

    \begin{lemma}\label{lemma:lift_property_closed_random}
        The set $\mathrm{Lift}_\Omega(\mu_t)$, provided it is non-empty, is closed under narrow convergence of sequences.
    \end{lemma}

    \begin{proof}
        Take a sequence $(\pi_n)_{n \in \mathbb{N}} \subset \mathrm{Lift}_\Omega(\mu_t)$ with $\pi_n \to \pi $ narrowly on $P_\Omega (\Gamma_T)$. Fix $t \in I$ and note that there is a $\mathbb{P}$-zero set $\mathrm{Z}_0$ outside of which $(e_t)_{\#} \pi_n^\omega  = \mu_t^\omega$ for all $n \in \mathbb{N}$. Let $\phi \in \mathrm{Lip}_b (\mathcal{X})$ and $F \in  \mathcal{F}$  be arbitarty and note that $(\gamma,\omega) \mapsto \mathds{1}_{F} (\omega) \phi \circ e_t (\gamma) \in C_{b,\Omega}(\Gamma_T)$. Thus, we have 
        \begin{align}
             \int_{\Omega} \int_{\Gamma_T} \mathds{1}_{F} (\omega) \phi \circ e_t (\gamma) \d \pi^\omega (\gamma) \d \mathbb{P} (\omega)
             & = \lim_{n \to \infty} \int_{\Omega} \int_{\Gamma_T} \mathds{1}_{F} (\omega) \phi \circ e_t (\gamma) \d \pi_n^\omega (\gamma) \d \mathbb{P} (\omega) \\
             & = \lim_{n \to \infty} \int_{\Omega \setminus \mathrm{Z}_0} \mathds{1}_{F} (\omega ) \left(\int_{\Gamma_T}  \phi (\gamma_t) \d \pi_n^\omega (\gamma) \right) \d \mathbb{P} (\omega) \\
             & = \int_{\Omega } \int_{\mathcal{X}} \mathds{1}_{F} (\omega ) \phi (x) \d \mu_t^\omega (x)  \d \mathbb{P} (\omega),
        \end{align}
        which implies (by \cite[Lemma 3.14]{Crauel2002}) that $(e_t)_{\#} \pi^\omega  = \mu_t^\omega$ a.s. and thus $\pi \in \mathrm{Lift}_\Omega (\mu_t)$. 
    \end{proof}

\subsection{The \texorpdfstring{$\nu$}{nu}-based Wasserstein distance}\label{subsec:optimal_transport_nu_based}
    %In this section, we fix the definition of the distance motivated in \eqref{eq:compatibility_ineq} for more general cases of Hilbert spaces. 
    Let $\mathcal{X}$ be a separable Hilbert space and $P(\mathcal{X})$ be the set of Borel probability measures on $\mathcal{X}$.
    As before, let $P_p(\mathcal{X})$, $p\geq 1$, be the subset of measure with finite $p$-moment. In addition, let $P^r(\mathcal{X})$ be the subset of regular measures (\cite[Definition 6.2.2]{AGS2008GFs}), and set $P_p^r(\mathcal{X}) \coloneqq P_p(\mathcal{X}) \cap P^r(\mathcal{X})$.
    We recall that in the case $\mathcal{X} = \mathbb{R}^{\mathrm{d}}$, the subset $P^r(\mathcal{X})$ coincides with $P^{ac}(\mathcal{X})$, the set of measures that are absolutely continuous with respect to the Lebesgue measure. For $p > 1$, the solution to the optimal transport problem from $\nu \in P^r_p(\mathcal{X})$ to $\mu \in P_p(\mathcal{X})$ for the cost function $|x-y|^p$ is given by a unique transport map, which we will denote by $T_\nu^\mu$ (\cite[Theorem 6.2.10]{AGS2008GFs}). 
    
    \begin{definition}[$\nu$-based Wasserstein distance]
        Let $\mathcal{X}$ be a separable Hilbert space and $\nu \in P_p^r (\mathcal{X})$ be a reference measure with $p >1$. For any $\mu_0,\mu_1 \in P_p(\mathcal{X})$, define
        \begin{equation}\label{eq:def_Wpnu}
            W_{p,\nu} (\mu_0,\mu_1) \coloneqq \left( \int_{\mathcal{X}} \big| T_{\nu}^{\mu_1} (y) - T_{\nu}^{\mu_0} (y) \big|^p \d \nu(y) \right)^{1/p}. 
        \end{equation}
    \end{definition}
    This induces a geodesic, called ``generalized geodesic'' between $\mu_0$ and $\mu_1$ with base $\nu$: \begin{equation}\label{eq:def_generalized_geodesic}
        \mu_t \coloneqq \left( (1-t) T_{\nu}^{\mu_0} + t T_{\nu}^{\mu_1} \right)_{\#} \nu, \qquad t \in [0,1],
    \end{equation}
    A recent study of this metric is conducted by \cite{NennaPass2023}. 

\subsection{Wasserstein distance between measures on \texorpdfstring{$\mathbb{R}$}{R}}\label{subsec:optimal_transport_R1}
    The solution to the optimal transport problem in $\mathbb{R}$ for a wide class of convex cost functions can be expressed in terms of cumulative distribution functions (CDF).
    The CDF of a given $\mu \in P(\mathbb{R})$ is defined by
    \begin{equation}\label{eq:CDF}
        F_{\mu} (x) \coloneqq \mu\big((-\infty,x]\big), \quad x \in \mathbb{R},
    \end{equation}
    which is a right-continuous function.
    We also recall the \emph{generalized} inverse of $F_{\mu}$ defined by 
    \begin{equation}\label{eq:CDF_inverse}
        F_{\mu}^{-1} (\alpha) \coloneqq \inf \{ x \in \mathbb{R} | F_{\mu} (x) \geq \alpha \},
    \end{equation}
    which is a left-continuous function. This is also called the left-continuous inverse of $F_{\mu}$ to distinguish it from the other definition, right-continuous inverse, in which $\geq$ above is replaced with $>$. We shall adopt the convention above. Note that 
    $$
     \quad (F_\mu^{-1})_{\#} \mathrm{Leb}|_{[0,1]} = \mu 
    $$
    The following result is traced back to Hoeffding and Fr\'echet and a proof (for general convex cost functions) can be found in \cite[Theorem 2.18]{Villani2003} and \cite[Theorem 6.0.2]{AGS2008GFs}. 
    
    \begin{theorem}
        Let $\mu, \nu \in P_p(\mathbb{R})$ with $1 \leq p < \infty$.
        Then 
        \begin{equation}
            W_p (\mu,\nu)  = \left(\int_{[0,1]} \big|F_{\mu}^{-1}(\alpha) - F_{\nu}^{-1}(\alpha) \big|^p \d \alpha  \right)^{1/p},
        \end{equation}
        and an optimal coupling is
        \begin{equation}
        \Upsilon \coloneqq (F_{\mu}^{-1},F_{\nu}^{-1})_{\#}\mathrm{Leb}|_{[0,1]},
        \end{equation}
        which is unique if $1<p$. Furthermore, if $\mu$ has no atom, i.e., $F_{\mu}$ is continuous, then 
        \begin{equation}
            T \coloneqq F_{\nu}^{-1} \circ F_{\mu}
        \end{equation}
        is an optimal transport map from $\mu$ onto $\nu$ and it is again unique if $1<p$. 
    \end{theorem}
    Finally, let us observe that the Wasserstein distance between measures on $\mathbb{R}$ can be regarded as $\mathrm{Leb}|_{[0,1]}$-based Wasserstein distance, as defined in \eqref{eq:def_Wpnu}.

    \subsection{Compatibility of random probability measure on \texorpdfstring{$\mathbb{R}$}{R} with finite \texorpdfstring{$p$}{p}-moments}

    \begin{remark}[Any collection $\mathcal{M} \subset P_{p,\Omega}(\mathbb{R})$ is compatible in the sense of Definition \ref{def:compatibility_random}]
        	Let $\mathcal{M} \subset P_{p,\Omega}(\mathbb{R})$ be an arbitrary collection, and let $\{\mu_i\}_{1\leq i\leq N} \subset \mathcal{M}$ be a finite subcollection. For each $i$, denote by $F_{\mu_i^\omega}$ the CDF of $\mu_i^\omega$, and by $F_{\mu_i^\omega}^{-1}$ its generalized inverse.  We define the following  multi-coupling
	 $$
	 \Upsilon^\omega
	 \coloneqq
	 \bigl(F_{\mu_1^\omega}^{-1},F_{\mu_2^\omega}^{-1},\ldots,F_{\mu_N^\omega}^{-1}\bigr)_{\#}\mathrm{Leb}|_{[0,1]},
	 $$
	 outside a $\mathbb{P}$-zero set.
	 Since $\mu_i$ is a random probability measure on $\mathbb{R}$, one can show that the map $(\omega,\alpha)\mapsto F_{\mu_i^\omega}^{-1}(\alpha)$ is measurable on $\Omega \times (0,1)$ (see the claim below). Hence $\Upsilon$ is a random probability measure on $\mathbb{R}^N$. For every $i,j\in\{1,\ldots,N\}$, the $(i,j)$-marginal of $\Upsilon^\omega$ is given by
	 $$
	 (\mathrm{Pr}^{i,j})_{\#}\Upsilon^\omega
	 =
	 \bigl(F_{\mu_i^\omega}^{-1},F_{\mu_j^\omega}^{-1}\bigr)_{\#}\mathrm{Leb}|_{[0,1]},
	$$
	which, as discussed in the previous section, is an optimal coupling for $(\mu_i^\omega,\mu_j^\omega)$  for $\mathbb{P}$-a.e. $\omega\in\Omega$. Therefore all two-dimensional marginals of $\Upsilon^\omega$ are optimal for the corresponding pairs of marginals, $\mathbb{P}$-a.s. This shows that $\mathcal{M}$ is compatible in $P_{p,\Omega}(\mathbb{R})$ in the sense of Definition \ref{def:compatibility_random}. It remains to verify the claim.
    
	\begin{itemize}[leftmargin=1.5em] 
	\item[] \textit{Claim.}
	Let $\mu \in P_{\Omega}(\mathbb{R})$. The map $(\omega,\alpha) \mapsto F_{\mu^\omega}^{-1}(\alpha)$ is measurable on $\Omega \times (0,1)$.
	
	\item[] Indeed, for every $q \in \mathbb{Q}$, the map
	$ \omega \mapsto F_{\mu^\omega}(q)=\mu^\omega((-\infty,q]) $
	is measurable by the definition of a random probability measure. Hence the map
	$(\omega,\alpha) \mapsto F_{\mu^\omega}(q)-\alpha$
	is measurable on $\Omega \times (0,1)$. Now, for every $r \in \mathbb{R}$, we have
	$$
	\left\{(\omega,\alpha): F_{\mu^\omega}^{-1}(\alpha)<r\right\}
	=
	\bigcup_{\substack{q\in\mathbb{Q}\\ q<r}}
	\left\{(\omega,\alpha): F_{\mu^\omega}(q)\geq \alpha\right\}.
	$$
	The right-hand side is measurable, since it is a countable union of measurable sets. Therefore,
	$(\omega,\alpha) \mapsto F_{\mu^\omega}^{-1}(\alpha)$
	is measurable. We may define the values at $\alpha=0$ and $\alpha=1$ arbitrarily, since these two points do not affect the push-forward with respect to $\mathrm{Leb}|_{[0,1]}$.
	\end{itemize}
    \end{remark}

    \subsection{Further remarks} In this paper, we frequently require a geodesic selection and interpolation map: given finitely many points, such a map selects constant-speed geodesics between consecutive points and interpolates between them. Here we discuss the measurability of such a map. It is enough to discuss the two-point case, where one seeks a selection map
	$\ell : \mathcal{X}^2 \to Geo([0,1];\mathcal{X})$. 
	In general, such a selection map need \emph{not} be Borel measurable.
	In the author's previous work in the deterministic setting, for a complete separable geodesic metric space $(\mathcal{X},d)$ and a fixed measure $\Upsilon \in P(\mathcal{X}^2)$, we discuss in \cite[Section 2.13]{Abedi2025paths} that one can find a map $\ell$ measurable with respect to the $\Upsilon$-completion of $\mathcal{B}(\mathcal{X}^2)$. This was sufficient to define the push-forward measure $\pi \coloneqq (\ell)_{\#}\Upsilon$. In the present random setting, however, we deal with random measures $\omega \mapsto \Upsilon^\omega$, and need a slightly stronger statement, namely the existence of a \emph{Borel} measurable selection map. The next remark provides this under the additional \emph{local compactness} assumption on $(\mathcal{X},d)$.

    \begin{remark}[Measurable selection of geodesics under local compactness]\label{rmk:measurable_geodesic_selection}
    Let $(\mathcal{X},d)$ be a complete, separable, and locally compact length metric space. We show that there exists a Borel measurable geodesic selection map $\ell : \mathcal{X}^2 \to Geo([0,1];\mathcal{X})$.
	\\
	By the Hopf–Rinow Theorem, $\mathcal{X}$ is a geodesic space.
	We thus define the following (non-empty) set 
	\begin{equation}
		\mathcal{G} \coloneqq \Big\{((x,y),\gamma) \in \mathcal{X}^2\times C([0,1];\mathcal{X}): \quad  \gamma \in Geo ([0,1];\mathcal{X}) , \gamma_0= x, \gamma_1 = y \Big\}.
	\end{equation}
	This set is closed, in particular,  it is Borel-measurable, i.e., it belongs to $\mathcal{B}(\mathcal{X}^2\times C([0,1];\mathcal{X}))$. 
	Given arbitrary points $(x,y) \in \mathcal{X}^2$, we define
	\begin{equation}
			\mathcal{G}_{(x,y)} \coloneqq \Big\{ \gamma \in C([0,1];\mathcal{X}): \quad ((x,y),\gamma) \in \mathcal{G} \Big\}.
	\end{equation}
	This set is nonempty. Moreover, we claim that it is compact. Indeed, let $\bar{x} \in \mathcal{X}$ be a fixed reference point, and note that firstly 
	\begin{equation}
		\sup_{\gamma \in \mathcal{G}_{(x,y)}} d (\gamma_0,\bar{x}) = d(x,\bar{x}) < + \infty. 
	\end{equation}
	and, secondly, every $\gamma \in \mathcal{G}_{(x,y)}$ satisfies
	\begin{equation}
		d (\gamma_t,\gamma_s) = d (x,y) |t-s| \qquad \forall t,s \in [0,1],
	\end{equation}
	and hence the family $\mathcal{G}_{(x,y)}$ is equicontinuous.
	By Arzel\`a-Ascoli theorem (note that by the Hopf–Rinow theorem, the closed bounded sets in $\mathcal{X}$ are compact), the two conditions above imply that $\mathcal{G}_{(x,y)}$ is relatively compact in $C([0,1];\mathcal{X})$. To see that it is in fact compact, take a sequence $(\gamma^n) \subset \mathcal{G}_{(x,y)}$. By relative
	compactness, we know that there is a convergent subsequence $\gamma^{n_k} \to \gamma$ in $C([0,1];\mathcal{X})$. Then clearly $\gamma_0=x$ and $\gamma_1=y$, and for any $t,s \in [0,1]$
	$$
	d (\gamma_t,\gamma_s) = \lim_{k \to \infty} d (\gamma^{n_k}_t,\gamma^{n_k}_s) = \lim_{k \to \infty} d (x,y) |t-s| = d (x,y) |t-s|
	$$
	which means that $\gamma \in \mathcal{G}_{(x,y)}$. Therefore, $\mathcal{G}_{(x,y)}$ is compact, in particular, $\sigma$-compact. By \cite[Theorem 6.9.6]{Bogachev2007MeasureTheory}, it follows that $\mathcal{G}$ contains the graph of a Borel map
	$
	\ell : \mathcal{X}^2 \to C([0,1];\mathcal{X}).
	$
	By construction, $\ell(x,y) \in Geo([0,1];\mathcal{X})$, $\ell(x,y)_0=x$, and $\ell(x,y)_1=y$ for every $(x,y)\in\mathcal{X}^2$. Thus $\ell$ is a Borel measurable geodesic selection map.
    \end{remark}

\section{Main results}\label{sec:results_s}
\subsection{Existence of a random minimizer for the variational problem}
    We consider Problem \ref{prob:variational_problem_random}, as formulated in the introduction.
    \begin{proposition}[Existence of a random minimizer]\label{prop:existence_of_minimizer_random}
         Let $(\mathcal{X}, d)$ be a complete separable metric space, $(\Omega, \mathcal{F},\mathbb{P})$ be a probability space, and $I \coloneqq [0,T] \subset \mathbb{R}$.
         Let $\Psi: C(I;\mathcal{X}) \to [0,+\infty]$ be a map whose sublevels are compact in $C(I;\mathcal{X})$.
         Assume that the infimum  \eqref{eq:variational_problem_random} is finite. Then there exists a random minimizer $\pi \in P_\Omega(C(I;\mathcal{X}) ) $ to Problem \ref{prob:variational_problem_random}. 
    \end{proposition}    

    \begin{proof}
        Let $(\pi_n)_{n \in \mathbb{N}} \subset \mathrm{Lift}_\Omega (\mu_t)$ be a minimizing sequence, i.e.,
        \begin{equation}
            \lim_{n \to \infty} \mathbb{E} \left[\int_{\Gamma_T} \Psi \d \pi_n \right]= \inf_{\pi \in \mathrm{Lift}_\Omega (\mu_t)}  \mathbb{E} \left[\int_{\Gamma_T} \Psi \d \pi \right] \eqqcolon M < + \infty.
        \end{equation}
        In particular, we have 
        \begin{equation}\label{eq:proof_existence_minimizer_random}
            \sup_{n} \, \mathbb{E} \left[ \int_{\Gamma_T}  \Psi \d \pi_n \right] < + \infty.
        \end{equation}
        Since the sublevels of $\Psi$ are compact, the condition \eqref{eq:proof_existence_minimizer_random} implies that the family of random measures $\{\pi_n \}_{n} \subset P_\Omega (\Gamma_T)$ is tight by Lemma \ref{lemma:tightness_criterion_random}.
        By Prokhorov's theorem in the random setting (Theorem \ref{thm:Prokhorov_random}), there exists a subsequence $\{\pi_{n_k} \}_{k \in \mathbb{N}}$ such that $\pi_{n_k} \to \pi$ narrowly on $P_\Omega(\Gamma_T)$ as $k \to \infty$. By Lemma \ref{lemma:lift_property_closed_random}, we conclude that $ \pi \in \mathrm{Lift}_\Omega (\mu_t)$ and thus $
        \mathbb{E}[\int_{\Gamma_T} \Psi \d \pi ] \geq M$. On the other hand, since $\Psi$ is in particular lower semi-continuous, we apply \eqref{eq:narrow_conv_lsc_func_random} and get the reverse inequality 
        \begin{equation}
            \mathbb{E} \left[\int_{\Gamma_T} \Psi \d \pi \right] \leq \liminf_{k \to \infty } \mathbb{E} \left[ \int_{\Gamma_T} \Psi \d \pi_{n_k} \right] = M. 
        \end{equation}
        Therefore, $\mathbb{E} [\int_{\Gamma_T} \Psi \d \pi ] = M$ holds and thus $\pi$ is a minimizer.
    \end{proof}

\subsection{From random path measures to random Wasserstein curves}\label{subsec:from_pi_to_mu_random}
    In this section, we start with a random path measure $\pi \in P_\Omega(C([0,T];\mathcal{X}))$ whose energy, with respect to certain regularity, is finite in expectation. 
    We then look at its one-dimensional time marginals $\mu_t^\omega \coloneqq (e_t)_{\#} \pi^\omega \in P(\mathcal{X})$ defined for all $t \in [0,T]$ and for almost all $\omega \in \Omega$. This induces a new random measure $\mu_t \in P_{\Omega} (\mathcal{X})$ for all $t \in [0,T]$, as the evaluation map $e_t : \Gamma_T \to \mathcal{X}$ is measurable.
    It turns out that in all three cases below, the sample paths $t \mapsto \mu_t^\omega$ lie in the  Wasserstein space $P_p(\mathcal{X})$ a.s. and have certain regularity with respect to the distance $W_p$. In addition, $t \mapsto \mu_t$ lies in  $P_{p,\Omega}(\mathcal{X})$ and has certain regularity with respect to the distance $\mathbb{W}_p$.
    The results of this section essentially rely on \cite[Section 3.2]{Abedi2025paths} and generalize it to the random setting. 
    
    \begin{theorem}\label{thm:lift_to_mu_Hol_random}
        Let $(\mathcal{X},d)$ be a complete separable metric space, and $(\Omega, \mathcal{F}, \mathbb{P})$ be a probability space.
        Let $\pi \in P_\Omega(C([0,T];\mathcal{X}))$ satisfy
        %be concentrated on $C^{\theta\textrm{-} \mathrm{H\ddot{o}l}} ([0,T];\mathcal{X})$
         \begin{equation}
            \mathbb{E} \left[ \int_{\Gamma_T} \Big( d(\gamma_0,\bar{x})^p + | \gamma |_{\theta\textrm{-}\mathrm{H\ddot{o}l};[0,T]}^p \Big) \d \pi(\gamma) \right] < + \infty
        \end{equation}
        for some $p \in [1,\infty)$ and $\theta\in(0,1]$ and $\bar{x}\in \mathcal{X}$.
        Then, for $\mathbb{P}$-a.e. $\omega \in \Omega$, the curve $t \mapsto \mu_t^\omega \coloneqq {(e_t)}_{\#} \pi^\omega $ is in $ C^{\theta\textrm{-} \mathrm{H\ddot{o}l}} ([0,T];P_p(\mathcal{X}))$ with
        \begin{equation}\label{eq:lift_to_mu_Hol_random_ineq_pointwise}
            |\mu^\omega|_{\theta\textrm{-}\mathrm{H\ddot{o}l};[s,t]}^p \leq \int_{\Gamma_T} |\gamma|_{\theta\textrm{-}\mathrm{H\ddot{o}l};[s,t]}^p \d \pi^\omega (\gamma)
        \end{equation}
        for any $0\leq s<t\leq T$. In addition, $ (\mu_t) \in  C^{\theta\textrm{-} \mathrm{H\ddot{o}l}} ([0,T];P_{p,\Omega}(\mathcal{X}))$ with
        \begin{equation}\label{eq:lift_to_mu_Hol_random_ineq_exp}
            \Vert\mu\Vert_{\theta\textrm{-}\mathrm{H\ddot{o}l};[s,t]}^p \leq \mathbb{E}\left[ \int_{\Gamma_T} |\gamma|_{\theta\textrm{-}\mathrm{H\ddot{o}l};[s,t]}^p \d \pi (\gamma)\right]
        \end{equation}
        for any $0\leq s<t\leq T$.
    \end{theorem}
    
    \begin{proof} The integrability assumption implies that there exists a $\mathbb{P}$-full measure set $\Omega_\circ \subseteq \Omega$ on which the quantity inside the expectation is finite. For all $\omega \in \Omega_\circ$ and $t \in [0,T]$, we define 
    $\mu_t^\omega \coloneqq {(e_t)}_{\#} \pi^\omega$.
    From the proof of \cite[Theorem 3.3]{Abedi2025paths}, we have the estimate
        \begin{align}\label{eq:lift_to_mu_Hol_random_proof_1}
            \int d(\bar{x},x)^p \d \mu_t^\omega \leq 2^{p-1} \int d(\gamma_0,\bar{x})^p \d \pi^\omega + 2^{p-1}  |T|^{p\theta} \int |\gamma|_{\theta\textrm{-}\mathrm{H\ddot{o}l};[0,T] }^p \d\pi^\omega < + \infty,
        \end{align}
        along with the pointwise bound \eqref{eq:lift_to_mu_Hol_random_ineq_pointwise}. Hence, $(\mu_t^\omega) \in C^{\theta\textrm{-} \mathrm{H\ddot{o}l}} ([0,T];P_p(\mathcal{X}))$ a.s.
        \\
        As for the second statement, taking the expectation of \eqref{eq:lift_to_mu_Hol_random_proof_1} and disregarding the $\mathbb{P}$-zero measure set, we conclude that $\mu_t \in P_{p,\Omega} (\mathcal{X})$ for all $t\in[0,T]$.
        Then, we observe that
        \begin{align}
            \mathbb{W}_p^p(\mu_u,\mu_v) & \leq \mathbb{E} \left[ \int_{\Gamma_T} d(\gamma_u,\gamma_v)^p \d\pi(\gamma)\right] \leq |v-u|^{p \theta}  \mathbb{E} \left[ \int_{\Gamma_T} |\gamma|_{\theta\textrm{-}\mathrm{H\ddot{o}l};[s,t] }^p \d\pi(\gamma) \right],
        \end{align}
        for any $s \leq u<v \leq t$,
        which then implies that
        \begin{equation}
            \Vert\mu\Vert_{\theta\textrm{-}\mathrm{H\ddot{o}l};[s,t]} \coloneqq \sup_{s\leq u< v \leq t} \frac{\mathbb{W}_p(\mu_u, \mu_v)}{|v-u|^\theta} \leq \left( \mathbb{E} \left[ \int_{\Gamma_T} |\gamma|_{\theta\textrm{-}\mathrm{H\ddot{o}l};[s,t] }^p \d\pi(\gamma) \right] \right)^{1/p}.
        \end{equation}
        Raising this expression to the power $p$, we obtain the result. 
    \end{proof}

    \begin{theorem}\label{thm:lift_to_mu_var_random}
        Let $(\mathcal{X},d)$ be a complete separable metric space, and $(\Omega, \mathcal{F}, \mathbb{P})$ be a probability space.
        Let $\pi \in P_{\Omega}(C([0,T];\mathcal{X})) $ satisfy
         \begin{equation}\label{eq:lift_to_mu_var_integrability_random}
            \mathbb{E} \left[ \int_{\Gamma_T} \Big( d(\gamma_0,\bar{x})^p + | \gamma |_{q\textrm{-}\mathrm{var};[0,T]}^p \Big) \d \pi(\gamma)  \right] < + \infty
        \end{equation}
        for some $1 \leq q \leq p < \infty$ and  $\bar{x}\in \mathcal{X}$.
        Then, for $\mathbb{P}$-a.e. $\omega \in \Omega$, the curve  $t \mapsto \mu_t^\omega \coloneqq {(e_t)}_{\#} \pi^\omega $ is in $ C^{p\textrm{-} \mathrm{var}} ([0,T];P_p(\mathcal{X}))$ with
        \begin{equation}\label{eq:lift_to_mu_var_random_ineq_pointwise}
            |\mu^\omega|_{p\textrm{-}\mathrm{var};[s,t]}^p \leq \int_{\Gamma_T} |\gamma|_{q\textrm{-}\mathrm{var};[s,t]}^p \d \pi^\omega(\gamma)
        \end{equation}
        for any $0\leq s<t\leq T$.  In addition, $ (\mu_t) \in  C^{p\textrm{-} \mathrm{var}} ([0,T];P_{p,\Omega}(\mathcal{X}))$ with
        \begin{equation}
            \Vert\mu\Vert_{p\textrm{-}\mathrm{var};[s,t]}^p \leq \mathbb{E} \left[ \int_{\Gamma_T} |\gamma|_{q\textrm{-}\mathrm{var};[s,t]}^p \d \pi(\gamma) \right]
        \end{equation}
        for any $0\leq s<t\leq T$.
    \end{theorem}

    \begin{proof} There exists a $\mathbb{P}$-full measure set $\Omega_\circ \subseteq \Omega$ on which the quantity inside the expectation in \eqref{eq:lift_to_mu_var_integrability_random} is finite. For all $\omega \in \Omega_\circ$ and $t \in [0,T]$, we set 
    $\mu_t^\omega \coloneqq {(e_t)}_{\#} \pi^\omega$.
    From the proof of \cite[Theorem 3.4]{Abedi2025paths}, we conclude that first
    \begin{equation}\label{eq:lift_to_mu_var_random_proof_1}
        \int_\mathcal{X} d(\bar{x},x)^p \d \mu_t^\omega(x) \leq 2^{p-1} \int_{\Gamma_T} \Big(  d(\bar{x},\gamma_0)^p + | \gamma |_{q\textrm{-}\mathrm{var};[0,T]}^p \Big) \d\pi^\omega(\gamma) < + \infty,
    \end{equation}
    secondly, $t \mapsto \mu_t^\omega$ is continuous in $P_p(\mathcal{X})$, and thirdly, the bound  \eqref{eq:lift_to_mu_var_random_ineq_pointwise} holds. In short, we have $(\mu_t^\omega) \in C^{p\textrm{-} \mathrm{var}} ([0,T];P_p(\mathcal{X}))$ a.s.
    \\
    Next, by taking the expectation of \eqref{eq:lift_to_mu_var_random_proof_1}, we see that $\mu_t \in P_{p,\Omega} (\mathcal{X})$ for all $t\in[0,T]$. 
    We would like to show that $t \mapsto \mu_t$ is continuous in $P_{p,\Omega}(\mathcal{X})$. Fix $t$ and take an arbitrary sequence $t_n \to t$ as $n \to \infty$.
    By Lebesgue's dominated convergence theorem on the measure space $(\Gamma_T\times \Omega, \mathcal{B}(\Gamma_T) \otimes \mathcal{F}, \bm{\pi})$, we obtain
    \begin{align}
        \lim_{n \to \infty} \mathbb{W}_p^p (\mu_t, \mu_{t_n})
        & \leq \lim_{n \to \infty} \int_{\Omega} \int_{\Gamma_T} d(\gamma_t,\gamma_{t_n})^p \d \pi^\omega (\gamma) \d \mathbb{P} (\omega) \\
        & = \lim_{n \to \infty} \int_{\Gamma_T \times \Omega } d(\gamma_t,\gamma_{t_n})^p \d \bm{\pi} (\gamma,\omega) \\
        & = \int_{\Gamma_T \times \Omega } \left( \lim_{n \to \infty} d(\gamma_t,\gamma_{t_n})^p \right) \d\bm{\pi} (\gamma,\omega) = 0,
    \end{align}
    where $d(\gamma_s,\gamma_t)^p \leq | \gamma |_{q\textrm{-}\mathrm{var}}^p$ gives us the domination, which is integrable by assumption. It remains to show that $t \mapsto \mu_t$ has finite $p$-variation in $P_{p,\Omega}$. Take $0 \leq s <t \leq T$ and let $D=(t_i)$ be an arbitrary partition of the interval $[s,t]$. We have
    \begin{align}
        \sum_{i} \mathbb{W}_p^p (\mu_{t_{i}}, \mu _{t_{i+1}}) & \leq \mathbb{E} \left[  \int_{\Gamma_T} \sum_{i} d(\gamma_{t_{i}}, \gamma _{t_{i+1}})^p \d\pi(\gamma) \right] \leq \mathbb{E} \left[ \int_{\Gamma_T}  |\gamma|_{q\textrm{-}\mathrm{var};[s,t]}^p \d\pi(\gamma) \right],
    \end{align}
    where we used the fact that $|\gamma|_{p\textrm{-}\mathrm{var}}\leq |\gamma|_{q\textrm{-}\mathrm{var}}$ on the space of continuous paths if $1 \leq q \leq p$.
    Taking the supremum over all dissections completes the proof.
    \end{proof}
    
    \begin{theorem}\label{thm:lift_to_mu_Walphap_stoch}
    Let $(\mathcal{X},d)$ be a complete separable metric space, and $(\Omega, \mathcal{F}, \mathbb{P})$ be a probability space. Let $\pi \in P_{\Omega} (C([0,T];\mathcal{X})) $ satisfy %concentrated on $W^{\alpha, p} ([0,T];\mathcal{X})$
    \begin{equation}\label{eq:lift_to_mu_Walphap_integrability_stoch}
            \mathbb{E} \left[  \int_{\Gamma_T} \Big( d(\gamma_0,\bar{x})^p + | \gamma |_{W^{\alpha,p};[0,T]}^p  \Big) \d \pi(\gamma) \right] < + \infty
        \end{equation}
        for some $1<p<\infty$ and $ \frac{1}{p} < \alpha < 1$ and $\bar{x}\in \mathcal{X}$.
        Then, for $\mathbb{P}$-a.e. $\omega \in \Omega$, the curve $t \mapsto  \mu_t^\omega \coloneqq {(e_t)}_{\#} \pi^\omega$ is in $W^{\alpha, p} ([0,T];P_p(\mathcal{X}))$ with
        \begin{equation}\label{eq:lift_to_mu_Walphap_random_ineq_pointwise} 
            |\mu^\omega|_{W^{\alpha,p};[s,t]}^p \leq \int_{\Gamma_T} |\gamma|_{W^{\alpha,p};[s,t]}^p \d \pi^\omega(\gamma)
        \end{equation}
        for any $0\leq s<t\leq T$. In addition, $(\mu_t) \in W^{\alpha, p} ([0,T];P_{p,\Omega}(\mathcal{X}))$ with 
        \begin{equation}\label{eq:lift_to_mu_Walphap_random_ineq_exp} 
            \Vert\mu\Vert_{W^{\alpha,p};[s,t]}^p \leq \mathbb{E} \left[\int_{\Gamma_T} |\gamma|_{W^{\alpha,p};[s,t]}^p \d \pi (\gamma) \right]
        \end{equation}
        for any $0\leq s<t\leq T$. The same statement holds for $|\cdot|_{b^{\alpha,p}}$.
    \end{theorem}
    \begin{proof}
        As in the previous proofs, the integrability assumption implies that there exists a $\mathbb{P}$-full measure set $\Omega_\circ \subseteq \Omega$ on which the quantity inside the expectation is finite. For all $\omega \in \Omega_\circ$ and $t \in [0,T]$, we set $\mu_t^\omega \coloneqq {(e_t)}_{\#} \pi^\omega$.
        From the proof of \cite[Theorem 3.5]{Abedi2025paths}, we have that
        \begin{equation}\label{eq:lift_to_mu_Walphap_stoch_proof_1}
            \int d(\bar{x},x)^p \d \mu_t^\omega \leq 2^{p-1} \int d(\gamma_0,\bar{x})^p \d \pi^\omega + 2^{p-1} \bar{c}^p |T|^{\alpha p - 1 }\int |\gamma|_{W^{\alpha,p};[0,T] }^p \d\pi^\omega < + \infty
        \end{equation}
        and  $(\mu_t^\omega) \in W^{\alpha, p} ([0,T];P_p(\mathcal{X}))$ a.s. with the following point-wise bounds
        \begin{equation}\label{eq:lift_to_mu_Walphap_stoch_proof_2}
            |\mu^\omega|_{W^{\alpha,p};[s,t]}^p \leq \int_{\Gamma_T} |\gamma|_{W^{\alpha,p};[s,t]}^p \d \pi^\omega(\gamma)
        \end{equation}
        for any $0\leq s<t\leq T$ and 
        \begin{equation}\label{eq:lift_to_mu_Walphap_stoch_proof_3}
            |\mu^\omega|_{b^{\alpha,p}}^p \leq \int_{\Gamma} |\gamma|_{b^{\alpha,p}}^p \d \pi^\omega(\gamma)
        \end{equation}
        when $T=1$. 
        \\
        To prove the second statement, we note that taking the expectation of \eqref{eq:lift_to_mu_Walphap_stoch_proof_1} and ignoring the $\mathbb{P}$-zero set yields that $\mu_t \in P_{p,\Omega} (\mathcal{X})$ for all $t\in[0,T]$.
        Then, we observe that
        \begin{align}
             \Vert\mu\Vert_{W^{\alpha,p};[s,t]}^p
             & \coloneqq \iint_{[s,t]^2} \frac{\mathbb{E} [W_p^p (\mu_u,\mu_v)]}{|v-u|^{1+\alpha p}} \d u \d v \\
             & \leq \iint_{[s,t]^2} \frac{\mathbb{E}  \left[\int_{\Gamma_T} d(\gamma_u,\gamma_v)^p \d \pi (\gamma) \right] }{|v-u|^{1+\alpha p}} \d u \d v \\
             & = \mathbb{E} \left[\int_{\Gamma_T} |\gamma|_{W^{\alpha,p};[s,t]}^p \d \pi (\gamma) \right] < + \infty.
        \end{align}
        Similarly, for $b^{\alpha,p}$, using  Lemma \ref{lemma:Exp_balphap_mu} and the inequality \ref{eq:lift_to_mu_Walphap_stoch_proof_3}, we obtain
         \begin{align}
             \Vert\mu\Vert_{b^{\alpha,p}}^p
             & = \mathbb{E} \left[ |\mu|^p_{b^{\alpha,p}} \right] \\
             & \leq \mathbb{E} \left[ \int_\Gamma |\gamma|_{b^{\alpha,p}}^p \d \pi (\gamma) \right] < + \infty.
         \end{align}
        We thus proved that $(\mu_t) \in W^{\alpha, p} ([0,T];P_{p,\Omega}(\mathcal{X}))$ with respect to the metric $\mathbb{W}_p$.
    \end{proof}

    \begin{proposition}\label{prop:equality_implies_compatibility_random}
    Let $(\mathcal{X},d)$ be a complete separable metric space, and $(\Omega, \mathcal{F}, \mathbb{P})$ be a probability space.
    Let $\pi \in P_{\Omega}(C([0,T];\mathcal{X})) $ satisfy \eqref{eq:lift_to_mu_Walphap_integrability_stoch} for some $1<p<\infty$ and $ \frac{1}{p} < \alpha < 1$ and $\bar{x}\in \mathcal{X}$.
    Assume that $\pi$ and $t \mapsto \mu_t \coloneqq {(e_t)}_{\#} \pi $ satisfy the equality
        \begin{align}
            \Vert\mu\Vert_{W^{\alpha,p};[0,T]}^p = \mathbb{E} \left[\int_{\Gamma_T} |\gamma|_{W^{\alpha,p};[0,T]}^p \d \pi (\gamma)\right],
        \end{align}
        then $(\mu_t)_{t\in [0,T]}$ is compatible in $P_{p,\Omega}(\mathcal{X})$.
    \end{proposition}

    \begin{proof}
        The equality implies 
        \begin{align}
             0 =  \iint_{[0,T]^2} \underbrace{ \left(  \frac{\mathbb{E}\left[\int_{\Gamma_T} d(\gamma_s,\gamma_t)^p \d \pi (\gamma)\right]}{|t-s|^{1+\alpha p}} - \frac{\mathbb{W}_p^p (\mu_s,\mu_t)}{|t-s|^{1+\alpha p}}  \right)}_{\eqqcolon f(s,t)}  \d s \d t .
         \end{align}
         According to the definition of $\mathbb{W}_p$, the function $f$ is non-negative, because $\pi$ is just a random lift of $(\mu_t)$. Thus, the equality above implies that $f=0$ almost everywhere. We now show that $f$ is continuous and then conclude that $f=0$ everywhere. We first note that
         the application
         $$
         (s,t) \mapsto \mathbb{W}_p^p (\mu_s,\mu_t)
         $$
         is continuous becasue $(\mu_t) \in W^{\alpha, p} ([0,T];P_{p,\Omega}(\mathcal{X})) \subset C ([0,T];P_{p,\Omega}(\mathcal{X}))$ on the given parameter range. Also, the application
         $$
         (s,t) \mapsto \mathbb{E} \left[ \int_{\Gamma_T} d(\gamma_s,\gamma_t)^p \d \pi (\gamma)  \right] 
         $$
         is also continuous. Take a sequence $(s_m,t_m) \to (s,t)$ as $m \to \infty$ and note that
         \begin{equation}
             \lim_{m \to \infty } \mathbb{E} \left[ \int_{\Gamma_T} d(\gamma_{s_m},\gamma_{t_m})^p \d \pi \right] = \mathbb{E} \left[ \int_{\Gamma_T} \left(  \lim_{m \to \infty } d(\gamma_{s_m},\gamma_{t_m})^p \right) \d \pi  \right] = \mathbb{E} \left[ \int_{\Gamma_T} d(\gamma_{s},\gamma_{t})^p \d \pi  \right]
         \end{equation}
         by Lebesgue's dominated convergence theorem, where the dominated function
         $$
             d (\gamma_s,\gamma_t)^p \leq \bar{c}^p T^{\alpha p  - 1} | \gamma |^p_{W^{\alpha,p};[0,T]},
         $$
         coming from \eqref{eq:distance_estimate_GRR_r_paper}, is integrable on $\Omega \times \Gamma_T$ by assumption.
         \\
         Accordingly, for any finite collection of time points $\{t_i \in [0,T]$: $i \in \left\{1, \cdots, N \right\}\}$, the projection   $\omega \mapsto (e_{t_1},e_{t_2}, \cdots, e_{t_N})_\# \pi^\omega$ provides a random multi-coupling for $\{\mu_{t_i}: i \in \left\{1, \cdots, N \right\} \}$ whose two-dimensional marginals are all optimal. So $(\mu_t)$ is compatible in the sense of Definition \ref{def:compatibility_random}.
    \end{proof}

    \subsection{From measure-valued processes to random path measures: a stochastic superposition principle} We consider Construction \hyperref[itm:ConstructionB_random]{{\footnotesize$\bigstar$}} as explained in the introduction.
    \begin{theorem}\label{thm:optimal_lift_mu_Walphap_compatible_random}
        Let $(\mathcal{X}, d)$ be a complete, separable, and locally compact length metric space, and $I \coloneqq [0,T] \subset \mathbb{R}$.
        Let $(\mu_t)_{t \in I}$ be a probability measure-valued stochastic process defined on a probability space $(\Omega, \mathcal{F}, \mathbb{P})$ such that $(\mu_t) \in W^{\alpha,p} (I; P_{p,\Omega}(\mathcal{X}))$ for some $1<p<\infty$ and $\frac{1}{p}< \alpha <  1$. 
        Assume that $(\mu_t)_{t \in I}$ is compatible in $P_{p,\Omega}(\mathcal{X})$.
        Then, construction \hyperref[itm:ConstructionB_random]{{\footnotesize$\bigstar$}} converges narrowly (up to a subsequence) to a random probability measure $\pi \in P_{\Omega} (C(I;\mathcal{X}))$ satisfying
        \begin{enumerate}[label=(\roman*), font=\normalfont]
            \item $\pi^\omega$ is concentrated on $W^{\alpha,p}(I;\mathcal{X})$ and $t \mapsto {(e_t)}_{\#} \pi^\omega$ is in $ W^{\alpha, p} (I;P_p(\mathcal{X}))$ a.s.;
            \item $(e_t)_\#\pi^\omega =\mu_t^{\omega} $ a.s. for all $t\in I$;
            \item $(e_s,e_t)_\# \pi^\omega \in \mathrm{OptCpl}(\mu_s^{\omega}, \mu_t^{\omega})$ a.s. for all $s,t \in I$; and in particular,
            \begin{equation}\label{eq:optimality_pi_1_random}
                \Vert\mu\Vert_{W^{\alpha,p}}^p = \mathbb{E} \left[\int_{\Gamma_T} |\gamma|_{W^{\alpha,p}}^p \d \pi (\gamma) \right].
            \end{equation}
        \end{enumerate}
    The same statement holds for $|\cdot|_{b^{\alpha,p}}$.
    \end{theorem}

    \begin{proof}
        The proof follows the same steps as in \cite[Section 3.3]{Abedi2025paths} in the deterministic setting, with a generalization to the random setting. The strategy is indeed based on \cite[Theorem 5]{Lisini2007}. 
        Without loss of generality, we take $T=1$, i.e., $I=[0,1]$.
        \\
        To begin with, let us note that, by the Fractional Sobolev-H\"{o}lder embedding Theorem \ref{thm:FractionalSobolev-Holder_r_paper} applied on the metric space $(P_{p,\Omega}(\mathcal{X}), \mathbb{W}_p)$, we get
        \begin{equation}\label{eq:E_Wpp_nu}
                \mathbb{W}_p (\mu_s,\mu_t)  \leq \bar{c} |t-s|^{\alpha - \tfrac{1}{p}} \Vert \mu  \Vert_{W^{\alpha,p}}.
        \end{equation}
        Thus, $t \mapsto \mu_t$ is continuous in $P_{p,\Omega}(\mathcal{X})$. In particular, $t \mapsto \mu_t$ is narrowly continuous in $P_{\Omega}(\mathcal{X})$ by Proposition \ref{prop:narrow convergence_vs_Ewpp}. But, the sample paths $t \mapsto \mu_t^\omega$ need not be narrowly continuous in $P(\mathcal{X})$.  
        \\
        \textbf{Step 0 (Construction of $\{\pi_n\}_{n \in \mathbb{N}} \subset P_{\Omega}(\Gamma)$)}.
        We carry out construction \hyperref[itm:ConstructionB_random]{{\footnotesize$\bigstar$}} as follows. For each $n \in \mathbb{N}$, we take the dyadic partition of the time interval $[0,1]$ with points denoted by $t_i^{(n)} \coloneqq i/2^n$, where $i \in \{0,\cdots, 2^n \eqqcolon N_n\}$.
        %The assumption $\mu_{t} \in P_{p,\Omega} (\mathcal{X})$ implies that $\mu_t^\omega \in P_{p} (\mathcal{X})$ outside a $\mathbb{P}$-zero set depending on $t$.
        The compatibility assumption in the sense of Definition \ref{def:compatibility_random} implies that there is a random multi-coupling on the product space
        $\bm{\mathcal{X}}_n \coloneqq \mathcal{X}_0 \times \mathcal{X}_1 \times \cdots \times \mathcal{X}_{N_n}$
        with $\{ \mathcal{X}_i \}$ representing copies of $\mathcal{X}$,
        $$\Upsilon_n \in P_{\Omega}(\bm{\mathcal{X}}_n)$$ 
        such that, outside a $\mathbb{P}$-zero set denoted by $\mathrm{Z}_n$, all of its two-dimensional marginals are optimal, namely
    \begin{equation}\label{eq:2D_marginal_Upsilon_omega_all}
        (\mathrm{Pr}^{i,j})_{\#} \Upsilon_n^\omega
        \in
        \mathrm{OptCpl}
        \bigl(
            \mu^\omega_{t_i^{(n)}},
            \mu^\omega_{t_j^{(n)}}
        \bigr)
        \qquad
        \forall i,j \in \{0,\ldots,N_n\}.
    \end{equation}
    In particular, we have
         \begin{gather}\label{eq:2D_marginal_Upsilon_omega}
            (\text{Pr}^{i,i+\frac{2^n}{2^m}})_{\#} \Upsilon_{n}^\omega \in \text{OptCpl}\big(\mu^\omega_{t^{(n)}_{i}},\mu^\omega_{t^{(n)}_{i+\frac{2^n}{2^m}}}\big) , 
        \end{gather}
        for all $i \in \big\{k \frac{2^n}{2^m} \big| k \in \{ 0,1,\cdots, 2^m-1\} \big\} $ and $ m \in \{0,1,\cdots n\}$,
        outside $\mathrm{Z}_n$. (For all $\omega$ inside $\mathrm{Z}_n$, we can let $\Upsilon^\omega_{n}$ be an arbitrary multi-coupling, take the product measure, for instance.
        As we will see and as explained in Remark \ref{rmk:zero_set}, 
        the definition of $\Upsilon^\omega_{n}$ on the $\mathbb{P}$-zero measure set plays no role here and we even do not need to do the completion of the $\sigma$-algebra $\mathcal{F}$). 
        We have so far a sequence of random measures $(\Upsilon_n)_{n \in \mathbb{N}}$ whose definition involves different $\mathbb{P}$-zero sets. 
        Let $\mathrm{Z}$ denote the countable union of all $\mathrm{Z}_n$s, which is again a $\mathbb{P}$-zero measure set. From now on, we shall only care about $\omega \in \Omega \setminus \mathrm{Z}$. We build a sequence of path measures using a measurable geodesic selection and interpolation map 
        \begin{equation}\label{eq:def_geod_interp_0_r_paper}
            \ell : \, \bm{x}=(x_0,\cdots ,x_{N_n}) \, \in \, \bm{\mathcal{X}}_n \mapsto \ell_{\bm{x}} \in C(I;\mathcal{X})
        \end{equation}
        defined by
        \begin{align}\label{eq:def_geod_interp_r_paper}
            \ell_{\bm{x}}(t) & \coloneqq x_i, \quad t = t^{(n)}_i, \quad  i \in \{ 0,1, \cdots, N_n \};
        \end{align}
        and connecting with constant-speed geodesics in between.
        Such a Borel measurable selection map exists under the assumptions of this theorem; see Remark \ref{rmk:measurable_geodesic_selection}.
        We set
        \begin{equation}\label{eq:pi_n_Upsilon_random}
             \pi_n^\omega \coloneqq (\ell)_{\#} \Upsilon_n^\omega \in P (\Gamma), \quad \forall n \in \mathbb{N}, \quad \forall \omega \in \Omega \setminus \mathrm{Z}. 
        \end{equation}     
         As a result, it is easy to check that
        \begin{align}
        \pi_n: \mathcal{B}(\Gamma) \times \Omega & \to \, [0,1] \\
        (A, \omega) \,  & \mapsto \,  \pi_n^{\omega}  (A)
       \end{align}
       defines a random probability measure. Indeed, for any $\omega$, $\pi_n^\omega$ is a probability measure in $P (\Gamma)$. It remains to verify that for fixed $A \in \mathcal{B}(\Gamma)$, the function $\omega \mapsto \pi_n^{\omega}  (A) $ is $\mathcal{F}$-$\mathcal{B}([0,1])$-measurable. Take an arbitrary set $B \in \mathcal{B}([0,1])$. By definition of the push-forward measures, we have
       \begin{align}
           \Big\{ \omega \in \Omega:  \pi_n^{\omega}  (A) \in B \Big\}
           & = \Big\{ \omega \in \Omega :  \Upsilon_n^\omega \big(\ell^{-1} (A) \big) \in B \Big\} \\
           & = \Big\{ \omega \in \Omega :  \Upsilon_n^\omega \big( \{ \bm{x} \in  \bm{\mathcal{X}}_n : \ell_{\bm{x}} \in A \}  \big)  \in B \Big\} \in \mathcal{F},
       \end{align}
       which is indeed in $\mathcal{F}$ because $\Upsilon_n$ is a random probability measure.
       In summary, we have constructed a sequence of random  path measures
       $$ \pi_n \in  P_{\Omega}(\Gamma), \quad \forall n \in \mathbb{N}.$$
        \textbf{Step 1 (Tightness of $\{\pi_n\}_{n \in \mathbb{N}} \subset P_{\Omega}(\Gamma)$).} We would like to show that the family of random measures $\{ \pi_n\}$ is tight in $P_\Omega(C([0,1];\mathcal{X}))$. To this end,  we use the tightness condition formulated in Corollary \ref{prop:tightness_balphap_random}. Our goal is to show 
        \begin{equation}\label{eq:tightness_computation_random_claim}
             \sup_{n \in \mathbb{N}} \mathbb {E} \left[ \int_{\Gamma} \Big(  d(\gamma_0, \bar{x})^p +  | \gamma |^p_{b^{\alpha,p}} \Big) \d \pi_n (\gamma) \right] < +\infty. 
        \end{equation}
        In what follows, we compute the expectation by ignoring the $\mathbb{P}$-zero set $\mathrm{Z}$.
        For the first term, we have
    \begin{equation}
        \mathbb{E} \left[  \int_{\Gamma} d(\gamma_0, \bar{x})^p \d \pi_n (\gamma) \right]
        %& = \int_{\Gamma} d(e_0 (\gamma), \bar{x})^p \d \pi_n (\gamma) \\
        %& = \int_{\bm{\mathcal{X}}_n} d(e_0 (\ell_{\bm{x}}), \bar{x})^p \d \Upsilon_n (\bm{x}) \\
        = \mathbb{E} \left[ \int_{\bm{\mathcal{X}}_n} d(x_0, \bar{x})^p \d \Upsilon_n (\bm{x}) \right]
        = \mathbb{E} \left[ \int_{\mathcal{X}} d (x,\bar{x})^p \d \mu_0 (x) \right]  < + \infty \label{eq:tightness_computation_random_step_00}, 
    \end{equation}
    which is indeed finite by assumption.
    \\
    For the second term, we estimate 
    \begin{align}
        \sup_{n \in \mathbb{N}} & \, \mathbb {E}  \left[ \int_{\Gamma}   | \gamma |^p_{b^{\alpha,p}} \d \pi_n (\gamma) \right]
         = \sup_{n \in \mathbb{N}} \mathbb {E} \left[ \int_{\bm{\mathcal{X}}_n}  | \ell_{\bm{x}} |^p_{b^{\alpha,p}} \d \Upsilon_n (\bm{x}) \right]
        \leq  \mathbb {E} \left[ \sup_{n \in \mathbb{N}} \int_{\bm{\mathcal{X}}_n}  | \ell_{\bm{x}} |^p_{b^{\alpha,p}} \d \Upsilon_n (\bm{x}) \right] \\
        & \overset{(a)}{=} \mathbb {E} \left[ \sup_{n \in \mathbb{N}} \left(\sum_{m=0}^{n} 2^{m(\alpha p -1)} \,\, \sum_{k=0}^{2^m-1} W_p^p( \mu_{t^{(m)}_{k}},\mu_{t^{(m)}_{k+1}})  + \frac{2^{n (\alpha p - 1 )}}{2^{(p - \alpha p)}-1}  \sum_{i=0}^{2^n-1} W_p^p (\mu_{t^{(n)}_{i}},\mu_{t^{(n)}_{i+1}}) \right) \right]\\
        & \overset{(b)}{\leq} \mathbb {E} \left[ \frac{1}{ 1- 2^{-(p - \alpha p)}} \, | \mu |^p_{b^{\alpha,p}} \right] \\
        & \overset{(c)}{=} \frac{1}{ 1- 2^{-(p - \alpha p)}} \Vert \mu \Vert^p_{b^{\alpha,p}}  \\
        & \overset{(d)}{\leq} \frac{c_2(\alpha,p)^p}{ 1- 2^{-(p - \alpha p)}} \,  \Vert \mu \Vert^p_{W^{\alpha,p}} < + \infty \label{eq:tightness_computation_random_step_11}. 
    \end{align}
    In step ($a$), we used \cite[Lemma 2.31]{Abedi2025paths}, which computes the $b^{\alpha,p}$-norm of piecewise geodesic curves, and the optimality of the two-dimensional marginals of $\Upsilon_n$ in \eqref{eq:2D_marginal_Upsilon_omega}.
    The estimate ($b$) follows from the structure of the $b^{\alpha,p}$-norm. 
    Steps ($c$) and ($d$) follow from Lemma \ref{lemma:Exp_balphap_mu} and Theorem \ref{thm:Walphap_balphap_r_paper}, respectively. The last quantity is finite by assumption. 
    \\
    In contrast to the deterministic case, here we do not use the estimate involving $W^{\alpha, p}$-norm immediately because $\omega \mapsto |\mu^\omega|_{W^{\alpha,p}}$ may not be measurable and its expectation may not be defined. However, $\omega \mapsto |\mu^\omega|_{b^{\alpha,p}}$, as a function that only involves a countable sum, is always measurable.
    Here $W^{\alpha, p}$-norm appears in the last step, where we used the equivalence between these norms on the space of continuous paths on the metric space $(P_{p,\Omega}(\mathcal{X}), \mathbb{W}_p)$. 
    \\
    As the bound \eqref{eq:tightness_computation_random_claim} holds, we conclude that the family of random measures $\{\pi_n \}_n \subset P_{\Omega} (\Gamma) $ is tight in the sense of Definition \ref{def:tightness_random}. Then Prokhorov Theorem \ref{thm:Prokhorov_random} for random measures implies that the set  $\{\pi_n\}_{n\in \mathbb{N}}$ is relatively (sequentially) compact with respect to the narrow topology of $P_\Omega (\mathcal{X})$, i.e., there exists a subsequence $\{\pi_{n_k}\}_{k\in \mathbb{N}} $ such that $\pi_{n_k} \to \pi$ narrowly on $P_\Omega(\Gamma)$ as $k \to \infty$ to a limit point $\pi \in P_\Omega (\Gamma)$. In particular, we emphasize that $\omega \mapsto \pi^\omega$ is measurable. 
    \\
    \textbf{Step 2 ($\pi^\omega$ is concentrated on $W^{\alpha,p}(I;\mathcal{X})$ $\mathbb{P}$-a.s.).} By lower semi-continuity of the map $\gamma \mapsto |\gamma|^p_{W^{\alpha,p}}$ and \eqref{eq:narrow_conv_lsc_func_random}, we have
        \begin{align}
             \mathbb{E} \left[ \int_{\Gamma}|\gamma|_{W^{\alpha, p}}^p \d \pi (\gamma) \right]
             & \leq \liminf_{k \to \infty }  \mathbb{E} \left[  \int_{\Gamma}  | \gamma |^p_{W^{\alpha,p}} \d \pi_{n_k} (\gamma) \right] \\
             & \leq \frac{1}{c_1(\alpha,p)^p} \, \liminf_{k \to \infty }  \mathbb{E} \left[  \int_{\Gamma}  | \gamma |^p_{b^{\alpha,p}} \d \pi_{n_k} (\gamma) \right] < + \infty, \label{eq:support_computation_random}
        \end{align}
       where the integral on the right-hand side has been shown, in the previous step, to be bounded independent of $n_k$.
       Therefore, we conclude that for $\mathbb{P}$-a.e. $\omega \in \Omega$,
       $$
       \int_{\Gamma} |\gamma|_{W^{\alpha, p}}^p \d \pi^{\omega} (\gamma) < + \infty,
       $$
       and thus
       $$
            |\gamma|_{W^{\alpha, p}} < + \infty \quad \textrm{for } \pi^\omega \textrm{-a.e. } \gamma \in \Gamma.
       $$ 
    \\
    \textbf{Step 3 ($(e_t)_\#\pi^\omega=\mu_t^\omega$ $\mathbb{P}$-a.s. for all $t\in I$).}
    The marginal of the random path measure $\pi \in P_{\Omega} (\Gamma)$ under the evaluation map $e_t : \Gamma \to \mathcal{X}$ defines a new random measure, that we denote by $\nu_t \in P_\Omega (\mathcal{X})$, i.e., for a given $t$ and $\omega$, we put 
    \begin{equation}
        \nu_t^\omega \coloneqq (e_t)_\#\pi^\omega.  
    \end{equation}
    Our goal is to prove that    \begin{equation}\label{eq:mu_nu_modification}
    \mathbb{P} \Big( \left\{ \omega \in \Omega: \,  \nu_t^\omega =\mu_t^\omega  \right\} \Big)  = 1 \quad \textrm{for each } \, t \in I,
    \end{equation}
    or, in the terminology of stochastic analysis, they are modifications of each other. To prove that two random measures are equal a.s., it is enough by \cite[Lemma 3.14]{Crauel2002} to show that
    \begin{equation}
        \int_{\Omega} \int_{\mathcal{X}} \mathds{1}_{F} (\omega) \phi(x) \d \nu_t^\omega (x) \d \mathbb{P} (\omega) = \int_{\Omega} \int_{\mathcal{X}}  \mathds{1}_{F} (\omega) \phi (x) \d \mu_t^\omega (x) \d \mathbb{P} (\omega)
    \end{equation}
    holds for all $\phi \in \mathrm{Lip}_b (\mathcal{X})$ and $F \in  \mathcal{F}$. We have 
    \begin{align}
        \int_{\Omega} \int_{\mathcal{X}} \mathds{1}_{F} (\omega) \phi(x) \d ({e_t}_{\#}\pi^\omega) (x) \d \mathbb{P} (\omega)
        & = \int_{\Omega} \int_{\Gamma} \mathds{1}_{F} (\omega) \phi (\gamma_t) \d \pi^\omega (\gamma) \d \mathbb{P} (\omega) \, \\
        & \overset{(a)}{=}  \lim_{k \to \infty}  \int_{\Omega} \int_{\Gamma} \mathds{1}_{F} (\omega) \phi (\gamma_{[2^{n_{k}} t]/2^{n_{k}} })  \d \pi^\omega_{n_k} (\gamma)  \d \mathbb{P} (\omega) \\
        & \overset{(b)}{=}  \lim_{k \to \infty} \int_{\Omega} \int_{\bm{\mathcal{X}}_{n_k}} \mathds{1}_{F} (\omega) \phi(x_{[2^{n_k} t]}) \d \Upsilon^\omega_{n_k} (\bm{x}) \d \mathbb{P} (\omega) \\
        & \overset{(c)}{=} \lim_{k \to \infty}  \int_{\Omega} \int_{\mathcal{X}} \mathds{1}_{F} (\omega) \phi (x) \d \mu^\omega_{[2^{n_k} t]/2^{n_k} } (x) \d \mathbb{P} (\omega) \\
        & \overset{(d)}{=} \int_{\Omega} \int_{\mathcal{X}} \mathds{1}_{F} (\omega) \phi (x) \d \mu_t^\omega (x) \d \mathbb{P} (\omega). \label{eq:proof_random_1dim_time_marginal}
    \end{align}
    As for the step $(a)$, observe that
    \begin{multline}
        \bigg| \iint \mathds{1}_{F} (\omega) \phi (\gamma_t)  \d \pi^\omega \d \mathbb{P}   -  \iint \mathds{1}_{F} (\omega)  \phi (\gamma_{[2^{n_k} t]/2^{n_k} } )  \d \pi^\omega_{n_k} \d \mathbb{P}  \bigg|  \\
         \leq  \bigg|   \iint \mathds{1}_{F} (\omega)  \phi (\gamma_t)  \d \pi^\omega \d \mathbb{P}   -  \iint \mathds{1}_{F} (\omega)  \phi (\gamma_t)  \d \pi^\omega_{n_k}  \d \mathbb{P}  \bigg| \\
         + \bigg|  \iint \mathds{1}_{F} (\omega)  \phi (\gamma_t)  \d \pi^\omega_{n_k}  \d \mathbb{P}   -  \iint \mathds{1}_{F} (\omega)  \phi (\gamma_{[2^{n_k} t]/2^{n_k} } )  \d \pi^\omega_{n_k} \d \mathbb{P} \bigg|. \label{eq:convergence_1_dim_marginal_random}
    \end{multline}
    By taking limit $k \to \infty$, the first term on the right-hand side goes to zero by the narrow convergence of $(\pi_{n_k})_k$ to $\pi$ in $P_\Omega(\mathcal{X})$ (note that the map $(\gamma, \omega) \mapsto \phi (\gamma_t) \mathds{1}_{F} (\omega)$ from $\Gamma \times \Omega \to \mathbb{R} $ is indeed a random continuous bounded function).
    To show that the second term also vanishes in the limit, we further estimate 
    \begin{align}    
        \bigg|  \iint \mathds{1}_{F} (\omega) \phi (\gamma_t)  \d \pi^\omega_{n_k}  \d \mathbb{P} & - \iint  \mathds{1}_{F} (\omega) \phi (\gamma_{[2^{n_k} t]/2^{n_k} } )  \d \pi^\omega_{n_k}  \d \mathbb{P}  \bigg| \\
        & \leq \mathrm{Lip} (\phi) \iint \mathds{1}_{F} (\omega) d (\gamma_t, \gamma_{[2^{n_{k}} t]/2^{n_{k}} } )  \d \pi^\omega_{n_k}  \d \mathbb{P}  \\
        & \leq \mathrm{Lip} (\phi)  \bar{c}(\alpha,p) \left(t- \frac{[2^{n_{k}} t]}{2^{n_{k}} }\right)^{\alpha - \frac{1}{p}} \left(\mathbb{E} \left[ \int | \gamma |^p_{W^{\alpha,p}} \d \pi_{n_k} \right] \right)^{\frac1p}, \label{eq:convergence_1_dim_marginal_random_partII}
     \end{align}
     where we first bounded the expression using the Lipschitz constant of the test function, and then used the fact that for $\mathbb{P}$-a.e. $\omega$, the measure $\pi_{n_k}^\omega$ is concentrated on continuous curves with finite $W^{\alpha,p}$-semi-norm, and thus $d (\gamma_t, \gamma_{[2^{n_{k}} t]/2^{n_{k}} } )$ can be estimated using \eqref{eq:distance_estimate_GRR_r_paper}. The integral on the right-hand side of \eqref{eq:convergence_1_dim_marginal_random_partII} is uniformly bounded by Step 1. Thus, in the limit $k \to \infty$, the last expression approaches zero. 
     \\
     Steps $(b)$-$(c)$ simply follow from the construction. For fixed $\omega$ and $k$, the 1-D marginals of $\pi^\omega_{n_k}$ at time $t$ coincide with $\mu_t^\omega$ whenever $t$ is of the form $t = \frac{i}{2^{n_k}}$ for some integer $i$. 
    \\
    Finally, step $(d)$ follows from the fact that $t \mapsto \mu_t$ is a narrowly continuous curve in $P_{\Omega}(\mathcal{X})$, as we pointed out at the beginning of the proof. This means that the application
    $$t \mapsto \int_{\Omega}   \int_{\mathcal{X}} \mathds{1}_{F} (\omega) \phi (x) \d \mu_t^\omega \d \mathbb{P} $$
    is continuous. Note that the application
    $$
    t \mapsto   \int_{\mathcal{X}} \phi (x)  \d \mu_t^\omega 
    $$
    need not be continuous a.s. On the contrary, in the next step, we prove that 
    $$
    t \mapsto   \int_{\mathcal{X}} \phi (x)  \d \nu_t^\omega 
    $$
    is continuous a.s.
    \\
    \textbf{Step 4 ($t \mapsto  {(e_t)}_{\#} \pi^\omega \in W^{\alpha, p} (I;P_p(\mathcal{X}))$ $\mathbb{P}$-a.s.).} Eq. \eqref{eq:support_computation_random} together with \eqref{eq:mu_nu_modification} at $t=0$ gives 
    \begin{equation}
        \mathbb{E} \left[  \int_{\Gamma_T} \Big( d(\gamma_0,\bar{x})^p + | \gamma |_{W^{\alpha,p}}^p  \Big) \d \pi(\gamma) \right] < + \infty,
    \end{equation}
    which is exactly the integrability condition \eqref{eq:lift_to_mu_Walphap_integrability_stoch}. Hence,  Theorem \ref{thm:lift_to_mu_Walphap_stoch} applies and we obtain: 
    \begin{itemize}[leftmargin=*]
        \item[$\diamond$] $(\nu_t) \in W^{\alpha, p} (I;P_{p,\Omega}(\mathcal{X}))$ and thus $t \mapsto \nu_t $ is narrowly continuous in $P_{\Omega}(\mathcal{X})$ by Proposition \ref{prop:narrow convergence_vs_Ewpp}.
        \item[$\diamond$] $(\nu_t^\omega) \in W^{\alpha, p} (I;P_{p}(\mathcal{X}))$ a.s. and thus $t \mapsto \nu_t^\omega $ is narrowly continuous in $P(\mathcal{X})$ a.s.
    \end{itemize} 
    \textbf{Step 5 ($(e_s,e_t)_\# \pi^\omega \in \mathrm{OptCpl}(\mu_s^\omega, \mu_t^\omega)$ $\mathbb{P}$-a.s. for all $s,t \in I$).}
    Here, our goal is to prove that 
    \begin{equation}
    \mathbb{P} \Big( \big\{ \omega \in \Omega: \,  W_p^p(\mu_t^\omega,\mu_s^\omega) = \int_{\Gamma} d(\gamma_t,\gamma_s)^p \d \pi^\omega  \big\} \Big)  = 1 \quad \textrm{for  each } \, s,t \in I.
    \end{equation}
    Fix $s,t \in I$.
    By previous step, we have a $\mathbb{P}$-full measure set on which  $(e_s,e_t)_\# \pi^\omega \in \mathrm{Cpl}(\mu_s^\omega, \mu_t^\omega)$. By assumption, we have a $\mathbb{P}$-full measure set on which $\mu_s^\omega, \mu_t^\omega \in P_p(\mathcal{X})$. Thus, we have 
    \begin{equation}
        W_p^p ( \mu_t^\omega , \mu_s^\omega ) \leq  \int_\Gamma d (\gamma_t,\gamma_s)^p \d \pi^\omega (\gamma) \,\,\,\, \text{a.s.}
    \end{equation}
    Integrating this together with a test function of the form $\omega \mapsto \mathds{1}_{F}(\omega)$ with $F \in \mathcal{F}$, we obtain 
    \begin{equation}\label{eq:proof_opti_inequ1_random}
            \int_{\Omega} \mathds{1}_{F}(\omega) W_p^p ( \mu_t^\omega , \mu_s^\omega ) \d \mathbb{P} (\omega) \leq  \int_{\Omega} \int_\Gamma \mathds{1}_{F}(\omega) d (\gamma_t,\gamma_s)^p \d \pi^\omega (\gamma) \d \mathbb{P} (\omega).
        \end{equation}
    If we show that the reverse inequality holds as well, the proof of this step is complete. (Recall the following:  If $X$ and $Y$ are two real-valued integrable random variables defined on $(\Omega, \mathcal{F}, \mathbb{P})$ such that $\int_{F} X \d \mathbb{P} = \int_{F} Y \d \mathbb{P}$ for all $F \in \mathcal{F}$, then $X = Y$ a.s.) As any random measure $\pi^{\omega}$ can be viewed as the disintegration of a non-random measure $\bm{\pi} \in P (\Gamma \times \Omega)$, as defined in \ref{eq:bold_mu_def}, we can write 
        \begin{align}
            \int_{\Omega} \int_\Gamma \mathds{1}_{F} (\omega) d (\gamma_t,\gamma_s)^p \d \pi^\omega & (\gamma)  \d \mathbb{P} (\omega)
            = \int_{\Gamma \times \Omega} \mathds{1}_{F} (\omega) d (\gamma_t,\gamma_s)^p \d \bm{\pi} (\gamma, \omega)\\
            & \overset{(a)}{\leq }  \liminf_{m \to \infty}  \int_{\Gamma \times \Omega} \mathds{1}_{F} (\omega) d (\gamma_{[2^m t]/2^m}, \gamma_{[2^m s]/2^m} )^p \d \bm{\pi} (\gamma, \omega) \\
             & \overset{(b)}{\leq }  \liminf_{m \to \infty} \left( \liminf_{k \to \infty} \int_\Omega \int_{\Gamma}  \mathds{1}_{F} (\omega)  d (\gamma_{[2^m t]/2^m}, \gamma_{[2^m s]/2^m} )^p  \d \pi^\omega_{n_k}  \d \mathbb{P}(\omega) \right) \\
             & \overset{(c)}{=} \liminf_{m \to \infty} \int_\Omega \mathds{1}_{F} (\omega) W_p^p \big( \mu^\omega_{[2^m t]/2^m} , \mu^\omega_{[2^m s]/2^m} \big) \d \mathbb{P} (\omega) \\
            & \overset{(d)}{=} \int_\Omega \mathds{1}_{F}(\omega) W_p^p ( \mu_t^\omega , \mu_s^\omega ) \d \mathbb{P} (\omega).
            \label{eq:proof_opti_inequ2_random}
        \end{align}
    In step $(a)$, we applied Fatou's lemma on the measure space $(\Gamma \times \mathbb{P}, \mathcal{B}(\Gamma) \otimes \mathcal{F}, \bm{\pi})$.
    Clearly, for any $(\gamma,\omega) \in \Gamma \times \Omega$,
    \begin{equation}
        \mathds{1}_{F} (\omega) d (\gamma_t,\gamma_s)^p = \lim_{m \to \infty} \mathds{1}_{F} (\omega) d (\gamma_{[2^m t]/2^m}, \gamma_{[2^m s]/2^m} )^p .
    \end{equation}
    To achieve step $(b)$, we first revert the integral back to its disintegrated form.
    Next, we note that for fixed $m$ in the integrand, the function
    $$
    (\gamma, \omega) \mapsto \mathds{1}_{F} (\omega) d (\gamma_{[2^m t]/2^m}, \gamma_{[2^m s]/2^m} )^p
    $$
    is continuous in $\gamma$ (for any $\omega$), in particular, it is lower-semi continuous in $\gamma$ (for any $\omega$).
    Thus, this step is due to the narrow convergence of the random measures $(\pi_{n_k})_k$ to $\pi$ and \eqref{eq:narrow_conv_lsc_func_random}.
    \\
    Step $(c)$ is a direct consequence of the construction of $\pi_{n_k}$ and the compatibility assumption (here we use the all-pairs optimality condition \eqref{eq:2D_marginal_Upsilon_omega_all}). Ignoring the zero sets in the integration again, we notice that once $n_k$ gets larger than $m$, the measure $\pi_{n_k}^\omega$ gives us the optimal coupling between any two measures at times of the form $i/2^m$ with $i$ being an integer, and hence the value of the integral does not change anymore. 
    \\
    Finally, step $(d)$ is due to the fact that the curve $t \mapsto \mu$ is a continuous curve $P_{p, \Omega}(\mathcal{X})$, making 
    $$(t,s) \mapsto \int_{\Omega} \mathds{1}_{F}(\omega) W_p^p ( \mu_t^\omega , \mu_s^\omega ) \d \mathbb{P} (\omega) $$
    continuous.
    Indeed, for any sequence $(t_m,s_m) \to (t,s)$, we can estimate
    \begin{multline}
        \left| \int_{\Omega} \mathds{1}_{F}(\omega) W_p^p ( \mu_{t_m}^\omega , \mu_{s_m}^\omega ) \d \mathbb{P} (\omega) - \int_{\Omega} \mathds{1}_{F}(\omega) W_p^p ( \mu_{t}^\omega , \mu_{s}^\omega ) \d \mathbb{P} (\omega) \right| \\
        \leq \lVert W_p^p ( \mu_{t_m} , \mu_{s_m} ) - W_p^p ( \mu_{t} , \mu_{s} ) \rVert_{L^1(\mathbb{P})}
    \end{multline}
    where the right-hand side approaches zero because $W_p ( \mu_{t_m} , \mu_{s_m} ) \overset{L^p(\mathbb{P})}{\rightarrow}   W_p ( \mu_{t} , \mu_{s} )$.
    \\
    \textbf{Step 6 ($\Vert\mu\Vert_{W^{\alpha,p}}^p = \mathbb{E} [ \int |\gamma|_{W^{\alpha,p}}^p \d \pi ] $).} Choosing $F = \Omega$ in the previous step, we obtain 
    $$
    \mathbb{W}^p_p (\mu_t,\mu_s) = \mathbb{E} \left[ \int_{\Gamma} d (\gamma_t,\gamma_s)^p \d \pi \right] 
    $$
    Therefore, we can compute:
        \begin{align}
             \Vert\mu\Vert_{W^{\alpha,p}}^p
             & \coloneqq \iint_{[0,1]^2} \frac{\mathbb{W}_p^p (\mu_s,\mu_t)}{|t-s|^{1+\alpha p}} \d s \d t \\
             & = \iint_{[0,1]^2} \frac{\mathbb{E}\left[\int_\Gamma d(\gamma_s,\gamma_t)^p \d \pi (\gamma)\right]}{|t-s|^{1+\alpha p}} \d s \d t \\
             & = \mathbb{E} \left[ \int_\Gamma |\gamma|_{W^{\alpha,p}}^p \d \pi(\gamma)\right] < + \infty, 
         \end{align}
        where we used Tonelli's theorem. Similarly, we have
         \begin{align}
             \Vert\mu\Vert_{b^{\alpha,p}}^p
             & \coloneqq \sum_{m=0}^{\infty} 2^{m(\alpha p -1)} \sum_{k=0}^{2^m-1} \mathbb{W}_p^p (\mu_{t_{k}^{(m)}}, \mu_{t_{k+1}^{(m)}}) \\
             & = \lim_{M \to \infty} \int_{\Gamma \times \Omega} \sum_{m=0}^{M} 2^{m(\alpha p -1)} \sum_{k=0}^{2^m-1} d (\gamma_{t_{k}^{(m)}}, \gamma_{t_{k+1}^{(m)}})^p \d \bm{\pi} (\gamma,\omega) \\
             & = \int_{\Gamma \times \Omega} \sum_{m=0}^{\infty} 2^{m(\alpha p -1)} \sum_{k=0}^{2^m-1} d (\gamma_{t_{k}^{(m)}}, \gamma_{t_{k+1}^{(m)}})^p \d \bm{\pi} (\gamma,\omega) \\
             & = \mathbb{E} \left[ \int_\Gamma |\gamma|_{b^{\alpha,p}}^p \d \pi (\gamma) \right] < + \infty,
         \end{align}
         where we used Beppo Levi's lemma. This completes the proof. 
    \end{proof}

    \begin{remark}\label{rmk:random_curves}
        If in addition to the assumptions of  Theorem \ref{thm:optimal_lift_mu_Walphap_compatible_random}, we have that $\mu$ is given by a measurable map $\mu: (\Omega,\mathcal{F},\mathbb{P}) \to C(I;P_p(\mathcal{X}))$, then the second and third statements of the theorem can be strengthen as follows:
        \textit{
         \begin{enumerate}[label=(\roman*), font=\normalfont]
            \item[(ii)] $(e_t)_\#\pi^\omega =\mu_t^{\omega} $ for all $t\in I$ a.s.;
            \item[(iii)] $(e_s,e_t)_\# \pi^\omega \in \mathrm{OptCpl}(\mu_s^{\omega}, \mu_t^{\omega})$ for all $s,t \in I$ a.s.; and,
            \begin{equation}\label{eq:rmk_random_curves}
            |\mu^\omega|_{W^{\alpha,p}}^p = \int_{\Gamma_T} |\gamma|_{W^{\alpha,p}}^p \d \pi^\omega(\gamma) \quad \textrm{a.s}.
            \end{equation}
        \end{enumerate}
        }
        \noindent
        Indeed, if two continuous processes are modifications of each other, they are indistinguishable. To see \eqref{eq:rmk_random_curves}, we note that since $\mu: (\Omega,\mathcal{F},\mathbb{P}) \to C(I;P_p(\mathcal{X}))$ is a measurable map, we can exchange the integrals and write: $\Vert\mu\Vert_{W^{\alpha,p}}^p = \mathbb{E}[|\mu|_{W^{\alpha,p}}^p]$ (which is not true in general as discussed in Remark \ref{rmk:Exp_Walphap_mu}). This together with \eqref{eq:optimality_pi_1_random} yields
        \begin{align}
             \mathbb{E} \left[\int_{\Gamma_T} |\gamma|_{W^{\alpha,p}}^p \d \pi (\gamma) \right] - \Vert\mu\Vert_{W^{\alpha,p}}^p = \mathbb{E} \Big[ \underbrace{\int_{\Gamma_T} |\gamma|_{W^{\alpha,p}}^p \d \pi - |\mu|_{W^{\alpha,p}}^p}_{\geq 0} \Big] = 0. 
        \end{align}
        As $\pi^\omega$ is a lift of $(\mu_t^\omega)$ a.s., we know by \cite[Theorem 1.2]{Abedi2025paths} that the quantity inside the expectation is non-negative. Thus, it must be zero a.s.  
    \end{remark}

    \begin{corollary}\label{thm:optimal_lift_mu_Holder_compatible_random}
        Let $(\mathcal{X}, d)$ be a complete, separable, and locally compact length metric space, and $I \coloneqq [0,T] \subset \mathbb{R}$.
        Let $(\mu_t)_{t \in I}$ be a probability measure-valued stochastic process defined on a probability space $(\Omega, \mathcal{F}, \mathbb{P})$ such that $(\mu_t) \in C^{\upgamma\textrm{-} \mathrm{H\ddot{o}l}}(I;P_{p,\Omega}(\mathcal{X}))$ for some $1<p<\infty$ and $\frac{1}{p}<\upgamma \leq 1$.
        Assume that $(\mu_t)_{t \in I}$ is compatible in $P_{p,\Omega}(\mathcal{X})$.
        Then, construction \hyperref[itm:ConstructionB_random]{{\footnotesize$\bigstar$}} converges narrowly (up to a subsequence) to a random probability measure $\pi \in P_{\Omega} (C(I;\mathcal{X}))$ satisfying
    \begin{enumerate}[label=(\roman*), font=\normalfont]
                \item $\pi^\omega$ is concentrated on $W^{\alpha,p}(I;\mathcal{X}) \subset  C^{(\alpha - \frac{1}{p})\textrm{-} \mathrm{H\ddot{o}l}}(I;\mathcal{X}) $ a.s. and
                \\
                $t\mapsto (e_t)_\#\pi^\omega$ is in $W^{\alpha,p} (I;P_p(\mathcal{X}))\subset C^{(\alpha-\frac{1}{p})\textrm{-} \mathrm{H\ddot{o}l}}$ \, $(I;P_p(\mathcal{X}))$ a.s. for any $\alpha \in (\frac{1}{p},\upgamma)$;
                \item $(e_t)_\#\pi^\omega=\mu_t^\omega$ a.s. for all $t\in I$;
                % W^{\alpha,p}(I;\mathcal{X}) \subset 
                % W^{\alpha,p}(I;P_p(\mathcal{X})) \subset
                \item $(e_s,e_t)_\# \pi^\omega \in \mathrm{OptCpl}(\mu_s^\omega, \mu_t^\omega)$ a.s. for all $s,t \in I$; and for any $\alpha \in (\frac{1}{p},\upgamma)$, we have  \eqref{eq:optimality_pi_1_random} and 
                \begin{align}
                   \Vert \mu \Vert_{\upgamma\textrm{-}\mathrm{H\ddot{o}l}}^p \geq c \, \mathbb{E} \left[ \int_{\Gamma_T} | \gamma |^p_{\alpha - \frac{1}{p}\textrm{-}\mathrm{H\ddot{o}l}} \d \pi (\gamma) \right] \geq c \, \Vert \mu \Vert_{\alpha-\frac{1}{p}\textrm{-}\mathrm{H\ddot{o}l}}^p,
                \end{align}
                where $c = c(\upgamma,\alpha,p,T)$ is an explicit positive constant. In addition, for any $\alpha \in (\frac{1}{p},\upgamma)$, we have
        \begin{equation}
             \Vert \mu \Vert_{\upgamma\textrm{-}\mathrm{H\ddot{o}l}}^p \geq c \, \mathbb{E} \left[ \int_{\Gamma_T} | \gamma |^p_{\frac{1}{\alpha} \textrm{-}\mathrm{var}} \d \pi (\gamma) \right] \geq c \, \Vert \mu \Vert_{p\textrm{-}\mathrm{var}}^p,
        \end{equation}
        where $c = c(\upgamma,\alpha,p,T)$ is another explicit positive constant.
        \end{enumerate}
    \end{corollary}

    \begin{remark}
    We note that the random paths $t \mapsto \mu_t^\omega$ and $t \mapsto (e_t)_{\#} \pi^\omega$ do not necessarily coincide a.s.
    In the deterministic case, the curve $((e_t)_\#\pi)$ coincides with $(\mu_t)$, making $\pi$ a \emph{lift} by its very definition. Here, item (ii) states that the measure-valued process $((e_t)_\#\pi)$ is a modification $(\mu_t)$, making $\pi$ a \emph{random lift} in the sense of \eqref{eq:def:random_lift_intro}. 
    Item (i) tells us that the sample paths of this modification lie a.s. in the Wasserstein space and are a.s. continuous. Thus, on the level of Wasserstein space, this is similar to Kolmogorov--\v Centsov continuity theorem. The additional statement is that the modification here arises as the one-dimensional time marginals of a random path measure $\pi$ on the underlying space with optimality (iii), both of which are constructed \emph{at the same time}. 
    This avoids the issue of losing the regularity twice, as raised in Remark \ref{rmk:treatment}. Furthermore, the construction ensures the measurability of $\omega \mapsto \pi^\omega$.
    \end{remark}
    
    \begin{proof}[Proof of Corollary \ref{thm:optimal_lift_mu_Holder_compatible_random}]
    We only need to show the inequalities. 
    Take an arbitrary $\alpha \in (\frac{1}{p},\upgamma)$. We apply the previous results as follows:
    \begin{multline}\label{eq:optimal_lift_mu_Holder_compatible_random_proof}
            \Vert \mu \Vert_{\upgamma\textrm{-}\mathrm{H\ddot{o}l}}^p  \overset{\text{Rem.\ref{rmk:Holder-FractionalSobolev_r_paper}}}{\geq} \tilde{c} \Vert\mu\Vert_{W^{\alpha,p}}^p \overset{\text{Thm.\ref{thm:optimal_lift_mu_Walphap_compatible_random}}}{=} \tilde{c} \mathbb{E} \left[  \int_{\Gamma_T} |\gamma|_{W^{\alpha,p}}^p \d \pi (\gamma) \right] \\ \overset{\text{Thm.\ref{thm:FractionalSobolev-Holder_r_paper}}}{\geq} \frac{\tilde{c}}{{\bar{c}}^p} \mathbb{E} \left[ \int_{\Gamma_T} | \gamma |^p_{\alpha - \frac{1}{p}\textrm{-}\mathrm{H\ddot{o}l}} \d \pi (\gamma) \right]\overset{\text{Thm.\ref{thm:lift_to_mu_Hol_random}}}{\geq}  \frac{\tilde{c}}{{\bar{c}}^p} \Vert \mu \Vert_{\alpha-\frac{1}{p}\textrm{-}\mathrm{H\ddot{o}l}}^p.
    \end{multline}
            where 
            $$
            \tilde{c} \coloneqq \frac{(\upgamma p - \alpha p)(\upgamma p - \alpha p+1)}{2 T^{(\upgamma p - \alpha p+1)}}, \quad \text{and} \quad {\bar{c}\,}^p \coloneqq 32 \frac{\alpha p +1}{\alpha p -1}. 
            $$
            Similarly, we have 
            \begin{align}
            \tilde{c} \, \mathbb{E} \left[  \int_{\Gamma_T} |\gamma|_{W^{\alpha,p}}^p \d \pi (\gamma) \right] \overset{\text{Thm.\ref{thm:FractionalSobolev-Holder_r_paper}}}{\geq} \frac{\tilde{c}}{{\bar{c}}^p} \frac{1}{T^{\alpha p -1}} \mathbb{E} \left[ \int_{\Gamma_T} | \gamma |^p_{\frac{1}{\alpha}\textrm{-}\mathrm{var}} \d \pi (\gamma) \right]\overset{\text{Thm.\ref{thm:lift_to_mu_var_random}}}{\geq}   \frac{\tilde{c}}{{\bar{c}}^p} \frac{1}{T^{\alpha p -1}} \Vert \mu \Vert_{p\textrm{-}\mathrm{var}}^p.
            \end{align}
    \end{proof}

    \subsection{Expectation of random lifts}
    
    As a corollary of the previous results, we note that taking expectation of the constructed random lift $\pi \in P_{\Omega}(C(I;\mathcal{X}))$ for $(\mu_t) \subset P_{p,\Omega}(\mathcal{X})$ immediately produces a (not necessarily optimal) lift $\mathbb{E} \pi \in P(C(I;\mathcal{X}))$ for the (not necessarily compatible) deterministic curve $(\mathbb{E} \mu_t) \subset P_p(\mathcal{X})$. 

    \begin{corollary}\label{crl:lift_Emu_Walphap}
        Under the assumptions of Theorem \ref{thm:optimal_lift_mu_Walphap_compatible_random}, we have $(\mathbb{E} \mu_t) \in W^{\alpha, p} (I;P_p (\mathcal{X}))$, and the expectation of $\pi$ obtained therein, $\mathbb{E} \pi \in P(C(I;\mathcal{X}))$, satisfies:
        \begin{enumerate}[label=(\roman*), font=\normalfont]
            \item $\mathbb{E}\pi$ is concentrated on $W^{\alpha,p}(I;\mathcal{X})$;
            \item $(e_t)_\# \mathbb{E} \pi =\mathbb{E} \mu_t $ for all $t\in I$;
            \item we have for all $s,t \in I$ that
            \begin{equation}
            W_p^p (\mathbb{E}\mu_s,\mathbb{E}\mu_t) \leq  \int_{\Gamma_T} d(\gamma_t,\gamma_s)^p \d \mathbb{E} \pi (\gamma) = \mathbb{W}_p^p (\mu_t,\mu_s),
            \end{equation}
             and, moreover,
             \begin{align}
            |\mathbb{E} \mu |_{W^{\alpha,p}}^p \leq  \int_{\Gamma_T} |\gamma|_{W^{\alpha,p}}^p \d \mathbb{E} \pi (\gamma) = \Vert \mu \Vert_{W^{\alpha,p}}^p .
            \end{align}
        \end{enumerate}
    \end{corollary}

    \begin{proof}
       We recall that $\mu_t \in P_{p,\Omega} (\mathcal{X})$ implies $\mathbb{E} \mu_t \in P_{p} (\mathcal{X})$. Also, by inequality \eqref{eq:WpEmuEnu_Wpmunu}, we have
        \begin{align}\label{eq:lift_Emu_Walphap_proof}
             |\mathbb{E} \mu|_{W^{\alpha,p}}^p
              \coloneqq \iint \frac{W_p^p (\mathbb{E} \mu_s,\mathbb{E} \mu_t)}{|t-s|^{1+\alpha p}} \d s \d t
              \leq \iint \frac{\mathbb{W}_p^p (\mu_s,\mu_t)}{|t-s|^{1+\alpha p}} \d s \d t \eqqcolon \Vert \mu \Vert_{W^{\alpha,p}}^p < + \infty.
        \end{align}
        The rest of the statements are essentially a consequence of the simple identity \eqref{eq:Exp_int_phi_mu}.
        \\
        ``(i)'' It follows from the finiteness of the integral \eqref{eq:optimality_pi_1_random} that 
        \begin{equation}
             \int_{\Gamma_T}|\gamma|_{W^{\alpha, p}}^p \d  \mathbb{E} \pi (\gamma)    < + \infty,
        \end{equation}
        which implies that $\mathbb{E} \pi$-a.e. $\gamma \in \Gamma_T$ has finite $W^{\alpha,p}$-norm. 
        \\
        ``(ii)'' By Theorem \ref{thm:optimal_lift_mu_Walphap_compatible_random} (ii), we know that $(e_t)_{\#} \pi^\omega = \mu_t^\omega$ a.s at each time. This simply implies that $(e_t)_{\#} \mathbb{E}\pi = \mathbb{E}\mu_t$ at each time. Take $\varphi \in C_b(\mathcal{X})$ and observe that
        \begin{align}
            \int_{\Gamma_T} \varphi (\gamma_t) \d \mathbb{E} \pi (\gamma) = \mathbb{E} \left[ \int_{\Gamma_T} \varphi (\gamma_t) \d \pi (\gamma) \right] = \mathbb{E} \left[ \int_{\mathcal{X}} \varphi (x) \d \mu_t (x) \right] =  \int_{\mathcal{X}} \varphi (x) \d \mathbb{E} \mu_t (x).
        \end{align}
        ``(iii)'' By the previous part, we have $(e_s,e_t)_{\#} \mathbb{E} \pi \in \mathrm{Cpl}(\mathbb{E} \mu_s, \mathbb{E} \mu_t)$, which, together with Theorem \ref{thm:optimal_lift_mu_Walphap_compatible_random} (iii), provide us with the first estimate. The second estimate finally follows from \eqref{eq:lift_Emu_Walphap_proof} and Theorem \ref{thm:optimal_lift_mu_Walphap_compatible_random} (iii):
        \begin{equation}
            |\mathbb{E} \mu|_{W^{\alpha,p}}^p \leq \Vert \mu \Vert_{W^{\alpha,p}}^p  = \mathbb{E} \left[\int_{\Gamma_T} |\gamma|^p_{W^{\alpha,p}} \d \pi (\gamma)\right] = \int_{\Gamma_T} |\gamma|^p_{W^{\alpha,p}} \d \mathbb{E} \pi (\gamma).
        \end{equation}
    \end{proof}

    \begin{corollary}\label{crl:lift_Emu_Holder}
        Under the assumptions of Corollary \ref{thm:optimal_lift_mu_Holder_compatible_random}, we have $(\mathbb{E} \mu_t) \in C^{\upgamma\textrm{-} \mathrm{H\ddot{o}l}} (I;P_p (\mathcal{X}))$, and the expectation of $\pi$ obtained therein, $\mathbb{E} \pi \in P(C(I;\mathcal{X}))$, satisfies:
        \begin{enumerate}[label=(\roman*), font=\normalfont]
            \item $\mathbb{E}\pi$ is concentrated on $W^{\alpha,p}(I;\mathcal{X})\subset C^{(\alpha-\frac{1}{p})\textrm{-} \mathrm{H\ddot{o}l}}(I;\mathcal{X})$ for any $\alpha \in (\frac{1}{p},\upgamma)$;
            \item $(e_t)_\# \mathbb{E} \pi =\mathbb{E} \mu_t $ for all $t\in I$;
            \item we have for all $s,t \in I$ that
            \begin{equation}
            W_p^p (\mathbb{E}\mu_s,\mathbb{E}\mu_t) \leq  \int_{\Gamma_T} d(\gamma_t,\gamma_s)^p \d \mathbb{E} \pi (\gamma) = \mathbb{W}_p^p (\mu_t,\mu_s),
            \end{equation}
             and, moreover, for any $\alpha \in (\frac{1}{p},\upgamma)$,
             \begin{align}
                   \Vert \mu \Vert_{\upgamma\textrm{-}\mathrm{H\ddot{o}l}}^p \geq c \,  \int_{\Gamma_T} | \gamma |^p_{\alpha - \frac{1}{p}\textrm{-}\mathrm{H\ddot{o}l}} \d \mathbb{E} \pi (\gamma)  \geq c \, \Vert \mu \Vert_{\alpha-\frac{1}{p}\textrm{-}\mathrm{H\ddot{o}l}}^p,
            \end{align}
            where $c = c(\upgamma,\alpha, p, T)$ is an explicit positive constant. 
        \end{enumerate}
    \end{corollary}

    \begin{proof}
        $(\mu_t) \subset P_{p,\Omega} (\mathcal{X})$ implies $(\mathbb{E} \mu_t) \subset P_{p} (\mathcal{X})$. By inequality \eqref{eq:WpEmuEnu_Wpmunu}, we deduce
        \begin{equation}
            | \mathbb{E} \mu |^p_{\upgamma\textrm{-}\mathrm{H\ddot{o}l}}
            \coloneqq \sup_{0 \leq s < t \leq T} \frac{W_p^p(\mathbb{E} \mu_s,\mathbb{E} \mu_t)}{|t-s|^\upgamma}
            \leq 
            \sup_{0 \leq s < t \leq T} \frac{\mathbb{W}_p^p(\mu_s, \mu_t)}{|t-s|^\upgamma} \eqqcolon \Vert \mu \Vert^p_{\upgamma\textrm{-}\mathrm{H\ddot{o}l}} < + \infty.
        \end{equation}
        The rest follows as in Corollary \ref{crl:lift_Emu_Walphap} and Corollary \ref{thm:optimal_lift_mu_Holder_compatible_random}. 
    \end{proof}

    \subsection{Dynamic formulation of Wasserstein distance between random measures via Besov energy}
    As the final result of this section, we provide a dynamic formulation of $\mathbb{W}_p$ in terms of the Besov energy.

    \begin{proposition}\label{crl:dynamic_formulation_EWp}
		Let $(\mathcal{X}, d)$ be a complete, separable, and locally compact length metric space, and $(\Omega, \mathcal{F}, \mathbb{P})$ be a probability space.  Let  $1<p<\infty$ and $ 0< \alpha <  1$. Then for every $\mu,\nu \in P_{p,\Omega}(\mathcal{X})$, we have 
		\begin{multline}
			\mathbb{W}^p_p(\mu,\nu) =  \big(1-2^{-(p-\alpha p)}\big)\min \bigg\{ \mathbb{E} \left[\int_{\Gamma} |\gamma|_{b^{\alpha,p}}^p \d \pi (\gamma)\right]  : \\
			 \pi \in P_{\Omega}(C([0,1];\mathcal{X})), \,\, (e_0)_\# \pi^\omega  = \mu^\omega, \, (e_1)_\# \pi^\omega  = \nu^\omega  \,\, \mathbb{P}\text{-a.s.}\bigg\}. 
		\end{multline}
		In addition, $\pi$ is a minimizer if and only if  $(e_0,e_1)_{\#} \pi \in \mathrm{OptCpl}_{\Omega}(\mu,\nu)$ and $\pi^\omega$ is concentrated on $ Geo([0,1];\mathcal{X})$ a.s.
	\end{proposition}
	
	\begin{proof}
		Let $\pi$ be an admissible random path measure in the minimization above with finite Besov energy in expectation. Using \cite[Section 3.5]{Abedi2025paths}, we have 
		\begin{equation}\label{eq:proof_dynamic_EWp_balphap}
			\mathbb{W}_p^p (\mu,\nu) \leq \mathbb{E} \left[ \int_{\Gamma} d(\gamma_0,\gamma_1)^p \d \pi (\gamma) \right] \leq  \big(1-2^{-(p-\alpha p)}\big) \mathbb{E} \left[ \int_{\Gamma} |\gamma|^p_{b^{\alpha, p}} \d \pi (\gamma)\right]. 
		\end{equation}
		Now we show by direct construction that a minimizer exists.
		Let $\Upsilon \in \mathrm{OptCpl}_{\Omega} (\mu,\nu)$. Note that $\mathcal{X}$ is a geodesic space by the Hopf–Rinow Theorem. By Remark \ref{rmk:measurable_geodesic_selection} on the existence of a Borel measurable selection of geodesics $\ell : \mathcal{X}\times\mathcal{X} \to Geo([0,1];\mathcal{X})$ under local compactness assumption on the state space, we set $\pi_*^\omega \coloneqq (\ell)_{\#} \Upsilon^\omega$ a.s., which defines a random path measure $\pi_* \in P_{\Omega}(C([0,1];\mathcal{X})) $.
		Note that, by its very construction, $(e_0,e_1)_{\#} \pi_* \in \mathrm{OptCpl}_{\Omega}(\mu,\nu)$, and $\pi_*^\omega$ is concentrated on $Geo([0,1];\mathcal{X})$ a.s. By \cite[Section 3.5]{Abedi2025paths}, we have 
		\begin{equation}
			\mathbb{W}_p^p(\mu,\nu)  = \mathbb{E} \left[\int_{Geo} d(\gamma_0,\gamma_1)^p \d \pi_* (\gamma) \right] = \big(1-2^{-(p-\alpha p)}\big) \mathbb{E} \left[ \int_{Geo} |\gamma|^p_{b^{\alpha, p}} \d \pi_* (\gamma)\right]. 
		\end{equation}
		Finally, note that an admissible random path measure $\pi_*$ is a minimizer if and only if the inequalities in  \eqref{eq:proof_dynamic_EWp_balphap} are equalities, which happens if and only if $(e_0,e_1)_{\#} \pi \in \mathrm{OptCpl}_{\Omega}(\mu,\nu)$ and $\pi^\omega$ is concentrated on $ Geo([0,1];\mathcal{X})$ a.s. by \cite[Section 3.5]{Abedi2025paths}. 
	\end{proof}

\section{Corollaries}\label{sec:corollaries_s}
Theorem \ref{thm:optimal_lift_mu_Walphap_compatible_random} (construction of a realizing random lift) can be immediately applied to the case of $\mathcal{X} = \mathbb{R}$, on which all probability measures with finite $p$-moments form a compatible family. 
As compatibility can easily fail in higher dimensions, here we give an alternative assumption in the case of $\mathcal{X} = \mathbb{R}^\mathrm{d}$. This assumption is obtained by replacing the Wasserstein metric in the regularity assumption with the $\nu$-based Wasserstein metric ${W}_{p,\nu}$ \eqref{eq:def_Wpnu}. Under this new assumption, the lift is constructed differently: rather than connecting measures at given time points with Wasserstein geodesics and transporting the particles optimally, we connect them with the $\nu$-based generalized geodesics. This is illustrated in Fig. \ref{fig:construction_path_measure}.
While this result may be interesting for applications, it is not from a theoretical perspective.
Firstly, because it is no longer intrinsic and depends on an external measure $\nu$.
Secondly, the construction leads to a trivial lift, as both the Kolmogorov extension theorem and the Kolmogorov–\v Centsov continuity theorem are applicable here.
In fact, Corollaries \ref{crl:lift_mu_Holder_Rd} and \ref{crl:lift_mu_Holder_R1} in the deterministic settings below can be immediately proved using \cite[Theorem 2.1.6]{StroockVaradhan2006}.  However, the stochastic versions of them do not seem immediate. Therefore, we first prove the stochastic results and then let the deterministic results follow afterward.
In this section, we let $\mathrm{d} \in \mathbb{N}$ and consider $\mathcal{X} = \mathbb{R}^\mathrm{d}$ equipped with the Euclidean distance.
All the proofs of the corollaries below are collected in \cref{subsec:proof_corollaries}.     
    
    \begin{figure}%[h]
    \centering
    \begin{tikzpicture}[scale=5, baseline]
    \draw [gray][-]
    (0,0.25)--
    (0.015380859375,0.3612632708068568)--
    (0.031005859375,0.3891164605040001)--
    (0.046630859375,0.3164720425429702)--
    (0.062255859375,0.3292680944756214)--
    (0.077880859375,0.35906552801519)--
    (0.093505859375,0.3297119891524707)--
    (0.109130859375,0.348720923202514)--
    (0.124755859375,0.312649164179907)--
    (0.140380859375,0.3137370229548458)--
    (0.156005859375,0.3298335290141453)--
    (0.171630859375,0.3139689973529306)--
    (0.187255859375,0.272450498411577)--
    (0.202880859375,0.225778142632884)--
    (0.218505859375,0.208878232423458)--
    (0.234130859375,0.231376029457259)--
    (0.249755859375,0.208989827924455)--
    (0.265380859375,0.243175347880716)--
    (0.281005859375,0.230905230707569)--
    (0.296630859375,0.243831667899003)--
    (0.312255859375,0.321969678542121)--
    (0.327880859375,0.330539575785065)--
    (0.343505859375,0.34170923398277)--
    (0.359130859375,0.394110789019707)--
    (0.374755859375,0.433145862774257)--
    (0.390380859375,0.397813061630579)--
    (0.406005859375,0.451465730268383)--
    (0.421630859375,0.582415030660174)--
    (0.437255859375,0.494237431018321)--
    (0.452880859375,0.463052943810623)--
    (0.468505859375,0.507413070506955)--
    (0.484130859375,0.48277670767766)--
    (0.499755859375,0.437049979688923)--
    (0.515380859375,0.440003412666466)--
    (0.531005859375,0.416266620789955)--
    (0.546630859375,0.348207616141129)--
    (0.562255859375,0.418227275322591)--
    (0.577880859375,0.445692142202932)--
    (0.593505859375,0.38482033325536)--
    (0.609130859375,0.441833126521721)--
    (0.624755859375,0.370287747091853)--
    (0.640380859375,0.475213938666167)--
    (0.656005859375,0.482944598610605)--
    (0.671630859375,0.501711686484074)--
    (0.687255859375,0.49756066772299)--
    (0.702880859375,0.454712042809014)--
    (0.718505859375,0.468874352925224)--
    (0.734130859375,0.519369346574088)--
    (0.749755859375,0.559367060958715)--
    (0.765380859375,0.539319336746755)--
    (0.781005859375,0.573589392153039)--
    (0.796630859375,0.587086858022874)--
    (0.812255859375,0.524838104215206)--
    (0.827880859375,0.519056954334893)--
    (0.843505859375,0.479212453969141)--
    (0.859130859375,0.450896985308748)--
    (0.874755859375,0.483767596963912)--
    (0.890380859375,0.535935970933879)--
    (0.906005859375,0.507416319512363)--
    (0.921630859375,0.522262735487998)--
    (0.937255859375,0.550390028285906)--
    (0.952880859375,0.649058024988311)--
    (0.968505859375,0.618966730942909)--
    (0.984130859375,0.591265045544213)--
    (0.999755859375,0.552537486937108);
    % -------------------------------
    % 5 points
    \filldraw [black] (0,0.25) circle (0.16pt);
    \filldraw [black] (0.249755859375,0.208989827924455) circle (0.16pt);
    \filldraw [black] (0.499755859375,0.437049979688923) circle (0.16pt);
    \filldraw [black] (0.749755859375,0.559367060958715) circle (0.16pt);
    \filldraw [black] (0.999755859375,0.552537486937108) circle (0.16pt);
    % -------------------------------
    % draw lines between points
    \draw [-]
    (0,0.25) --
    (0.249755859375,0.208989827924455) --
    (0.499755859375,0.437049979688923) --
    (0.749755859375,0.559367060958715) --
    (0.999755859375,0.552537486937108);
    % -------------------------------
    % label the measures
    \draw (0,0.25)
    node[anchor=east] {\scriptsize $\mu_{t_0}$};
    \draw (0.249755859375,0.208989827924455 +0.015)
    node[anchor=south] {\scriptsize $\mu_{t_1}$};
    \draw (0.499755859375 - 0.028 ,0.437049979688923 +0.015)
    node[anchor=south west] {\scriptsize $\mu_{t_2}$};
    \draw (0.749755859375,0.559367060958715 )
    node[anchor=south] {\scriptsize $\mu_{t_3}$};
    \draw (0.999755859375,0.552537486937108)
    node[anchor=west] {\scriptsize $\mu_{t_4}$};
    \end{tikzpicture}
    % -------------------------------
    % -------------------------------
    %\hspace{1.1cm}
    % -------------------------------
    % -------------------------------
    \begin{tikzpicture}[scale=5, baseline]
    \draw [gray][-]
    (0,0.25)--
    (0.015380859375,0.3612632708068568)--
    (0.031005859375,0.3891164605040001)--
    (0.046630859375,0.3164720425429702)--
    (0.062255859375,0.3292680944756214)--
    (0.077880859375,0.35906552801519)--
    (0.093505859375,0.3297119891524707)--
    (0.109130859375,0.348720923202514)--
    (0.124755859375,0.312649164179907)--
    (0.140380859375,0.3137370229548458)--
    (0.156005859375,0.3298335290141453)--
    (0.171630859375,0.3139689973529306)--
    (0.187255859375,0.272450498411577)--
    (0.202880859375,0.225778142632884)--
    (0.218505859375,0.208878232423458)--
    (0.234130859375,0.231376029457259)--
    (0.249755859375,0.208989827924455)--
    (0.265380859375,0.243175347880716)--
    (0.281005859375,0.230905230707569)--
    (0.296630859375,0.243831667899003)--
    (0.312255859375,0.321969678542121)--
    (0.327880859375,0.330539575785065)--
    (0.343505859375,0.34170923398277)--
    (0.359130859375,0.394110789019707)--
    (0.374755859375,0.433145862774257)--
    (0.390380859375,0.397813061630579)--
    (0.406005859375,0.451465730268383)--
    (0.421630859375,0.582415030660174)--
    (0.437255859375,0.494237431018321)--
    (0.452880859375,0.463052943810623)--
    (0.468505859375,0.507413070506955)--
    (0.484130859375,0.48277670767766)--
    (0.499755859375,0.437049979688923)--
    (0.515380859375,0.440003412666466)--
    (0.531005859375,0.416266620789955)--
    (0.546630859375,0.348207616141129)--
    (0.562255859375,0.418227275322591)--
    (0.577880859375,0.445692142202932)--
    (0.593505859375,0.38482033325536)--
    (0.609130859375,0.441833126521721)--
    (0.624755859375,0.370287747091853)--
    (0.640380859375,0.475213938666167)--
    (0.656005859375,0.482944598610605)--
    (0.671630859375,0.501711686484074)--
    (0.687255859375,0.49756066772299)--
    (0.702880859375,0.454712042809014)--
    (0.718505859375,0.468874352925224)--
    (0.734130859375,0.519369346574088)--
    (0.749755859375,0.559367060958715)--
    (0.765380859375,0.539319336746755)--
    (0.781005859375,0.573589392153039)--
    (0.796630859375,0.587086858022874)--
    (0.812255859375,0.524838104215206)--
    (0.827880859375,0.519056954334893)--
    (0.843505859375,0.479212453969141)--
    (0.859130859375,0.450896985308748)--
    (0.874755859375,0.483767596963912)--
    (0.890380859375,0.535935970933879)--
    (0.906005859375,0.507416319512363)--
    (0.921630859375,0.522262735487998)--
    (0.937255859375,0.550390028285906)--
    (0.952880859375,0.649058024988311)--
    (0.968505859375,0.618966730942909)--
    (0.984130859375,0.591265045544213)--
    (0.999755859375,0.552537486937108);
    % -------------------------------
    % 5 points
    \filldraw [black] (0,0.25) circle (0.16pt);
    \filldraw [black] (0.249755859375,0.208989827924455) circle (0.16pt);
    \filldraw [black] (0.499755859375,0.437049979688923) circle (0.16pt);
    \filldraw [black] (0.749755859375,0.559367060958715) circle (0.16pt);
    \filldraw [black] (0.999755859375,0.552537486937108) circle (0.16pt);
    % -------------------------------
    % draw lines between points
    \draw [dashed]
    (0,0.25) --
    (0.249755859375,0.208989827924455) --
    (0.499755859375,0.437049979688923) --
    (0.749755859375,0.559367060958715) --
    (0.999755859375,0.552537486937108);
    % -------------------------------
    % label the measures
    \draw (0,0.25)
    node[anchor=east] {\scriptsize $\mu_{t_0}$};
    \draw (0.249755859375,0.208989827924455 +0.015)
    node[anchor=south] {\scriptsize $\mu_{t_1}$};
    \draw (0.499755859375 - 0.028 ,0.437049979688923 +0.015)
    node[anchor=south west] {\scriptsize $\mu_{t_2}$};
    \draw (0.749755859375,0.559367060958715 )
    node[anchor=south] {\scriptsize $\mu_{t_3}$};
    \draw (0.999755859375,0.552537486937108)
    node[anchor=west] {\scriptsize $\mu_{t_4}$};
    % -------------------------------
    % fixed measure \nu
    \draw (0.35,-0.04) node[anchor=north west] {\scriptsize $\nu \in P_p^{ac}(\mathcal{X})$};
    % -------------------------------
    % lines between \nu and \mu_t_i
    \draw [black][-] (0.4,-0.05) -- (0,0.25);
    \draw [black][-] (0.4,-0.05) -- (0.249755859375,0.208989827924455);
    \draw [black][-] (0.4,-0.05) -- (0.499755859375,0.437049979688923);
    \draw [black][-] (0.4,-0.05) -- (0.749755859375,0.559367060958715);
    \draw [black][-] (0.4,-0.05) -- (0.999755859375,0.552537486937108);
    % fixed measure \nu
    \draw [fill=white] (0.4,-0.05) circle (0.16pt);
    \end{tikzpicture}
    \captionsetup{font=footnotesize}
    \caption{Two ways of constructing $\pi_n$. Solid lines represent optimal couplings while dashed lines are not necessarily optimal. Measure are connected by Wasserstein geodesics \textbf{(left)} and by $\nu$-based generalized geodesics \textbf{(right)}.}  
    \label{fig:construction_path_measure}
    \end{figure}

    \begin{table}[h]%[tbhp]
    \begin{center}
    \label{table:corollaries}
    \captionsetup{font=footnotesize}
    \caption{Arrangement of corollaries in Euclidean setting.}
    {\footnotesize
	\begin{tabular}{c|ll|ll}
		\hline
		&
        \multicolumn{2}{c|}{$\mathbb{R}^{\mathrm{d}}, \, \mathrm{d} \in \mathbb{N}$}
		&
        \multicolumn{2}{c}{$\mathbb{R}^{1}$ }
        \\
		\hline
		deterministic &
        Corollary \ref{crl:lift_mu_Walphap_Rd}. $W^{\alpha,p}$-regularity
        &
        w.r.t. $W_{p,\nu}$
        &
        Corollary \ref{crl:lift_mu_Walphap_R1}.
        $W^{\alpha,p}$-regularity
        &
        w.r.t. $W_{p}$
        \\
        setting
        &
        Corollary \ref{crl:lift_mu_Holder_Rd}.
        $C^{\upgamma\textrm{-}\mathrm{H\ddot{o}l}}$-regularity
        &
        w.r.t. $W_{p,\nu}$
        &
        Corollary \ref{crl:lift_mu_Holder_R1}. $C^{\upgamma\textrm{-}\mathrm{H\ddot{o}l}}$-regularity
        &
        w.r.t. $W_{p}$
        \\
        \hline
        stochastic
        &
        Corollary \ref{crl:lift_mu_Walphap_Rd_random}.
        $W^{\alpha,p}$-regularity
        &
        w.r.t. $\mathbb{W}_{p,\nu}$
        &
        Corollary \ref{crl:lift_mu_Walphap_R1_random}.
        $W^{\alpha,p}$-regularity
        &
        w.r.t. $\mathbb{W}_{p}$
        \\
        setting
        &
        Corollary \ref{crl:lift_mu_Holder_Rd_random}.
        $C^{\upgamma\textrm{-}\mathrm{H\ddot{o}l}}$-regularity
        &
        w.r.t. $\mathbb{W}_{p,\nu}$
        &
        Corollary \ref{crl:lift_mu_Holder_R1_random}.
        $C^{\upgamma\textrm{-}\mathrm{H\ddot{o}l}}$-regularity
        &
        w.r.t. $\mathbb{W}_{p}$
        \\
		\hline
	\end{tabular}
    }
    \end{center}
\end{table}

    \subsection{Corollaries in \texorpdfstring{$\mathbb{R}^\mathrm{d}$}{Rd}: deterministic setting}

    \begin{corollary}\label{crl:lift_mu_Walphap_Rd}
        Let $\mathcal{X} = \mathbb{R}^{\mathrm{d}}$  and $I \coloneqq [0,T] \subset \mathbb{R}$. 
        Let $(\mu_t) \in W^{\alpha,p} (I; (P_p(\mathcal{X}),W_{p,\nu}) )$ with $1<p<\infty$ and $\frac{1}{p}<\alpha < 1$ for some measure $\nu \in P_p^{ac} (\mathcal{X})$. Denote by $T_t$ the unique optimal transport map from $\nu$ to $\mu_t$.       
        Then there exists a unique probability measure $\pi \in P(C(I;\mathcal{X}))$ such that
        \begin{equation}\label{eq:lift_mu_Walphap_Rd_pi}
            \pi = ((T_t)_{t \in I} )_\# \nu \quad\, \text{on} \quad\,  \big(\mathcal{X}^I, \mathcal{B}(\mathcal{X})^I\big).
        \end{equation}
        In particular, $\pi$ satisfies
        \begin{enumerate}[label=(\roman*), font=\normalfont]
            \item $\pi$ is concentrated on $W^{\alpha,p}(I;\mathcal{X})$; 
            \item $(e_t)_\#\pi=\mu_t$ for all $t\in I$;
            \item we have for all $s,t \in I$ that
            \begin{equation}
                W^p_p (\mu_s,\mu_t) \leq   \int_{\Gamma_T} |\gamma_t - \gamma_s|^p \d \pi(\gamma) = W_{p,\nu}^p (\mu_s,\mu_t);
            \end{equation}
           and in particular\footnote{Here, $| \mu |_{W^{\alpha,p},\nu}$ denotes the $W^{\alpha,p}$-regularity of the curve $(\mu_t)$ with respect to $W_{p,\nu}$.},
           \begin{align}\label{eq:lift_mu_Walphap_Rd_energy}
                |\mu|_{W^{\alpha,p}}^p \leq  \int_{\Gamma_T} |\gamma|_{W^{\alpha,p}}^p \d \pi (\gamma) = |\mu|_{W^{\alpha,p},\nu}^p.
            \end{align}
        \end{enumerate}
    \end{corollary}
    \noindent
    \textbf{Remark.}
    The equality \eqref{eq:lift_mu_Walphap_Rd_pi} means
    \begin{equation}\label{eq:lift_mu_Walphap_Rd_pi_equivalent}
          (e_{t_1}, \cdots, e_{t_j})_\#\pi = (T_{t_1}, \cdots, T_{t_j} )_\# \nu 
          \end{equation}
    for any finite sequence $t_1, \cdots, t_j \in I$.

    \begin{corollary}\label{crl:lift_mu_Holder_Rd}
        Let $\mathcal{X} = \mathbb{R}^{\mathrm{d}}$ and $I \coloneqq [0,T] \subset \mathbb{R}$.
        Let $(\mu_t) \in C^{\upgamma\textrm{-}\mathrm{H\ddot{o}l}} (I;(P_p(\mathcal{X}),W_{p,\nu}))$ with $1<p<\infty$ and $\frac{1}{p}<\upgamma \leq 1$ for some measure $\nu \in P_p^{ac} (\mathcal{X})$.
        Denote by $T_t$ the unique optimal transport map from $\nu$ to $\mu_t$.       
        Then there exists a unique probability measure $\pi \in P(C(I;\mathcal{X}))$ such that
        \begin{equation}\label{eq:lift_mu_Holder_Rd_pi}
            \pi = ((T_t)_{t \in I} )_\# \nu \quad\, \text{on} \quad\,  \big(\mathcal{X}^I, \mathcal{B}(\mathcal{X})^I\big).
        \end{equation}
        In particular, $\pi$ satisfies
        \begin{enumerate}[label=(\roman*), font=\normalfont]
            \item $\pi$ is concentrated on $W^{\alpha,p}(I;\mathcal{X}) \subset C^{(\alpha - \frac{1}{p})\textrm{-} \mathrm{H\ddot{o}l}}(I;\mathcal{X})$ for any $\alpha \in (\frac{1}{p},\upgamma)$;
            %, and in particular, 
            %\begin{enumerate}[label=(\roman*), font=\normalfont]
            %    \item[$\diamond$] $\pi$ is concentrated on $ C^{(\alpha - \frac{1}{p})\textrm{-} \mathrm{H\ddot{o}l}}(I;\mathcal{X})$;
            %    \item[$\diamond$] $\pi$ is concentrated on $ C^{\frac{1}{\alpha}\textrm{-} \mathrm{var}}(I;\mathcal{X})$;
            %\end{enumerate}
            \item $(e_t)_\#\pi=\mu_t$ for all $t\in I$;
            \item  we have for all $s,t \in I$ that
            \begin{equation}
                W^p_p (\mu_s,\mu_t) \leq  \int_{\Gamma_T} |\gamma_t - \gamma_s|^p \d \pi(\gamma) = W_{p,\nu}^p (\mu_s,\mu_t);
            \end{equation}
           and, for any $\alpha \in (\frac{1}{p},\upgamma)$,  we have \eqref{eq:lift_mu_Walphap_Rd_energy} and           \begin{align}\label{eq:optimal_lift_mu_Holder_Rd}
               | \mu |_{\upgamma\textrm{-}\mathrm{H\ddot{o}l},\nu}^p \geq c \int_{\Gamma_T} | \gamma |^p_{\alpha - \frac{1}{p}\textrm{-}\mathrm{H\ddot{o}l}} \d \pi (\gamma) \geq c | \mu |_{\alpha-\frac{1}{p}\textrm{-}\mathrm{H\ddot{o}l}}^p,
            \end{align}
            where $c = c(\upgamma,\alpha,p,T)$ is an explicit positive constant.
        \end{enumerate}
    \end{corollary}

\subsection{Corollaries in \texorpdfstring{$\mathbb{R}^\mathrm{d}$}{Rd}: stochastic setting}
    \begin{corollary}\label{crl:lift_mu_Walphap_Rd_random}
        Let $\mathcal{X} = \mathbb{R}^{\mathrm{d}}$ and $I \coloneqq [0,T] \subset \mathbb{R}$. 
        Let $(\mu_t)_{t \in I}$ be a probability measure-valued stochastic process defined on a probability space $(\Omega, \mathcal{F}, \mathbb{P})$ such that $(\mu_t) \in W^{\alpha,p} (I; (P_{p,\Omega}(\mathcal{X}),\mathbb{W}_{p,\nu}) )$ with $1<p<\infty$ and $\frac{1}{p}<\alpha < 1$ for some measure $\nu \in P_p^{ac} (\mathcal{X})$. Denote by $T^\omega_t$ the optimal transport map from $\nu$ to $\mu^\omega_t$, which exists and is unique a.s.     
        Then there exists a random probability measure $\pi \in P_{\Omega}(C(I;\mathcal{X}))$ such that
        \begin{equation}\label{eq:lift_mu_Walphap_Rd_pi_random}
            (e_{t_1}, \cdots, e_{t_j})_\#\pi^\omega = (T^\omega_{t_1}, \cdots, T^\omega_{t_j} )_\# \nu \quad \mathbb{P}\textrm{-a.s. }
        \end{equation}
         for any finite sequence $t_1, \cdots, t_j \in I$. In particular, $\pi$ satisfies
        \begin{enumerate}[label=(\roman*), font=\normalfont]
            \item[(i) ] $\pi^\omega$ is concentrated on $W^{\alpha,p}(I;\mathcal{X})$ and $t \mapsto {(e_t)}_{\#} \pi^\omega$ is in $ W^{\alpha, p} (I;P_p(\mathcal{X}))$ a.s.;
            \item[(ii) ] $(e_t)_\#\pi^\omega =\mu_t^{\omega} $ a.s. for any $t\in I$;
            \item[(iii) ] we have for all $s,t \in I$ that
            \begin{equation}
                \mathbb{W}^p_p (\mu_s,\mu_t) \leq   \mathbb{E} \left[\int_{\Gamma_T} |\gamma_t - \gamma_s|^p \d \pi(\gamma)\right]=  \mathbb{W}_{p,\nu}^p (\mu_s,\mu_t);
            \end{equation}
           and in particular\footnote{Here, $\Vert \mu \Vert_{W^{\alpha,p},\nu}$ denotes the $W^{\alpha,p}$-regularity of the curve $(\mu_t)$ with respect to $\mathbb{W}_{p,\nu} (\cdot,\cdot) \coloneqq (\mathbb{E} [W^p_{p,\nu} (\cdot,\cdot)])^{1/p} $.},
           \begin{align}\label{eq:lift_mu_Walphap_Rd_random_energy}
                \Vert\mu\Vert_{W^{\alpha,p}}^p \leq  \mathbb{E} \left[ \int_{\Gamma_T} |\gamma|_{W^{\alpha,p}}^p \d \pi (\gamma)\right] = \Vert\mu\Vert_{W^{\alpha,p},\nu}^p.
            \end{align}
        \end{enumerate}
    \end{corollary}

    \begin{corollary}\label{crl:lift_mu_Holder_Rd_random}
        Let $\mathcal{X} = \mathbb{R}^{\mathrm{d}}$ and $I \coloneqq [0,T] \subset \mathbb{R}$. 
        Let $(\mu_t)_{t \in I}$ be a probability measure-valued stochastic process defined on a probability space $(\Omega, \mathcal{F}, \mathbb{P})$ such that $(\mu_t) \in C^{\upgamma\textrm{-} \mathrm{H\ddot{o}l}}(I;(P_{p,\Omega}(\mathcal{X}),\mathbb{W}_{p,\nu}))$ with $1<p<\infty$ and $\frac{1}{p}<\upgamma \leq 1$ for some measure $\nu \in P_p^{ac} (\mathcal{X})$. Denote by $T^\omega_t$ the optimal transport map from $\nu$ to $\mu^\omega_t$, which exists and is unique a.s.     
        Then there exists a random probability measure $\pi \in P_{\Omega}(C(I;\mathcal{X}))$ such that
        \begin{equation}
            (e_{t_1}, \cdots, e_{t_j})_\#\pi^\omega = (T^\omega_{t_1}, \cdots, T^\omega_{t_j} )_\# \nu \quad \mathbb{P}\textrm{-a.s. }
        \end{equation}
         for any finite sequence $t_1, \cdots, t_j \in I$. In particular, $\pi$ satisfies
    \begin{enumerate}[label=(\roman*), font=\normalfont]
                \item $\pi^\omega$ is concentrated on $W^{\alpha,p}(I;\mathcal{X}) \subset C^{(\alpha - \frac{1}{p})\textrm{-} \mathrm{H\ddot{o}l}}(I;\mathcal{X})$ a.s. and
                \\
                $t\mapsto (e_t)_\#\pi^\omega$ is in $W^{\alpha,p} (I;P_p(\mathcal{X}))\subset C^{(\alpha-\frac{1}{p})\textrm{-} \mathrm{H\ddot{o}l}}$\, $(I;P_p(\mathcal{X}))$ a.s. for any $\alpha \in (\frac{1}{p},\upgamma)$;
                % \subset C^{(\alpha - \frac{1}{p})\textrm{-} \mathrm{H\ddot{o}l}}
                \item $(e_t)_\#\pi^\omega=\mu_t^\omega$ a.s. for all $t\in I$;
                % W^{\alpha,p}(I;\mathcal{X}) \subset 
                % W^{\alpha,p}(I;P_p(\mathcal{X})) \subset
                \item we have for all $s,t \in I$ that
                \begin{equation}
                    \mathbb{W}^p_p (\mu_s,\mu_t) \leq   \mathbb{E} \left[\int_{\Gamma_T} |\gamma_t - \gamma_s|^p \d \pi(\gamma)\right]=  \mathbb{W}_{p,\nu}^p (\mu_s,\mu_t);
                \end{equation} 
                and, for any $\alpha \in (\frac{1}{p},\upgamma)$,  we have \eqref{eq:lift_mu_Walphap_Rd_random_energy} and 
                \begin{align}\label{eq:lift_mu_Holder_Rd_random_energy}
                   \Vert \mu \Vert_{\upgamma\textrm{-}\mathrm{H\ddot{o}l},\nu}^p \geq c \, \mathbb{E} \left[ \int_{\Gamma_T} | \gamma |^p_{\alpha - \frac{1}{p}\textrm{-}\mathrm{H\ddot{o}l}} \d \pi (\gamma) \right] \geq c \, \Vert \mu \Vert_{\alpha-\frac{1}{p}\textrm{-}\mathrm{H\ddot{o}l}}^p,
                \end{align}
                where $c = c(\upgamma,\alpha,p,T)$ is an explicit positive constant.
            \end{enumerate}
    \end{corollary}

\subsection{Corollaries in \texorpdfstring{$\mathbb{R}$}{R}: deterministic setting}\label{subsec:corollaries_R_deterministic}
    
    \begin{corollary}\label{crl:lift_mu_Walphap_R1}
        Let $\mathcal{X} = \mathbb{R}$ and $I \coloneqq [0,T]\subset \mathbb{R}$. 
        Let $(\mu_t) \in W^{\alpha,p} (I;P_p(\mathcal{X}))$ with $1<p<\infty$ and $\tfrac1p < \alpha <1$. Denote by $F_t$ the CDF of $\mu_t$ and by $F_t^{-1}$ its generalized inverse.
        Then there exists a unique probability measure $\pi \in P(C(I;\mathcal{X}))$ such that
        \begin{equation}\label{eq:lift_mu_Walphap_R1_pi}
            \pi = ((F_{t}^{-1})_{t \in I} )_\# \mathrm{Leb}|_{[0,1]} \quad\, \text{on} \quad\,  \big(\mathcal{X}^I, \mathcal{B}(\mathcal{X})^I\big).
        \end{equation}
                In particular, $\pi$ satisfies
        \begin{enumerate}[label=(\roman*), font=\normalfont]
            \item $\pi$ is concentrated on $W^{\alpha,p}(I;\mathcal{X})$; 
            \item $(e_t)_\#\pi=\mu_t$ for all $t\in I$;
            \item $(e_s,e_t)_\# \pi \in \mathrm{OptCpl}(\mu_s, \mu_t)$ for all $t,s \in I$; and in particular,
            \begin{align}\label{eq:lift_mu_Walphap_R1_energy}
                |\mu|_{W^{\alpha,p}}^p = \int_{\Gamma_T} |\gamma|_{W^{\alpha,p}}^p \d \pi (\gamma).
            \end{align}
        \end{enumerate}
    \end{corollary}

    \begin{corollary}\label{crl:lift_mu_Holder_R1}
        Let $\mathcal{X} = \mathbb{R}$ and $I \coloneqq [0,T]\subset \mathbb{R}$. 
        Let $(\mu_t) \in C^{\upgamma\textrm{-}\mathrm{H\ddot{o}l}}(I;P_p(\mathcal{X}))$ with $1<p<\infty$ and $\frac{1}{p}<\upgamma \leq 1$. Denote by $F_t$ the CDF of $\mu_t$ and by $F_t^{-1}$ its generalized inverse.
        Then there exists a unique probability measure $\pi \in P(C(I;\mathcal{X}))$ such that
        \begin{equation}\label{eq:lift_mu_Holder_R1_pi}
            \pi = ((F_{t}^{-1})_{t \in I} )_\# \mathrm{Leb}|_{[0,1]} \quad\, \text{on} \quad\,  \big(\mathcal{X}^I, \mathcal{B}(\mathcal{X})^I\big).
        \end{equation}
                In particular, $\pi$ satisfies
        \begin{enumerate}[label=(\roman*), font=\normalfont]
            \item $\pi$ is concentrated on $W^{\alpha,p}(I;\mathcal{X})\subset C^{(\alpha - \frac{1}{p})\textrm{-} \mathrm{H\ddot{o}l}}(I;\mathcal{X})$ for any $\alpha \in (\frac{1}{p},\upgamma)$; 
            \item $(e_t)_\#\pi=\mu_t$ for all $t\in I$;
            \item $(e_s,e_t)_\# \pi \in \mathrm{OptCpl}(\mu_s, \mu_t)$ for all $t,s \in I$; and, for any $\alpha \in (\frac{1}{p},\upgamma)$,  we have \eqref{eq:lift_mu_Walphap_R1_energy} and
            \begin{align}\label{eq:lift_mu_Holder_R1}
               | \mu |_{\upgamma\textrm{-}\mathrm{H\ddot{o}l}}^p \geq c \int_{\Gamma_T} | \gamma |^p_{\alpha - \frac{1}{p}\textrm{-}\mathrm{H\ddot{o}l}} \d \pi (\gamma) \geq c | \mu |_{\alpha-\frac{1}{p}\textrm{-}\mathrm{H\ddot{o}l}}^p,
            \end{align}
            where $c = c(\upgamma,\alpha,p,T)$ is an explicit positive constant.
        \end{enumerate}
    \end{corollary}

\subsection{Corollaries in \texorpdfstring{$\mathbb{R}$}{R}: stochastic setting}\label{subsec:corollaries_R_stochastic}

    \begin{corollary}\label{crl:lift_mu_Walphap_R1_random}
        Let $\mathcal{X} = \mathbb{R}$ and $I \coloneqq [0,T]\subset \mathbb{R}$. 
        Let $(\mu_t)_{t \in I}$ be a probability measure-valued stochastic process defined on a probability space $(\Omega, \mathcal{F}, \mathbb{P})$ such that $(\mu_t) \in W^{\alpha,p} (I; P_{p,\Omega}(\mathcal{X}))$ for some $1<p<\infty$ and $\frac{1}{p}< \alpha <  1$.  Denote by $F_t$ the CDF of $\mu_t$ and by $F_t^{-1}$ its generalized inverse.
        Then there exists a random probability measure $\pi \in P_{\Omega}(C(I;\mathcal{X}))$ such that
        \begin{equation}
            (e_{t_1}, \cdots, e_{t_j})_\#\pi^\omega = ((F_{t_1}^\omega)^{-1}, \cdots,(F_{t_j}^\omega)^{-1} )_\# \mathrm{Leb}|_{[0,1]} \quad \mathbb{P}\textrm{-a.s. }
        \end{equation}
         for any finite sequence $t_1, \cdots, t_j \in I$. In particular, $\pi$ satisfies
        \begin{enumerate}[label=(\roman*), font=\normalfont]
            \item[(i) ] $\pi^\omega$ is concentrated on $W^{\alpha,p}(I;\mathcal{X})$ and $t \mapsto {(e_t)}_{\#} \pi^\omega$ is in $ W^{\alpha, p} (I;P_p(\mathcal{X}))$ a.s.;
            \item[(ii) ] $(e_t)_\#\pi^\omega =\mu_t^{\omega} $ a.s. for any $t\in I$;
            \item[(iii) ] $(e_s,e_t)_\# \pi^\omega \in \mathrm{OptCpl}(\mu_s^{\omega}, \mu_t^{\omega})$ a.s. for all $s,t \in I$; and in particular
            \begin{equation}\label{eq:lift_mu_Walphap_R1_random_energy}
                \Vert\mu\Vert_{W^{\alpha,p}}^p = \mathbb{E} \left[\int_{\Gamma_T} |\gamma|_{W^{\alpha,p}}^p \d \pi (\gamma) \right].
            \end{equation}
        \end{enumerate}
    \end{corollary}
    
\begin{corollary}\label{crl:lift_mu_Holder_R1_random}
        Let $\mathcal{X} = \mathbb{R}$ and $I \coloneqq [0,T]\subset \mathbb{R}$.
        Let $(\mu_t)_{t \in I}$ be a probability measure-valued stochastic process defined on a probability space $(\Omega, \mathcal{F}, \mathbb{P})$ such that $(\mu_t) \in C^{\upgamma\textrm{-} \mathrm{H\ddot{o}l}}(I;P_{p,\Omega}(\mathcal{X}))$ for some $1<p<\infty$ and $\frac{1}{p}<\upgamma \leq 1$.
        Denote by $F_t$ the CDF of $\mu_t$ and by $F_t^{-1}$ its generalized inverse.
        Then there exists a random probability measure $\pi \in P_{\Omega}(C(I;\mathcal{X}))$ such that
        \begin{equation}
            (e_{t_1}, \cdots, e_{t_j})_\#\pi^\omega = ((F_{t_1}^\omega)^{-1}, \cdots, (F_{t_j}^\omega)^{-1} )_\# \mathrm{Leb}|_{[0,1]} \quad \mathbb{P}\textrm{-a.s. }
        \end{equation}
         for any finite sequence $t_1, \cdots, t_j \in I$. In particular, $\pi$ satisfies
    \begin{enumerate}[label=(\roman*), font=\normalfont]
                \item $\pi^\omega$ is concentrated on $W^{\alpha,p}(I;\mathcal{X}) \subset  C^{(\alpha - \frac{1}{p})\textrm{-} \mathrm{H\ddot{o}l}}(I;\mathcal{X}) $ a.s. and
                \\
                $t\mapsto (e_t)_\#\pi^\omega$ is in $W^{\alpha,p} (I;P_p(\mathcal{X}))\subset C^{(\alpha-\frac{1}{p})\textrm{-} \mathrm{H\ddot{o}l}}$ \, $(I;P_p(\mathcal{X}))$ a.s. for any $\alpha \in (\frac{1}{p},\upgamma)$;
                \item $(e_t)_\#\pi^\omega=\mu_t^\omega$ a.s. for all $t\in I$;
                \item $(e_s,e_t)_\# \pi^\omega \in \mathrm{OptCpl}(\mu_s^\omega, \mu_t^\omega)$ a.s. for all $s,t \in I$; and for any $\alpha \in (\frac{1}{p},\upgamma)$, we have  \eqref{eq:lift_mu_Walphap_R1_random_energy} and 
                \begin{align}
                   \Vert \mu \Vert_{\upgamma\textrm{-}\mathrm{H\ddot{o}l}}^p \geq c \, \mathbb{E} \left[ \int_{\Gamma_T} | \gamma |^p_{\alpha - \frac{1}{p}\textrm{-}\mathrm{H\ddot{o}l}} \d \pi (\gamma) \right] \geq c \, \Vert \mu \Vert_{\alpha-\frac{1}{p}\textrm{-}\mathrm{H\ddot{o}l}}^p,
                \end{align}
                where $c = c(\upgamma,\alpha,p,T)$ is an explicit positive constant.
            \end{enumerate}
\end{corollary}

\subsection{Proof of corollaries}\label{subsec:proof_corollaries}
\begin{proof}[Proof of Corollary \ref{crl:lift_mu_Walphap_Rd_random}]
    First of all, we note that if we can find a random measure $\pi \in P_{\Omega}(C(I;\mathcal{X}))$ satisfying the finite-dimensional time-marginal property \eqref{eq:lift_mu_Walphap_Rd_pi_random}, the rest of the statements follow with simple computation and with the help of Theorem \ref{thm:lift_to_mu_Walphap_stoch}.
    To obtain such a measure, we first apply Theorem \ref{thm:optimal_lift_mu_Walphap_compatible_random} with the choice of the probability space $(\Omega \times \mathcal{Y},\mathcal{F} \otimes \mathcal{B}(\mathcal{Y}),\mathbb{P} \otimes \nu)$, where $\mathcal{Y} \coloneqq \mathcal{X}$, and then show that this measure indeed satisfies the finite-dimensional time-marginal property.
    For each $t\in I$, let us define the following random measure, now depending on two parameters $(\omega,y)$,
        \begin{equation}\label{eq:proof_clr_def_mutomegay}
        \mu_t^{\omega,y} \coloneqq \delta_{T^\omega_t(y)} \quad  \textrm{for } \mathbb{P} \otimes \nu \textrm{-a.e. } (\omega, y) \in \Omega \times \mathcal{Y}.
        \end{equation}
        Taking the average of this measure over the 2nd component (interpreted as in \eqref{eq:exp_mu_def}) recovers the random measure that is given to us here:
        \begin{equation}\label{eq:proof_clr_mutomegay}
        \mu_t^\omega = \int_{\mathcal{Y}} \mu_t^{\omega,y} \d \nu (y) \quad \textrm{for } \mathbb{P} \textrm{-a.e. } \omega \in \Omega.
        \end{equation}
        By what we have defined, we also have:
        \begin{equation}\label{eq:proof_clr_Wppomegay}
            \mathbb{W}^p_{p,\nu} (\mu_s,\mu_t) = \int_{\Omega \times \mathcal{Y}} W_p^p (\mu^{\omega,y}_s,\mu^{\omega,y}_t) \, \d \mathbb{P} \otimes \nu (\omega,y).
        \end{equation}
        Therefore, we can apply Theorem \ref{thm:optimal_lift_mu_Walphap_compatible_random} for the measure-valued process defined in \eqref{eq:proof_clr_def_mutomegay}, and obtain a random probability measure here denoted by $\eta \in P_{\Omega \times \mathcal{Y}} (C(I;\mathcal{X}))$ depending on two randomness parameters. 
        In particular, we obtain that  $(e_t)_{\#}\eta^{\omega,y} = \mu_t^{\omega,y}$ holds $\mathbb{P}\otimes\nu$-a.s.
        From the random measure $\eta$, we define another random measure $\pi \in P_{\Omega} (C(I;\mathcal{X}))$ by taking the average over the 2nd component:
        \begin{equation}\label{eq:proof_clr_etaomegay}
            \pi^\omega \coloneqq \int_{\mathcal{Y}} \eta^{\omega,y} \d \nu (y) \quad \textrm{for } \mathbb{P} \textrm{-a.e. } \omega \in \Omega.
        \end{equation}
        We claim that this is the desired measure. Let us first show that $\pi$ is a random lift of $(\mu_t)$. Take an arbitrary $\phi \in \mathrm{Lip}_b (\mathcal{X})$ and $F \in  \mathcal{F}$. We have 
        \begin{align}
            \int_{\Omega} \int_{\mathcal{X}} \mathds{1}_{F} (\omega) \phi(x) \d ({e_t}_{\#}\pi^\omega) (x) \d \mathbb{P} (\omega)
            &  = \int_{\Omega} \int_{\mathcal{Y}} \int_{\Gamma_T} \mathds{1}_{F} (\omega)  \phi (\gamma_t) \d \eta^{\omega,y} (\gamma) \d \nu (y) \d \mathbb{P} (\omega) \\
            & = \int_{\Omega} \int_{\mathcal{Y}} \int_{\mathcal{X}} \mathds{1}_{F} (\omega) \phi (x) \d \mu_t^{\omega,y} (x) \d \nu (y) \d \mathbb{P} (\omega) \\
            & = \int_{\Omega}  \int_{\mathcal{X}} \mathds{1}_{F} (\omega) \phi (x) \d \mu_t^{\omega} (x)  \d \mathbb{P} (\omega).
        \end{align}
        Therefore, $(e_t)_{\#}\pi^{\omega} = \mu_t^{\omega}$ $\mathbb{P}$-a.s., which already proves \eqref{eq:lift_mu_Walphap_Rd_pi_random} for $j=1$. 
        Let us also observe that  by \eqref{eq:optimality_pi_1_random},
            \begin{align} \Vert\mu\Vert_{W^{\alpha,p},\nu}^p =  \int_{\Omega \times \mathcal{Y}} \int_{ \Gamma_T} |\gamma|_{W^{\alpha,p}}^p \d \eta^{\omega,y} (\gamma) \d \mathbb{P}\otimes \nu (\omega,y)= \int_{\Omega} \int_{\Gamma_T} |\gamma|_{W^{\alpha,p}}^p \d \pi^{\omega} (\gamma) \d \mathbb{P}(\omega).
            \end{align}
        It thus remains to show \eqref{eq:lift_mu_Walphap_Rd_pi_random} for $j>1$. The idea is to demonstrate the same chain of equalities as in \eqref{eq:proof_random_1dim_time_marginal} but for higher-dimensional marginals. As for the test functions, it suffices to consider only the product of Lipschitz functions. We split the computations into four steps.
        \\
        \textbf{Setup.} Let $F \in  \mathcal{F}$, the test functions $\phi_1, \phi_2, \cdots \in \mathrm{Lip}_b(\mathcal{X})$, and $t_1, t_2, \cdots \in I$ be arbitrary. As in the  proof of Theorem \ref{thm:optimal_lift_mu_Walphap_compatible_random}, we denote by $\{\pi_{n_k}\}_{k \in \mathbb{N}}$ the subsequence converging narrowly to $\pi$. For any $t \in I$ and $k \in \mathbb{N}$, we set:
        $$
        t^k \coloneqq \frac{[2^{n_k} t ]}{2^{n_k}}. 
        $$ 
        \\
        \textbf{Claim 1.}  We claim
        \begin{equation}
        \lim_{k \to \infty}  \int_{\Omega} \int_{\Gamma_T} \left| \phi_1(\gamma_{t_1}) \cdots \phi_j(\gamma_{t_j}) - \phi_1(\gamma_{t_1^k}) \cdots \phi_j(\gamma_{t_j^k})  \right| \d \pi_{n_k}^{\omega} \d \mathbb{P} = 0, \qquad \forall j \in \mathbb{N}.
        \end{equation}
        We prove by induction:
        \\
        ``Base case $j=1$.'' By Lipschitz continuity of the function and Jensen's inequality, we obtain
        \begin{equation}\label{eq:proof_finite_dim_marginal_claim1_base}
        \int_{\Omega} \int_{\Gamma_T} \left| \phi_1(\gamma_{t_1}) - \phi_1(\gamma_{t_1^k})  \right| \d \pi_{n_k}^{\omega} \d \mathbb{P} \leq \text{Lip}(\phi_1) \left(\int_{\Omega} \int_{\Gamma_T}   |\gamma_{t_1}- \gamma_{t_1^k}|^p \d \pi_{n_k}^\omega  \d \mathbb{P} \right)^{1/p}.
        \end{equation}
        The right-hand side converges to zero by the same line of reasoning as in \eqref{eq:convergence_1_dim_marginal_random_partII}.
        \\
        ``Induction step $j \to j+1$.'' By adding and subtracting $\phi_1(\gamma_{t_1^k}) \cdots \phi_j(\gamma_{t_j^k})\phi_{j+1}(\gamma_{t_{j+1}})$, we get
        \begin{align}
            \int_{\Omega} \int_{\Gamma_T}  & \left| \phi_1(\gamma_{t_1}) \cdots \phi_j(\gamma_{t_j}) \phi_{j+1}(\gamma_{t_{j+1}}) - \phi_1(\gamma_{t_1^k}) \cdots \phi_j(\gamma_{t_j^k}) \phi_{j+1}(\gamma_{t_{j+1}^k})   \right| \d \pi_{n_k}^{\omega} \d \mathbb{P}  \\
            & \quad \leq \int_{\Omega} \int_{\Gamma_T} \left| \left[ \phi_1(\gamma_{t_1}) \cdots \phi_j(\gamma_{t_j}) - \phi_1(\gamma_{t_1^k}) \cdots \phi_j(\gamma_{t_j^k}) \right] \phi_{j+1}(\gamma_{t_{j+1}}) \right| \d \pi_{n_k}^{\omega} \d \mathbb{P} \\
            & \quad + \int_{\Omega} \int_{\Gamma_T} \left|  \phi_1(\gamma_{t_1^k}) \cdots \phi_j(\gamma_{t_j^k}) \left[ \phi_{j+1}(\gamma_{t_{j+1}}) - \phi_{j+1}(\gamma_{t_{j+1}^k}) \right] \right| \d \pi_{n_k}^{\omega} \d \mathbb{P} \\
            & \quad \leq \lVert \phi_{j+1} \rVert_{\infty} \int_{\Omega} \int_{\Gamma_T} \left|  \phi_1(\gamma_{t_1}) \cdots \phi_j(\gamma_{t_j}) - \phi_1(\gamma_{t_1^k}) \cdots \phi_j(\gamma_{t_j^k}) \right| \d \pi_{n_k}^{\omega} \d \mathbb{P} \\
            & \quad + \lVert \phi_{1} \rVert_{\infty} \cdots \lVert \phi_{j} \rVert_{\infty}   \int_{\Omega} \int_{\Gamma_T} \left|   \phi_{j+1}(\gamma_{t_{j+1}}) - \phi_{j+1}(\gamma_{t_{j+1}^k}) \right| \d \pi_{n_k}^{\omega} \d \mathbb{P}. \label{eq:proof_finite_dim_marginal_claim1_induction}
        \end{align}
        As $k\to \infty$, the first term vanishes by the induction hypothesis, and the second term as well by the base case.
        \\
        \textbf{Claim 2.} We claim
        \begin{align}
             \lim_{k \to \infty}  \int_{\Omega} \int_{\Gamma_T} \mathds{1}_{F} (\omega) \phi_1(\gamma_{t_1^k}) \cdots \phi_j(\gamma_{t_j^k})   \d \pi_{n_k}^{\omega} \d \mathbb{P} = \int_{\Omega} \int_{\Gamma_T} \mathds{1}_{F} (\omega) \phi_1(\gamma_{t_1}) \cdots \phi_j(\gamma_{t_j})  \d \pi^{\omega} \d \mathbb{P}.
        \end{align}
        Indeed,
        \begin{multline}
             \left| \int_{\Omega} \int_{\Gamma_T}  \mathds{1}_{F} (\omega) \phi_1(\gamma_{t_1}) \cdots \phi_j(\gamma_{t_j})  \d \pi^{\omega} \d \mathbb{P} - \int_{\Omega} \int_{\Gamma_T} \mathds{1}_{F} (\omega) \phi_1(\gamma_{t_1^k}) \cdots \phi_j(\gamma_{t_j^k})   \d \pi_{n_k}^{\omega} \d \mathbb{P} \right|   \\
             \leq \left| \int_{\Omega} \int_{\Gamma_T}  \mathds{1}_{F} (\omega) \phi_1(\gamma_{t_1}) \cdots \phi_j(\gamma_{t_j})  \d \pi^{\omega} \d \mathbb{P} - \int_{\Omega} \int_{\Gamma_T} \mathds{1}_{F} (\omega) \phi_1(\gamma_{t_1}) \cdots \phi_j(\gamma_{t_j})   \d \pi_{n_k}^{\omega} \d \mathbb{P} \right| \\ 
             + \left|
             \int_{\Omega} \int_{\Gamma_T} \mathds{1}_{F} (\omega) \left[ \phi_1(\gamma_{t_1}) \cdots \phi_j(\gamma_{t_j}) - \phi_1(\gamma_{t_1^k}) \cdots \phi_j(\gamma_{t_j^k}) \right]  \d \pi_{n_k}^{\omega} \d \mathbb{P}
             \right|.
        \end{multline}
        By passing to the limit, the first term on the right-hand side goes to zero  due to the narrow convergence of $\pi_{n_k}$ to $\pi$, and the second term also vanishes by Claim 1. 
        \\
        \textbf{Claim 3.}  We claim
        \begin{equation}
        \lim_{k \to \infty} \int_{\Omega} \int_{\mathcal{Y}} \left| \phi_1(T_{t_1}^\omega) \cdots \phi_j(T_{t_j}^\omega) - \phi_1(T_{t_1^k}^{\omega}) \cdots \phi_j(T_{t_j^k}^{\omega}) \right| \d \nu \d \mathbb{P} = 0, \qquad  \forall j \in \mathbb{N}. 
        \end{equation}
        We again prove by induction:
        \\
        ``Base case $j=1$.'' As in \eqref{eq:proof_finite_dim_marginal_claim1_base}, we estimate
    \begin{equation}
     \int_{\Omega} \int_{\mathcal{Y}} \left| \phi_1(T_{t_1}^\omega) - \phi_1(T_{t^k_1}^{\omega}) \right| \d \nu \d \mathbb{P} \leq \text{Lip}(\phi_1) \underbrace{ \left(  \int_{\Omega}\int_{\mathcal{Y}}   |T_{t_1}^\omega- T_{t^k_1}^{\omega}|^p \d \nu \d \mathbb{P} \right)^{1/p} }_{= \mathbb{W}_{p,\nu}(\mu_{t_1}, \mu_{t_1^k}) \to \, 0 \, \text{as} \, k \to \infty \, \text{by continuity.}}
    \end{equation}
    ``Induction step $j \to j+1$.'' Using the same approach as in \eqref{eq:proof_finite_dim_marginal_claim1_induction}, we obtain
    \begin{align}
    \int_{\Omega} \int_{\mathcal{Y}} & \left| \phi_1(T_{t_1}^\omega) \cdots \phi_{j}(T_{t_{j}}^\omega) \phi_{j+1}(T_{t_{j+1}}^\omega) - \phi_1(T_{t_1^k}^\omega) \cdots \phi_{j}(T_{t_{j}^k}^\omega) \phi_{j+1}(T_{t_{j+1}^k}^\omega) \right| \d \nu \d \mathbb{P} \\
    & \leq \lVert \phi_{j+1} \rVert_{\infty} \int_{\Omega} \int_{\mathcal{Y}} \left|\phi_1(T_{t_1}^\omega) \cdots \phi_j(T_{t_j}^\omega) - \phi_1(T_{t_1^k}^{\omega}) \cdots \phi_j(T_{t_j^k}^{\omega}) \right|  \d \nu \d \mathbb{P} \\
    & + \lVert \phi_{1} \rVert_{\infty} \cdots \lVert \phi_{j} \rVert_{\infty}   \int_{\Omega} \int_{\mathcal{Y}} \left|   \phi_{j+1}(T_{t_{j+1}}^\omega) - \phi_{j+1}(T_{t^k_{j+1}}^\omega) \right| \d \nu \d \mathbb{P},
    \end{align}
    which tends to zero as $k \to \infty$ by the induction hypothesis and the base case.
    \\
    \textbf{Final step.} We now put everything together to conclude similarly to  \eqref{eq:proof_random_1dim_time_marginal}:
    \begin{align}
        \int_{\Omega} \int_{\Gamma_T} \mathds{1}_{F} (\omega) \phi_1(\gamma_{t_1}) \cdots \phi_j(\gamma_{t_j}) \d \pi^\omega \d \mathbb{P}  &  \overset{\text{(by Claim 2)}}{=}  \lim_{k \to \infty}  \int_{\Omega} \int_{\Gamma_T} \mathds{1}_{F} (\omega) \phi_1(\gamma_{t_1^k}) \cdots \phi_j(\gamma_{t_j^k})   \d \pi_{n_k}^{\omega} \d \mathbb{P} \\
        & \overset{\text{(by Construction)}}{=}  \lim_{k \to \infty} \int_{\Omega} \int_{\mathcal{Y}} \mathds{1}_{F} (\omega) \phi_1(T_{t_1^k}^{\omega}) \cdots \phi_j(T_{t_j^k}^{\omega}) \d \nu \d \mathbb{P} \\
        & \overset{\text{(by Claim 3)}}{=}  \int_{\Omega} \int_{\mathcal{Y}} \mathds{1}_{F} (\omega) \phi_1(T_{t_1}^{\omega}) \cdots \phi_j(T_{t_j}^{\omega}) \d \nu \d \mathbb{P} \label{eq:proof_random_multidim_time_marginal},
    \end{align}
    which proves \eqref{eq:lift_mu_Walphap_Rd_pi_random}.
\end{proof}

\begin{proof}[Proof of the remaining corollaries] The other results follow as outlined below:
    \begin{itemize}
        \item Corollary \ref{crl:lift_mu_Holder_Rd_random} directly follows from  Corollary \ref{crl:lift_mu_Walphap_Rd_random} since $C^{\upgamma\textrm{-}\mathrm{H\ddot{o}l}} \subset W^{\alpha,p}$. The estimate \eqref{eq:lift_mu_Holder_Rd_random_energy} follows with the same line of reasoning as \eqref{eq:optimal_lift_mu_Holder_compatible_random_proof} with the help of Theorem \ref{thm:lift_to_mu_Hol_random}.
        \item Then we immediately obtain Corollaries \ref{crl:lift_mu_Walphap_Rd} and \ref{crl:lift_mu_Holder_Rd} as a particular case. Note that in the deterministic setting, the lift we obtain is unique since finite-dimensional marginal distributions determine the path measure on $(C(I;\mathcal{X}),\mathcal{C})$ uniquely. 
        \item Recalling \cref{subsec:optimal_transport_R1,subsec:optimal_transport_nu_based}, all corollaries in $\mathbb{R}$ 
        simply follow by replacing $\nu$ with $\mathrm{Leb}|_{[0,1]}$ and $T_t$ with the measurable map $F_t^{-1}$ in the corresponding results in $\mathbb{R}^{\mathrm{d}}$. 
    \end{itemize}
\end{proof}

\appendix \section{}\label{sec:appendix_s}
\begin{proof}[Proof of Corollary \ref{crl:minimizing_processes_sFPE_intro}]
        The proof sketch is illustrated in Fig. \ref{fig:proof_sketch}. 
        Assume first $p>1$ and let $(X_t)$ be the stochastic process constructed by \cite{LackerShkolnikovZhang2023} on a complete filtered probability space $(\tilde{\Omega},\tilde{\mathcal{F}},\tilde{\mathbb{F}}, \tilde{\mathbb{P}})$, extending $(\Omega,\mathcal{F},\mathbb{F},\mathbb{P})$, which solves the SDE \eqref{eq:SDE_associated_sFPE}. This process provides $(\mu_t)$ with a random lift in the sense of \eqref{eq:def:random_lift_intro}. Indeed, let $\tilde{\pi}$ be the random path measure  on $C([0,T];\mathbb{R}^\mathrm{d})$ given by
        $\tilde{\pi}  \coloneqq \mathrm{Law} (X | \mathcal{F}_T)
        $ a.s. and observe that for each $t \in [0,T]$ and $\varphi \in C_b (\mathbb{R}^\mathrm{d})$, we have 
        \begin{equation}\label{eq:conditional_random_lift}
            \int_{\Gamma_T} \varphi (\gamma_t) \d \tilde\pi (\gamma) = \tilde{\mathbb{E}} [ \varphi (X_t) |\mathcal{F}_T] = \tilde{\mathbb{E}} [ \varphi (X_t) |\mathcal{F}_t]  = \int_{\mathbb{R}^\mathrm{d}} \varphi(x) \d \mu_t (x) \quad \textrm{a.s.}
        \end{equation}
        where the measure zero set depends only on $t$. 
        Next, let us observe that 
        \begin{equation}
            \mathbb{E} \left[ \int_{\Gamma_T} |\gamma_0|^p \d \tilde\pi (\gamma)\right] =
            \tilde{\mathbb{E}} \left[ \tilde{\mathbb{E}} \big[ |X_0|^p \big| \mathcal{F}_T
            \big] \right] = \tilde{\mathbb{E}} \left[ |X_0|^p \right] = \mathbb{E} \left[ \int_{\mathbb{R}^\mathrm{d}} |x|^p \d \mu_{0} (x) \right] < + \infty \label{eq:March131712},
        \end{equation}
        which is finite by the assumption $\mu_0 \in P_{p,\Omega}(\mathbb{R}^\mathrm{d})$.
        Similarly, for any $0 \leq s < t \leq T$, we have 
        \begin{align}
        \mathbb{E} \bigg[ &  \int_{\Gamma_T} |\gamma_t - \gamma_s|^p \d \tilde\pi (\gamma)\bigg]    = 
         \tilde{\mathbb{E}} \bigg[ \tilde{\mathbb{E}} \big[|X_t - X_s|^p \big| \mathcal{F}_T \big] \bigg] \\ 
         & = \tilde{\mathbb{E}} \Big[ |X_t - X_s  |^p  \Big]
         \\
         &  \leq 2^{p-1} \tilde{\mathbb{E}} \left[ \left| \int_s^t b_r(X_r,\omega) \d r  \right|^p + \left| \int_s^t \alpha_r(X_r,\omega) \d B_r + \int_s^t \sigma_r(X_r,\omega) \d W_r \right|^p \right]\\
        & \leq  c_{p,\mathrm{d}} \, \tilde{\mathbb{E}} \left[ \left| \int_s^t |b_r(X_r,\omega)| \d r  \right|^p + \left| \int_s^t  |a_r(X_r,\omega)| \d r  \right|^{\frac{p}{2}}\right] \\
        & \leq c_{p,\mathrm{d}} \,  |t-s|^{p-1} \tilde{\mathbb{E}} \left[  \int_s^t \left| b_r(X_r,\omega) \right|^p \d r  \right]  +  c_{p,\mathrm{d}} \,  |t-s|^{\frac{p-1}{2}} \tilde{\mathbb{E}} \left[ \left| \int_s^t  |a_r(X_r,\omega)|^p \d r  \right|^{\frac12} \right]\\
        & \leq c_{p,\mathrm{d}} \, |t-s|^{\frac{p-1}{2}} T^{\frac{p-1}{2}} \mathbb{E} \left[  \int_s^t \int_{\mathbb{R}^\mathrm{d}} \left| b_r  \right|^p \d \mu_r \d r  \right] + c_{p,\mathrm{d}} \,  |t-s|^{\frac{p-1}{2}} \left| \mathbb{E} \left[  \int_s^t \int_{\mathbb{R}^\mathrm{d}} |a_r|^p \d \mu_r \d r   \right] \right|^{\frac12}
         \\
        & \leq c_{p,\mathrm{d}}  \, \big|t-s\big|^{\frac{p-1}{2}} \, \mathcal{E}_{p,T} (\mu,b,a) , \label{eq:March131713}
        \end{align}
        where we used standard estimates from H\"{o}lder's inequality and Burkholder--Davis--Gundy's inequality (e.g. as in \cite[Proposition 4.4]{LackerShkolnikovZhang2023}) and $2a = \sigma \sigma^\top +\alpha \alpha^\top  $. Here, $c$'s are different constants whose dependence on the parameters is specified.
        The estimate \eqref{eq:March131713} evaluated at $s=0$ together with \eqref{eq:March131712} guarantees that $\mu_t \in P_{p,\Omega}(\mathbb{R}^\mathrm{d})$ for all $t \in [0,T]$.
        Thus, it also implies:
        \begin{align}
            \mathbb{W}^p_p (\mu_s,\mu_t) \coloneqq \mathbb{E} \big[W^p_p (\mu_s,\mu_t)\big] = \tilde{\mathbb{E}} \big[W^p_p (\mu_s,\mu_t)\big] & \leq \tilde{\mathbb{E}} \bigg[ \tilde{\mathbb{E}} \big[|X_t - X_s|^p \big| \mathcal{F}_T \big] \bigg] \\ &= \tilde{\mathbb{E}} \big[ |X_t - X_s|^p \big] \leq c_{p,\mathrm{d}} \, \mathcal{E}_{p,T} (\mu,b,a) \, |t-s|^{\frac{p-1}{2}}.
        \end{align}
        Therefore, we have $(\mu_t) \in C^{\upgamma\textrm{-} \mathrm{H\ddot{o}l}}(I;P_{p,\Omega}(\mathbb{R}^\mathrm{d}))$ with $\upgamma \coloneqq \frac{1}{2} - \frac{1}{2p}.$
        In particular, we have $(\mu_t) \in W^{\alpha,p} (I;P_{p,\Omega}(\mathbb{R}^\mathrm{d}))$ for any $\alpha \in (0,\upgamma)$, by Remark \ref{rmk:Holder-FractionalSobolev_r_paper}.
        \\
        To apply Theorem \ref{thm:existence_of_Walphap_minimizer_random_intro}, it is enough to find a random lift with finite $W^{\alpha,p}$-energy for some $\alpha \in (\frac{1}{p},1) $. The $W^{\alpha,p}$-energy of our lift $\tilde{\pi}$ is 
        \begin{align}
            \mathbb{E} \left[  \int_{\Gamma_T} | \gamma |_{W^{\alpha,p}}^p  \d \tilde{\pi}(\gamma) \right] & =  {\mathbb{E}} \left[  \int_{\Gamma_T} \iint_{[0,T]^2} \frac{|\gamma_t- \gamma_s|^p }{|t-s|^{1+\alpha p}}  \d s \d t \d \tilde{\pi}(\gamma) \right] \\
            & = \iint_{[0,T]^2} \frac{{\mathbb{E}} \left[\int_{\Gamma_T} |\gamma_t- \gamma_s|^p  \d \tilde{\pi}  (\gamma) \right]}{|t-s|^{1+\alpha p}} \d s \d t \\
            & \leq c_{p,\mathrm{d}} \, \mathcal{E}_{p,T} (\mu,b,a)  \iint_{[0,T]^2} \frac{ 1}{|t-s|^{1+\alpha p - \upgamma p}} \d s \d t \\
            & = c_{\alpha, p,\mathrm{d}, T} \, \mathcal{E}_{p,T} (\mu,b,a),  
        \end{align}
        where the last double integral is finite whenever $\alpha \in (0,\upgamma)$, as mentioned in Remark \ref{rmk:Holder-FractionalSobolev_r_paper}.
        The two conditions on $\alpha$ imply that $ \alpha \in (\frac{1}{p},\frac{1}{2} - \frac{1}{2p})$, which in turn requires $p>3$. This completes the proof.
\end{proof}

\begin{figure}
        \centering
        \begin{tikzpicture}[node distance=80pt, auto]
        % Nodes
        \node (rect1) [draw, rectangle, text width=220 pt, align=center] 
        {\scriptsize $(\mu_t)$ solution of \eqref{eq:stochastic_Fokker-Planck} satisfying the
        
        $p$-integrability \eqref{eq:p_integrability_assumption_SFPE} and $\mu_0\in P_{p,\Omega} (\mathbb{R}^{\mathrm{d} })$ for some $p>3$};
        
        \node (rect2) [below of=rect1, draw, rectangle, text width=220 pt, align=center] 
        {\scriptsize existence of a particle representation 
        
        with finite $W^{\alpha,p}$-energy for any $\alpha \in (\frac1p , \frac12 -\frac{1}{2p})$};
        
        \node (rect3) [below of=rect2, draw, rectangle, text width=220 pt, align=center] 
        {\scriptsize existence of a particle representation 
        
        with minimum $W^{\alpha,p}$-energy for each $\alpha \in (\frac1p , \frac12 -\frac{1}{2p})$};
        
        % Arrows
        \draw[->] (rect1) -- (rect2)  node[midway] {\scriptsize \shortstack{stochastic superposition principle \\
        of Lacker--Shkolnikov--Zhang \cite{LackerShkolnikovZhang2023} \\
        + a standard energy estimation}};
        
        \draw[->] (rect2) -- (rect3) node[midway, right] {\scriptsize Theorem \ref{thm:existence_of_Walphap_minimizer_random_intro}};
    \end{tikzpicture}
    \captionsetup{font=footnotesize}
        \caption{Proof sketch of Corollary \ref{crl:minimizing_processes_sFPE_intro} for the existence of $W^{\alpha,p}$-energy-minimizing particle representations for solutions to stochastic Fokker--Planck--Kolmogorov equations. }
    \label{fig:proof_sketch}
    \end{figure}

%\noindent
%\textbf{Author Contributions Statement.}
%This is a single-authored article. The author is responsible for all aspects of the work.

%\bibliographystyle{plainnat}
%\bibliographystyle{acm}
%\bibliographystyle{siam}
%\bibliographystyle{abbrv}
\bibliographystyle{abbrvurl}
\bibliography{ms}

\end{document}